\documentclass[11pt]{amsart}  

\usepackage{amsmath}
\usepackage{amsthm}
\usepackage{amsfonts}
\usepackage{amssymb}
\usepackage{apalike}
\usepackage{enumitem}
\usepackage{eucal}
\usepackage[all]{xy}
\usepackage{amsxtra}
\usepackage{appendix} 
\usepackage{tensor} 
\usepackage{url}
\usepackage{upgreek}
\usepackage{etoolbox}
\apptocmd{\sloppy}{\hbadness 10000\relax}{}{}
\usepackage{comment}
\usepackage{mathtools}
\usepackage{hyperref}
\allowdisplaybreaks

\RequirePackage{color}
\definecolor{bwgreen}{rgb}{0.183,1,0.5}
\definecolor{bwmagenta}{rgb}{0.7,0.0,0.1}
\definecolor{bwblue}{rgb}{0.317,0.161,1}

\hypersetup{
pdfauthor = {\textcopyright\ Christopher Davis},
pdfcreator = {\LaTeX\ with package \flqq hyperref\frqq},
colorlinks = {true},
linkcolor={bwmagenta},	 
anchorcolor={},	 
citecolor={bwblue},	
filecolor={},	
menucolor={},	
runcolor={},	
urlcolor={bwgreen}, 
}

\DeclareFontFamily{OT1}{rsfs}{}
\DeclareFontShape{OT1}{rsfs}{n}{it}{<-> rsfs10}{}
\DeclareMathAlphabet{\mathscr}{OT1}{rsfs}{n}{it}

\DeclareFontFamily{OT1}{pzc}{}
\DeclareFontShape{OT1}{pzc}{n}{it}{<->s*[2.2]pzc}{}
\DeclareMathAlphabet{\mathpzc}{OT1}{pzc}{b}{sl}

\makeatletter

\newcommand{\Rmnum}[1]{\expandafter\@slowromancap\romannumeral #1@}
\makeatother

\DeclareMathOperator{\colim}{colim}

\DeclareMathOperator{\Frac}{Frac}

\DeclareMathOperator{\Hom}{Hom}

\DeclareMathOperator{\Ext}{Ext}

\DeclareMathOperator{\Spec}{Spec}

\DeclareMathOperator{\dlog}{\emph{d}\log}

\DeclareMathOperator{\im}{im}

\DeclareMathOperator{\Tor}{Tor}

\DeclareMathOperator{\THH}{THH}
\DeclareMathOperator{\trace}{trace}
\DeclareMathOperator{\TC}{TC}
\DeclareMathOperator{\TP}{TP}
\DeclareMathOperator{\TR}{TR}
\DeclareMathOperator{\cycl}{cycl}
\DeclareMathOperator{\cann}{can}
\DeclareMathOperator{\TF}{TF}
\DeclareMathOperator{\holim}{holim}
\DeclareMathOperator{\HH}{HH}

\newcommand*{\FF}{\ensuremath{\mathbf{F}}}   
   
\newcommand*{\CC}{\ensuremath{\mathbf{C}}}              
\newcommand*{\ZZ}{\ensuremath{\mathbf{Z}}}               
\newcommand*{\QQ}{\ensuremath{\mathbf{Q}}}

\newcommand{\Cp}{\CC_p}
\newcommand{\OCp}{\mathcal{O}_{\Cp}} 
         
\numberwithin{equation}{section}
\theoremstyle{plain}
  \newtheorem{theorem}[equation]{Theorem}
    
  \newtheorem{proposition}[equation]{Proposition}
  \newtheorem{lemma}[equation]{Lemma}
  \newtheorem{corollary}[equation]{Corollary}
\newtheorem{innercustomtheorem}{Theorem}
\newenvironment{customtheorem}[1]
  {\renewcommand\theinnercustomtheorem{#1}\innercustomtheorem}
  {\endinnercustomtheorem}

\theoremstyle{definition}
  \newtheorem{definition}[equation]{Definition}
  \newtheorem{notation}[equation]{Notation}

  \newtheorem{question}[equation]{Question}

\theoremstyle{remark}
  \newtheorem{example}[equation]{Example}
  \newtheorem{remark}[equation]{Remark}

\begin{document}



\title{On the de\thinspace Rham-Witt complex over perfectoid rings}

\author{Christopher Davis}
\address{University of California, Irvine}
\email{daviscj@uci.edu}

\author{Irakli Patchkoria}
\address{University of Aberdeen}
\email{irakli.patchkoria@abdn.ac.uk}

\maketitle

\begin{abstract}
Fix an odd prime $p$.  The results in this paper are modeled after work of Hesselholt and Hesselholt-Madsen on the $p$-typical absolute de\thinspace Rham-Witt complex in mixed characteristic.  We have two primary results.  The first is an exact sequence which describes the kernel of the restriction map on the de\thinspace Rham-Witt complex over $A$, where $A$ is the ring of integers in an algebraic extension of $\QQ_p$, or where $A$ is a $p$-torsion-free perfectoid ring. The second result is a description of the $p$-power torsion (and related objects) in the de\thinspace Rham-Witt complex over $A$, where $A$ is a $p$-torsion-free perfectoid ring containing a compatible system of $p$-power roots of unity.  Both of these results are analogous to results of Hesselholt and Madsen.  Our main contribution is the extension of their results to certain perfectoid rings. We also provide algebraic proofs of these results, whereas the proofs of Hesselholt and Madsen used techniques from topology.
\end{abstract}

\section{Introduction}

Let $p$ denote an odd prime and let $A$ denote a $\ZZ_{(p)}$-algebra.  The (absolute, $p$-typical) de\thinspace Rham-Witt complex over $A$ is defined by Hesselholt and Madsen in \cite[Introduction]{HM04} as the initial object in the category of Witt complexes over $A$. It is closely related to trace invariants such as topological Hochschild and cyclic homologies which can be used to compute algebraic K-theory. The main goal of the present paper is to prove certain algebraic properties of the de\thinspace Rham-Witt complex.  These results extend earlier work of Hesselholt and Madsen to a certain class of perfectoid rings.  (See Definition~\ref{perfectoid definition} for the definition of a perfectoid ring.) We also note that our proofs are purely algebraic, whereas the original proofs by Hesselholt and Madsen used topology.

Every element in the de\thinspace Rham-Witt complex over $A$ can be expressed using Witt vectors and the differential map $d: W_n \Omega^i_A \rightarrow W_n \Omega^{i+1}_A$.  Given two such expressions, however, it is difficult to determine if they are equal.  This phenomenon is already present in the module of K\"ahler differentials: every element in $\Omega^i_{A/R}$ can be expressed using $A$ and the differential map, but in general it is difficult to determine if two such expressions are equal.  The most challenging steps in our proofs involve showing that certain elements in $W_n\Omega^1_A$ are non-zero.  (This includes the case $n = 1$, i.e., a difficult step involves showing that certain elements in $\Omega^1_A = \Omega^1_{A/\ZZ}$ are non-zero.) 

Our main result is Theorem~\ref{Tate theorem} below.  Theorem~\ref{Tate theorem} concerns $p$-torsion in the level~$n$ and degree~$1$ component of the de\thinspace Rham-Witt complex, $W_n\Omega^1_A$, for certain perfectoid rings $A$ and for all positive integers $n \geq 1$.  Our proof uses induction on $n$.  The base case uses Theorem~\ref{intro relative Omega} below and the induction step uses Theorem~\ref{sequence theorem}.  Together, these three theorems are the main results of this paper, and we believe that Theorem~\ref{intro relative Omega} and Theorem~\ref{sequence theorem} are of interest independent of their use in the proof of Theorem~\ref{Tate theorem}.  We mention that both Theorem~\ref{intro relative Omega} and Theorem~\ref{sequence theorem} hold for all $p$-torsion-free perfectoid rings (unlike Theorem~\ref{Tate theorem}).  Also, we note that Theorem~\ref{intro relative Omega} and its cotangent complex counterpart, Theorem~\ref{intro cotangent}, do not involve the de\thinspace Rham-Witt complex.

We now briefly describe these three main results, going through them in the same order in which they appear later in the paper.  Our first result concerns $p$-torsion in the module of K\"ahler differentials.  

\begin{customtheorem}{A} \label{intro relative Omega}
Let $A$ denote a $p$-torsion-free perfectoid ring. 
The $p$-adic Tate module of $\Omega^1_{A/\ZZ_p}$, denoted $T_p(\Omega^1_{A/\ZZ_p})$, is a free $A$-module of rank~one.
\end{customtheorem}

See Theorem~\ref{main relative Omega theorem} below for the proof of Theorem~\ref{intro relative Omega}.  A similar result appears in \cite[Section~1]{Fon81} for the case that $A$ is the ring of integers in a separable closure of a local field.

We give an example of the sort of considerations that arise when trying to prove Theorem~\ref{intro relative Omega}.  

\begin{example} \label{intro example Kahler}
Let $R = \ZZ_p[\zeta_p]$; this ring is not perfectoid because the $p$-power map is not surjective modulo~$p$.   Because $d1 = 0 \in \Omega^1_{R/\ZZ_p}$, it follows directly from the Leibniz rule that $d\zeta_p$ is $p$-torsion.  Proving that $d\zeta_p \neq 0 \in \Omega^1_{R/\ZZ_p}$ is more subtle.  One approach is to use the isomorphism $\ZZ_p[\zeta_p] \cong \ZZ_p[x]/(x^{p-1} + x^{p-2} + \cdots + x + 1)$ together with an exact sequence involving K\"ahler differentials of a quotient ring (see \cite[Theorem~25.2]{Mat89} for the precise result, or see Remark~\ref{intro remark Kahler} below for a brief summary).  Similar considerations can be used for $\Omega^1_{S/\ZZ_p}$ when $S = \ZZ_p[\zeta_{p^n}]$ for any $n \geq 1$ or for their union $S = \ZZ_p[\zeta_{p^{\infty}}]$.  It seems significantly more difficult to extend this style of argument to the case of the $p$-adic completion $\ZZ_p[\zeta_{p^{\infty}}]^{\wedge}$ (which is a perfectoid ring, unlike the other rings mentioned in this example).  Different techniques seem necessary to deal with rings like $\ZZ_p[\zeta_{p^{\infty}}]^{\wedge}$, and our approach is to use the cotangent complex.  (For clarification, we briefly point out that $d\zeta_p$ is not an $S$-module generator of $p$-torsion in $\Omega^1_{S/\ZZ_p}$ for any of $S = \ZZ_p[\zeta_{p^n}]$, $n \geq 2$, or $S = \ZZ_p[\zeta_{p^{\infty}}]$ or $S = \ZZ_p[\zeta_{p^{\infty}}]^{\wedge}$.  Instead, exactly as in \cite{Hes06}, there exists a $p$-torsion element $\alpha$ such that $(\zeta_p - 1)\alpha = d\zeta_p$, and such an $\alpha$ is a generator of $p$-torsion.)
\end{example}

\begin{remark} \label{intro remark Kahler}
In Example~\ref{intro example Kahler} we referenced \cite[Theorem~25.2]{Mat89}; we briefly recall the context of that theorem.  Let $R_0 \rightarrow R$ denote a ring homomorphism and let $I \subseteq R$ denote an ideal.  (What we describe here works in complete generality for commutative rings with unity; the notation need not refer to any specific rings from Example~\ref{intro example Kahler}.) Then \cite[Theorem~25.2]{Mat89} describes an exact sequence of $R/I$-modules
\[
I/I^2 \rightarrow \Omega^1_{R/R_0} \otimes_R (R/I) \rightarrow \Omega^1_{(R/I)/R_0} \rightarrow 0.
\]
We point out two aspects of this sequence.  First, notice that the left-most map is not guaranteed to be injective.  This is again a reflection of the fact that it is in general difficult to prove elements are non-zero in K\"ahler differentials (and in the de\thinspace Rham-Witt complex).  Second, notice that if $I$ is a principal ideal generated by a non-zero-divisor, then $I/I^2$ is isomorphic to $R/I$.  The authors consider this one of the benefits of the requirement in the definition of perfectoid that $\ker \theta$ be a principal ideal; see for example the proof of \cite[Proposition~4.19(2)]{BMS18}, which is an essential result in our arguments involving K\"ahler differentials. 
\end{remark}

The strategy described in Example~\ref{intro example Kahler} relies heavily on having an explicit description of $\ZZ_p[\zeta_{p}]$.  For rings of the generality considered in this paper, we cannot expect such an explicit description.  To prove Theorem~\ref{intro relative Omega} for an arbitrary $p$-torsion-free perfectoid ring, we rely on properties of the cotangent complex, and the authors thank Bhargav Bhatt for suggesting this approach to Theorem~\ref{intro relative Omega}.  In particular, our proof of Theorem~\ref{intro relative Omega} is intertwined with the proof of the following.
Again, see Theorem~\ref{main relative Omega theorem} below for the proof. 

\begin{customtheorem}{A.2} \label{intro cotangent}
Let $A$ denote a $p$-torsion-free perfectoid ring.  The multiplication-by-$p$ map is surjective on $H^{-1}(L_{A/\ZZ_p})$, where $L_{A/\ZZ_p}$ denotes the cotangent complex of $\ZZ_p \rightarrow A$.
\end{customtheorem}
 
Many of the algebraic arguments in this paper are elementary. With the exception of the final section relating our results to algebraic $K$-theory, our arguments involving the cotangent complex are the most technically advanced part of this paper.  On the other hand, the cotangent complex arguments are used only to prove results concerning K\"ahler differentials, so the reader who is willing to accept these results can understand the paper without the cotangent complex.  The authors had originally hoped to bypass the cotangent complex at the expense of considering only $p$-torsion-free perfectoid rings which are rings of integers in a valued field.  Even in this case, technical difficulties arose.  To continue the scenario described in Example~\ref{intro example Kahler}, again following Fontaine, we were able to describe $p$-torsion in $\Omega^1_{\ZZ_p[\zeta_{p^{\infty}}]/\ZZ_p}$, but without using the cotangent complex, we were unable to describe $p$-torsion in $\Omega^1_{\ZZ_p[\zeta_{p^{\infty}}]^{\wedge}/\ZZ_p}$ (where $\ZZ_p[\zeta_{p^{\infty}}]^{\wedge}$ denotes the $p$-adic completion of $\ZZ_p[\zeta_{p^{\infty}}]$; perfectoid rings are by definition required to be $p$-adically complete).

We have already mentioned the difficulty of proving that elements in the module of K\"ahler differentials are non-zero.  We next discuss a result, Theorem~\ref{sequence theorem}, which helps to prove certain elements in the de\thinspace Rham-Witt complex are non-zero.  This result is based on a corresponding result proved by Hesselholt and Madsen; we give a reference to their results after stating the theorem.

The idea of Theorem~\ref{sequence theorem} is as follows.  We take as given that we have a satisfactory understanding of the K\"ahler differentials (e.g., we have an explicit description of all $p$-power torsion elements in the module of K\"ahler differentials), and inductively, we assume we also have a satisfactory understanding of the level~$n$ and degree~$1$ component of the de\thinspace Rham-Witt complex.  
When $A$ is a $p$-torsion-free perfectoid ring (or when $A = \mathcal{O}_K$ with $K/\QQ_p$ algebraic), Theorem~\ref{sequence theorem} below precisely describes the kernel of the surjective restriction map, $R: W_{n+1}\Omega^1_A \rightarrow W_n\Omega^1_A$, and thus enables us to describe the level~$n+1$ and degree~1 component of the de\thinspace Rham-Witt complex in terms of three other components: $A$, $\Omega^1_A$ and $W_n\Omega^1_A$.  

In Theorem~\ref{sequence theorem}, the kernel of restriction is described as a $W_{n+1}(A)$-module.  (The most complicated part of stating this theorem is describing the $W_{n+1}(A)$-module structure of the terms in the sequence.)  The short exact sequence~(\ref{sequence}) below is used in a variety of inductive arguments involving the de\thinspace Rham-Witt complex.  We give some basic examples of applying Theorem~\ref{sequence theorem} later in this introduction; for us, the most important application is in the proof of Theorem~\ref{Tate theorem}.

\begin{customtheorem}{B} \label{sequence theorem}
Let $p$ denote an odd prime.  Assume $A$ is a ring satisfying one of the following two conditions:
\begin{enumerate}
\item \label{perf A} We have that $A$ is a $p$-torsion-free perfectoid ring, in the sense of \cite[Definition~3.5]{BMS16}; or
\item \label{alg A} We have that $A = \mathcal{O}_K$, where $K/\QQ_p$ is an algebraic extension.
\end{enumerate}
Then for every integer $n \geq 1$, the following is an exact sequence of $W_{n+1}(A)$-modules:
\begin{equation} \label{sequence}
0 \rightarrow A \xrightarrow{(-d, p^n)}  \Omega^1_{A} \oplus A \xrightarrow{V^n \oplus dV^n} W_{n+1} \Omega^1_{A} \xrightarrow{R} W_n \Omega^1_{A} \rightarrow 0,
\end{equation}
where the $W_{n+1}(A)$-module structure is defined as follows.  The $W_{n+1}(A)$-module structure on the left-most $A$ in the sequence is induced by $F^n$, where $F$ denotes the Witt vector Frobenius.  The $W_{n+1}(A)$-module structure on $\Omega^1_A \oplus A$ is defined by 
\[
y \cdot (\alpha, a) = \big(F^n(y) \alpha - a F^n\left( dy \right), F^n(y) a\big) ,\text{ where } y \in W_{n+1}(A),\; \alpha \in \Omega^1_A,\; a \in A.
\] 
The $W_{n+1}(A)$-module structure on $W_{n+1}\Omega^1_A$ is the natural one, and the $W_{n+1}(A)$-module structure on $W_n \Omega^1_A$ is induced by restriction.  
\end{customtheorem}

See Section~\ref{proof of sequence theorem} below for the proof of Theorem~\ref{sequence theorem}.  Most of the claimed exactness was proved in complete generality (i.e., with no restrictions on the $\ZZ_{(p)}$-algebra~$A$) by Hesselholt and Madsen; see \cite[Proposition~3.2.6]{HM03}, and note that the de\thinspace Rham-Witt complex is a special case of the logarithmic de\thinspace Rham-Witt complex, attained by taking the multiplicative monoid $M$ (as in Hesselholt-Madsen's notation) to be trivial, $M = \{1\}$.  Applying Hesselholt and Madsen's results, our only remaining task is proving that, if $V^n(\omega) + dV^n(a) = 0 \in W_{n+1}\Omega^1_A$, then there exists $a_0 \in A$ such that $\omega = -da_0$ and $a = p^n a_0$.  For proofs by Hesselholt and Madsen related to this remaining task, see \cite[Proof of Theorem~3.3.8]{HM03} and  \cite[Proposition~2.2.1]{Hes06}.  Most of Hesselholt and Madsen's proof of exactness of (\ref{sequence}) does not involve any notions from topology, and we simply quote their result.  But the argument following from \cite[Proof of Theorem~3.3.8]{HM03} does use topology.  Our two main contributions with respect to Theorem~\ref{sequence theorem} are proving exactness in the case of $p$-torsion-free perfectoid rings, and providing an algebraic proof of the most difficult step, which we re-iterate for emphasis: if $V^n(\omega) + dV^n(a) = 0 \in W_{n+1}\Omega^1_A$, then there exists $a_0 \in A$ such that $\omega = -da_0$ and $a = p^n a_0$.

The benefit of having an exact sequence as guaranteed in Theorem~\ref{sequence theorem} is clear from the analogous situation involving Witt vectors.  Recall that, for every ring $A$ and every integer $n \geq 1$, we have a short exact sequence of $W_{n+1}(A)$-modules 
\begin{equation} \label{Witt sequence}
0 \rightarrow A \stackrel{V^n}{\longrightarrow} W_{n+1}(A) \stackrel{R}{\rightarrow} W_n(A) \rightarrow 0,
\end{equation}
where the $W_{n+1}(A)$-module structures are defined as follows.  Let $F$ denote the Witt vector Frobenius and let $R$ denote the restriction map.  The module structure on $A$ is induced by $F^n: W_{n+1}(A) \rightarrow W_1(A) \cong A$.  The module structure on $W_{n+1}(A)$ is the natural one, and the module structure on $W_n(A)$ is induced by restriction $R: W_{n+1}(A) \rightarrow W_n(A)$.  The sequence~(\ref{Witt sequence}) is always an exact sequence of $W_{n+1}(A)$-modules, for every ring~$A$, and this fact is very useful when making induction arguments.  For example, using induction on $n$ and exactness of the sequence in Equation~(\ref{Witt sequence}), one proves that if $A$ is a $\ZZ_{(p)}$-algebra, then for every positive integer $n$, the ring $W_n(A)$ is also a $\ZZ_{(p)}$-algebra; see for example \cite[Lemma~1.9]{Hes15}.

The sequence~(\ref{sequence}) from Theorem~\ref{sequence theorem} is analogous to the sequence~(\ref{Witt sequence}), but for general rings $A$, the sequence~(\ref{sequence}) may not be exact.  For example, the sequence~(\ref{sequence}) is never exact if the ring $A$ is an $\FF_p$-algebra.  Even when the sequence~(\ref{sequence}) is exact, it is typically difficult to prove exactness.   (In contrast, the proof of exactness of (\ref{Witt sequence}) is trivial.) The ring $A = \ZZ_p$ is the easiest case of Theorem~\ref{sequence theorem}, but even in this case we are not aware of a simple proof of exactness.

The $W_{n+1}(A)$-module structures described in Theorem~\ref{sequence theorem} can be imposing at first, but they are very natural.  For example, consider the degree~zero case considered in the sequence~(\ref{Witt sequence}).  Why do we use $F^n: W_{n+1}(A) \rightarrow A$ to equip $A$ with a $W_{n+1}(A)$-module structure, instead of using, for example, the restriction map, $R^n: W_{n+1}(A) \rightarrow A$?  The reason of course is the Witt vector formula $x V^n(y) = V^n(F^n(x)y)$; this formula is saying precisely that $A \stackrel{V^n}{\rightarrow} W_{n+1}(A)$ is a $W_{n+1}(A)$-module homomorphism for the proposed module structure using $F^n$.  The case is the same for our module structures described in Theorem~\ref{sequence theorem}.  Let $x \in W_{n+1}(A)$, $y \in \Omega^1_A$, and $z \in A$.  Using the same formula $x V^n(y) = V^n(F^n(x)y)$ and the Leibniz rule, we compute 
\begin{align*}
x \cdot (V^n(y) + dV^n(z) ) &= V^n (F^n(x) y) + x dV^n(z) \\
&=  V^n (F^n(x) y) + d (xV^n(z) ) - V^n(z) d x\\
&=  V^n (F^n(x) y) + d V^n(F^n(x)z) - V^n(z F^n(d x))\\
&= V^n \bigg(F^n(x) y - z F^n(dx) \bigg) + dV^n\bigg(F^n(x) z\bigg).
\end{align*}
As above, this formula is saying precisely that $\Omega^1_A \oplus A \xrightarrow{V^n + dV^n} W_{n+1}\Omega^1_A$ is a $W_{n+1}(A)$-module homomorphism for the proposed module structure.

We will use Theorem~\ref{sequence theorem} during our proof by induction of Theorem~\ref{Tate theorem}.  That is our most important application and our original motivation for considering Theorem~\ref{sequence theorem}.  Here we first record several other applications.

\begin{proposition} \label{dV1}
Let $A_0$ denote any $\ZZ_{(p)}$-algebra for which there exists a $p$-torsion-free perfectoid ring~$A$ and a ring homomorphism $A_0 \rightarrow A$.  (For example, any subring of $\OCp$ containing $\ZZ_{(p)}$ is a suitable choice of $A_0$.)
Then $dV^n(1) \in W_{n+1}\Omega^1_{A_0}$ is non-zero for every integer $n \geq 1$.  
\end{proposition}

\begin{proof}
The sequence (\ref{sequence}) is exact for the $p$-torsion-free perfectoid ring~$A$ by Theorem~\ref{sequence theorem}.  This shows that $dV^n(1) \neq 0 \in W_{n+1}\Omega^1_{A}$.  By considering the map $W_{n+1}\Omega^1_{A_0} \rightarrow W_{n+1}\Omega^1_{A}$ induced by functoriality, we see that we must also have $dV^n(1) \neq 0 \in W_{n+1}\Omega^1_{A_0}$.
\end{proof}

The fact that Proposition~\ref{dV1} is not obvious, even in the case $A = \ZZ_{(p)}$, underscores the difficulty of proving that elements in the de\thinspace Rham-Witt complex are non-zero.

We briefly point out in the following question that we are unsure how restrictive is the hypothesis from Proposition~\ref{dV1}.

\begin{question}
Does the hypothesis of Proposition~\ref{dV1} hold for every $p$-torsion-free, $p$-adically separated ring?
\end{question}

The following provides another basic application of Theorem~\ref{sequence theorem}.  Proposition~\ref{dV1} is stated mostly as a curiosity.  On the other hand, the following, Proposition~\ref{V is injective}, is more important and will be used at several different points later in this paper.

\begin{proposition} \label{V is injective}
Let $A$ denote any $p$-torsion-free $\ZZ_{(p)}$-algebra for which the sequence (\ref{sequence}) is exact for all integers $n \geq 1$.  Then for all integers $m,n \geq 1$, the map $V^m: W_n \Omega^1_A \rightarrow W_{n+m} \Omega^1_A$ is injective.
\end{proposition}

\begin{proof}
Because $V^m = V \circ \cdots \circ V$ ($m$ total iterations), it suffices to prove the map $V$ is injective.  We prove that $V$ is injective using induction on $n$.  When $n = 1$, $W_n \Omega^1_A = \Omega^1_A$, and the result follows from the exactness of the sequence~(\ref{sequence}), using the fact that $A$ is $p$-torsion free.  

Now assume we know that $V: W_n \Omega^1_A \rightarrow W_{n+1} \Omega^1_A$ is injective.   Consider an element $x \in W_{n+1} \Omega^1_A$ such that $V(x) = 0 \in W_{n+2} \Omega^1_A$.  Let $R$ denote the restriction map $R: W_{m+1}\Omega^1_A \rightarrow W_m\Omega^1_A$ (for some $m$). Because $R \circ V = V \circ R$ and $V: W_n \Omega^1_A \rightarrow W_{n+1} \Omega^1_A$ is injective by our induction hypothesis, we deduce that $R(x) = 0 \in W_n\Omega^1_A$.  Hence by exactness of the sequence~(\ref{sequence}), there exist $\alpha \in \Omega^1_A$ and  $a \in A$ such that 
\[
 x = V^n(\alpha) + dV^n(a)\in W_{n+1} \Omega^1_A.
\]
Applying $V$ to this element, we get
\[
0 = V(x) = V^{n+1}(\alpha) + VdV^n(a) = V^{n+1}(\alpha) + dV^{n+1}(pa)  \in W_{n+2} \Omega^1_A.  
\]
Using (the most difficult part of) exactness of the sequence~(\ref{sequence}), we have that $pa = p^{n+1} a_0$ and $\alpha = -da_0$ for some $a_0 \in A$.  Because $A$ is $p$-torsion free, we have that $a = p^n a_0$.  Thus 
\[
x = V^n(-da_0) + dV^n(p^n a_0) = 0 \in W_{n+1} \Omega^1_A,
\] 
which completes the proof.  
\end{proof}

\begin{remark}
The analogue of Proposition~\ref{V is injective} is not true in the classical case that $A$ is an $\FF_p$-algebra.  For example, if $A = \FF_p[x]$ and we consider $dx \in \Omega^1_{A} \cong W_1 \Omega^1_A$, then $dx \neq 0 \in W_1 \Omega^1_A$, but $V(dx) = dV(px) = dV(0) = 0 \in W_2 \Omega^1_A$.  (Nor does the analogue of Proposition~\ref{dV1} hold for the case that $A$ is an $\FF_p$-algebra.  Indeed, $V(1) = p = 0 \in W(A)$ when $A$ is an $\FF_p$-algebra, so $dV^n(1) = 0$ for all integers $n \geq 1$.) 
\end{remark}

We next describe the main result of this paper, Theorem~\ref{Tate theorem}.  It is completely modeled after the work of Hesselholt in \cite{Hes06}, where the case of the ring of integers in an algebraic closure of a local field is considered.  

\begin{customtheorem}{C} \label{Tate theorem}
Let $A$ denote a $p$-torsion-free perfectoid ring containing a compatible system of $p$-power roots of unity (see Notation~\ref{has roots of unity}).  Then the following hold:
\begin{enumerate}
\item \label{intro p^r torsion} For all integers $n \geq 1$ and $r \geq 1$, the $p^r$-torsion $W_n\Omega^1_A[p^r]$ is a free $W_n(A)/p^r W_n(A)$-module of rank~one.
\item \label{intro Tate dRW} For every integer $n \geq 1$, the $p$-adic Tate module $T_p \left( W_n \Omega^1_A \right)$ is a free $W_n(A)$-module of rank~one.
\item \label{intro Ainf} The inverse limit $\displaystyle \varprojlim_F T_p \left(W_n \Omega^1_A \right)$ is a free  $\displaystyle \varprojlim_F \left(W_n(A)\right) \cong W\left( \varprojlim_{x \mapsto x^p} A/pA \right)$-module of rank~one.
\end{enumerate}
\end{customtheorem}

The proofs of these claims are located as follows.  Claim~(\ref{intro Tate dRW}) follows from Theorem~\ref{generator of Tate}.  Once we know Claim~(\ref{intro Tate dRW}), then Claim~(\ref{intro p^r torsion}) is a consequence of Lemma~\ref{p-torsion vs Tate module}. Lastly, Claim~(\ref{intro Ainf}) follows from Corollary~\ref{inverse F}.  See \cite[Theorem~B, Proposition~2.3.2, and Proposition~2.4.2]{Hes06} for closely related results.

\begin{remark}
The following are possible choices of $A$ satisfying the conditions of Theorem~\ref{Tate theorem}: $\OCp$,  $\ZZ_p[\zeta_{p^{\infty}}]^{\wedge}$, $\ZZ_p[\zeta_{p^{\infty}}$, $p^{1/p^{\infty}}]^{\wedge}$,  $\OCp\langle T^{1/p^{\infty}} \rangle$.   An example of a $p$-torsion-free perfectoid ring which does not contain the $p$-power roots of unity is $\ZZ_p[p^{1/p^{\infty}}]^{\wedge}$.  Thus our Theorems~\ref{intro relative Omega} and \ref{sequence theorem} hold for $\ZZ_p[p^{1/p^{\infty}}]^{\wedge}$, but we are unsure if Theorem~\ref{Tate theorem} holds for this ring. 
\end{remark}

\begin{remark}
Our proofs are closely modeled on the work of Hesselholt \cite{Hes06} which concerned  the case of the ring of integers in an algebraic closure of a  local field,  such as $A = \mathcal{O}_{\overline{\QQ_p}}$. Notice that $\mathcal{O}_{\overline{\QQ_p}}$ is not covered by our Theorem~\ref{Tate theorem}, because $\mathcal{O}_{\overline{\QQ_p}}$ is not $p$-adically complete and hence not perfectoid.  If $K/\QQ_p$ is algebraic and $(\mathcal{O}_K)^{\wedge}$ is a perfectoid ring, then we expect analogues of most of our results (and the referenced results of Hesselholt) to hold for $A = \mathcal{O}_K$, and it should be straightforward to adjust our techniques to this situation.  For example, it should follow that the $p$-adic Tate module $T_p \left( W_n \Omega^1_{\mathcal{O}_K} \right)$ is a free $W_n(\mathcal{O}_K)^{\wedge}$-module of rank~one.   We do not consider non-$p$-adically complete rings in this paper for two reasons.  The first reason is that, by restricting our attention to the $p$-adically complete case, we do not have to reprove results which are essentially automatic in the perfectoid case (such as the kernel of $F: W_{n+1}(A) \rightarrow W_n(A)$ being a principal ideal).  The second reason is that we are unsure what the correct generality to consider is in the non-$p$-adically complete setting.   We thank Kiran Kedlaya for calling this question to our attention. 
\end{remark}

One of the main motivations for studying the $p$-adic Tate modules $T_p \left( W_n \Omega^1_A \right)$ and $\displaystyle \varprojlim_F T_p \left(W_n \Omega^1_A \right)$ comes from topology. The $p$-adic algebraic K-theory $K(A, \ZZ_p)$ maps via the trace maps to $\TR^n(A, \ZZ_p)$ and $\TF(A, \ZZ_p)$ which are spectra obtained from the cyclotomic structure on the topological Hochschild homology $\THH(A)$. Using results of Hesselholt (\cite{Hes04} and \cite{Hes06}) and their generalizations \cite[Remark 6.6]{BMS18}, one can conclude that the homotopy groups $\pi_2 \TR^n(A, \ZZ_p)$ and $\pi_2\TF(A, \ZZ_p)$ are isomorphic to the $p$-adic Tate modules  $T_p \left( W_n \Omega^1_A \right)$ and $\displaystyle \varprojlim_F T_p \left(W_n \Omega^1_A \right)$, respectively. Moreover, one can give an explicit description of the trace map $K_2(A, \ZZ_p) \to \pi_2 \TF(A, \ZZ_p)$ using the generator from the proof of Theorem \ref{Tate theorem}. For more details see the final Section \ref{tc section}, where we summarize some results on algebraic K-theory and topological cyclic homology and their connections with the main results of this paper. Most of the material in Section \ref{tc section} is well-known to the experts but we still find it important to give an account since it puts the algebraic objects of this paper in a broader topological context. The readers only interested in algebraic aspects of the de\thinspace Rham-Witt can safely skip the last section. 
 
There are two key hypotheses on the ring $A$ in Theorem~\ref{Tate theorem}.  The first is that $A$ contain a compatible sequence of $p$-power roots of unity.  The second  is that $A$ be a $p$-torsion-free perfectoid ring.   It will be repeatedly evident why the $p$-power root of unity hypothesis is important for our proof of Theorem~\ref{Tate theorem}: our constructions make constant use of the elements $\zeta_{p^n}$ for varying $n$.

The other key hypothesis on the ring $A$ in Theorem~\ref{Tate theorem} is that it be perfectoid. We offer three explanations for why this assumption is convenient.  
\begin{enumerate} 
\item A first reason is mentioned below in Remark~\ref{discrete Tate}.  That example suggests it is essential that the ring~$A$ be infinitely ramified over $\ZZ_p$, at least in the case of subrings of $\OCp$.
\item A second reason is implicit in Proposition~\ref{mult by p is surjective} below.  The condition that $A$ be perfectoid is closely related to the condition that, for all integers $n \geq 1$, the map $F: W_{n+1}(A) \rightarrow W_n(A)$ is surjective.  For example, the condition that $F: W_2(A) \rightarrow W_1(A)$ be surjective is the same as the condition that the $p$-power Frobenius map is surjective on $A/pA$.
\item A third reason was already mentioned above in Remark~\ref{intro remark Kahler}, where we commented that, in the usual notation regarding perfectoid rings, it is useful for us that $\ker \theta$ be a principal ideal.  
\end{enumerate}

\begin{remark} \label{discrete Tate}
Let $K/\QQ_p$ be an algebraic extension such that the ramification index of $K/\QQ_p$ is \emph{finite}.  Let $\mathcal{O}_K$ be its ring of integers.  By \cite[\S 2]{Fon81}, there exists an integer $N \geq 1$ such that $p^N \Omega^1_{\mathcal{O}_K/\ZZ_p} = 0$, and hence the $p$-adic Tate module $T_p(W_n\Omega^1_{\mathcal{O}_K})$ is trivial.  Thus, the analogue of the level $n = 1$ case of Theorem~\ref{Tate theorem} is false for $A = \mathcal{O}_K$.  (Using induction and an exact sequence analogous to the sequence in Proposition~\ref{zeta cor}, it should also follow that $T_p(W_n\Omega^1_{\mathcal{O}_K}) \cong 0$ for all integers $n \geq 1$, and thus the analogue of Theorem~\ref{Tate theorem} is false for all $n \geq 1$.)
\end{remark}

We say a ring $A$ is \emph{Witt-perfect} (at $p$) if for every integer $n \geq 1$, the Witt vector Frobenius $F: W_{n+1}(A) \rightarrow W_n(A)$ is surjective.  See \cite[Section~5]{DK14}. It turns out this is a much less restrictive condition than requiring that $F: W(A) \rightarrow W(A)$ be surjective.  All perfectoid rings are Witt-perfect, but for example, even $\OCp$ does not satisfy the condition of $F: W(\OCp) \rightarrow W(\OCp)$ being surjective; again see \cite{DK14}.

\begin{proposition} \label{mult by p is surjective}
Let $A$ denote a Witt-perfect ring (as defined in the preceding paragraph).  Then for all integers $n \geq 1$ and $d \geq 1$, the multiplication-by-$p$ map
\[
p: W_n\Omega^d_A \rightarrow W_n\Omega^d_A
\]
is surjective.
\end{proposition}

\begin{proof}
Note that any element in $W_n\Omega^d_A$ can be written as a sum of terms $x dy_1 \cdots dy_d$, where $x, y_i \in W_n(A)$.  Because $A$ is Witt-perfect, we can find $y_1' \in W_{n+1}(A)$ such that $F(y_1') = y_1$.  But then we have
\[
x dy_1 dy_2\cdots dy_d = x dF(y_1') dy_2 \cdots dy_d = px F(dy_1') dy_2 \cdots dy_d.
\]
This shows that every element in $W_n\Omega^d_A$ is a multiple of $p$.
\end{proof}

The most difficult part of the proof of Theorem~\ref{Tate theorem} is constructing suitably compatible $W_n(A)/pW_n(A)$-module generators for the $p$-torsion $W_n\Omega^1_A[p]$.  We construct them using induction on $n$, and our induction makes repeated use of Theorem~\ref{sequence theorem}.  
Once one has free $W_n(A)/pW_n(A)$-module generators for the $p$-torsion in $W_n\Omega^1_A$, it is relatively easy to produce free $W_n(A)/p^rW_n(A)$-module generators for the $p^r$-torsion, and by choosing everything compatibly, we are able to produce free $W_n(A)$-module generators for the $p$-adic Tate module of $W_n\Omega^1_A$ and free $\varprojlim_F \left(W_n(A)\right)$-module generators for $\varprojlim_F T_p(W_n\Omega^1_A)$.

\begin{remark}
Throughout this paper we consider only $p$-torsion-free perfectoid rings.  In the case of characteristic~$p$ perfectoid rings, the de\thinspace Rham-Witt complex is not interesting.  Namely, let $A$ denote a perfectoid ring of characteristic~$p$; this condition is equivalent to $A$ being a perfect ring of characteristic~$p$ (see, for example, \cite[Example~3.15]{BMS16}). Let $n \geq 1$ denote an integer.  Then $W_n\Omega^d_A$ a $W_n(A)$-module, and hence is $p^n$-torsion, on one hand, but on the other hand, multiplication by $p$ is surjective on $W_n\Omega^d_A$ for all degrees $d \geq 1$.  Thus $W_n\Omega^d_A \cong 0$ for all $d \geq 1$.   
We have not considered the case of perfectoid rings which are not characteristic~$p$ but which do have $p$-torsion.
\end{remark}

\begin{question}
Our construction of Frobenius-compatible generators of the $p$-torsion in $W_n\Omega^1_A$ is rather indirect for the levels $n \geq 2$.  The proof of Theorem~\ref{Tate theorem} would be simplified considerably if we could find a more direct construction.  We now describe one possible approach.  For every integer $n \geq 1$, one can check that there is a unique Witt vector $\frac{p^n}{[\zeta_{p^n}] - 1}\in W_n(A)$ such that $\frac{p^n}{[\zeta_{p^n}] - 1} \cdot ([\zeta_{p^n}] - 1) = p^n \in W_n(A)$.  The elements $\frac{p^n}{[\zeta_{p^n}] - 1} \dlog [\zeta_{p^{n+1}}]$ seemingly have all the compatibility properties needed to make our proofs work, although we are only able to prove they are $p$-torsion in the case $n = 1$.  Are these elements $p$-torsion?  If not, can they be modified  to produce generators for the $p$-torsion which are more explicit than what we work with in our proof of Theorem~\ref{Tate theorem}?
\end{question}

\begin{remark}
Our proofs require $p \neq 2$.  Attaining similar results for $p = 2$ would necessitate substantial changes.  For example, even the definition of the de\thinspace Rham-Witt complex we use is incorrect for $p = 2$ (see \cite{Cos08} for the correct definition).  The requirement $p \neq 2$ is also used in the first author's paper \cite{D17}, so the results from \cite{D17} which are used in the present paper would also have to be adjusted to allow for $p = 2$.
\end{remark}

\begin{notation}
Rings in this paper are commutative and have unity, and ring homomorphisms must send unity to unity.  Throughout this paper, $p \geq 3$ is a fixed \emph{odd} prime.  When we refer to Witt vectors, we mean $p$-typical Witt vectors with respect to this prime, and when we refer to a ring being \emph{perfectoid}, it is in the sense of \cite[Definition~3.5]{BMS16}, and it is with respect to this same prime~$p$.    When we say that $\zeta_p$ is a primitive $p$-th root of unity, we mean that $1 + \zeta_p + \cdots + \zeta_p^{p-1} = 0$.   

Here is some more specialized notation.  For a perfectoid ring  $A$ and an integer $n \geq 2$, we let $z_n \in W_n(A)$ denote a generator of $\ker F^{n-1}: W_n(A) \rightarrow W_1(A);$ in the case that we have identified a compatible system of roots of unity of $A$, then we will usually choose the generator $1 + [\zeta_{p^n}] + \cdots + [\zeta_{p^n}]^{p-1}$.  If $u$ is a unit in $A$ (or in $W_n(A)$), we write $\dlog u$ for $\frac{du}{u}$ in $\Omega^1_A$ (or in $W_n\Omega^1_A$).  
\end{notation}

\subsection*{Acknowledgments} The first author is very grateful to Lars Hesselholt, who introduced and explained many aspects of this project to him.  (The project began around 2014 when the first author was a postdoc of Lars Hesselholt at the University of Copenhagen.)  The first author would also like to especially thank Bhargav Bhatt for assistance at many different points, especially during a visit to the University of Michigan.  Furthermore, both authors thank Johannes Ansch\"utz, Bryden Cais, Dustin Clausen, Elden Elmanto, Kiran Kedlaya, Arthur-C\'esar Le Bras, Thomas Nikolaus, Peter Scholze, and David Zureick-Brown for useful conversations regarding this paper. Both authors thank the Department of Mathematical Sciences of the University of Copenhagen for its hospitality and pleasant working environment.

\section{Algebraic preliminaries}

In this section we gather miscellaneous algebraic results that will be used later.   This section is organized so that all the results concerning Witt vectors come at the end. The de\thinspace Rham-Witt complex does not appear in this section.  The reader is advised to skip this entire section and return to it as needed.

We begin with a few standard properties related to the cotangent complex.   We use the cotangent complex extensively in Section~\ref{Kahler section}, with regards to analyzing $p$-power torsion in modules of K\"ahler differentials.  (The cotangent complex is not used in the rest of the paper, so the reader who is willing to accept these results concerning K\"ahler differentials can skip the arguments involving the cotangent complex.)  There are two results concerning the cotangent complex that we will use repeatedly, the Jacobi-Zariski sequence and the Universal coefficient theorem.

\begin{proposition}[{The Jacobi-Zariski sequence \cite[Tag~08QX]{stacks-project}}]
Let $A \rightarrow B \rightarrow C$ be ring homomorphisms.  Then there is an exact triangle in the derived category of $C$-modules
\begin{equation} \label{Jacobi-Zariski}
L_{B/A} \otimes_B^L C \rightarrow L_{C/A} \rightarrow L_{C/B},
\end{equation}
where $\otimes^L_B$ denotes the derived tensor product.
\end{proposition}

In the following, note that the stated condition in Part~(\ref{tor part}) holds in particular when $R$ is a PID.

\begin{proposition}[The universal coefficient theorem] \label{universal coefficient theorem} \ 
\begin{enumerate}
\item \label{tor part} Let $R$ be a ring, let $C$ denote a chain complex of $R$-modules, and let $M$ denote an $R$-module.  Assume that, for every $R$-module $N$ and every $j \geq 2$, we have
\[
\Tor^R_j (N, M) \cong 0.
\]
Then for every $i \in \ZZ$ we have a short exact sequence of $R$-modules
\[
0 \rightarrow H^i(C) \otimes_R M \rightarrow H^i(C\otimes_R^L M) \rightarrow \Tor_1^R(H^{i+1}(C), M) \rightarrow 0,
\]
where $\otimes_R^L$ denotes the derived tensor product.
\item \label{ext part} Let $C$ denote a chain complex of abelian groups, and let $C^{\wedge}$ denote its derived $p$-completion, which by definition is equivalent to the derived Hom-complex 
\[
\mathbf{R}\Hom(\QQ_p/\ZZ_p[1], C).
\]
Then for every $i \in \ZZ$ we have a short exact sequence of abelian groups 
\[
0 \to \Ext^1(\QQ_p/\ZZ_p, H^i(C)) \to H^i(C^{\wedge}) \to \Hom(\QQ_p/\ZZ_p, H^{i+1}(C))   \to 0.
\]
\item \label{new ext part}  Let $C$ denote a chain complex of $R$-modules and let $G$ denote an abelian group.  For every $i \in \ZZ$, we have a short exact sequence of $R$-modules 
\[
0 \to \Ext^1(G, H^i(C)) \to H^i(\mathbf{R}\Hom(G[1], C)) \to \Hom(G, H^{i+1}(C))   \to 0.
\]
In particular, if we let $C^{\wedge}$ denote the derived $p$-completion of $C$, which by definition is equivalent to 
\[
\mathbf{R}\Hom(\QQ_p/\ZZ_p[1], C),
\]
then for every $i \in \ZZ$, we have a short exact sequence of $R$-modules 
\[
0 \to \Ext^1(\QQ_p/\ZZ_p, H^i(C)) \to H^i(C^{\wedge}) \to \Hom(\QQ_p/\ZZ_p, H^{i+1}(C))   \to 0.
\]
Also, for every integer $n \geq 1$ and every $i \in \ZZ$, we have a short exact sequence of $R$-modules
\[
0 \rightarrow H^{i}(C)/p^nH^i(C) \rightarrow H^i(C\otimes_{\ZZ}^L (\ZZ/p^n\ZZ)) \rightarrow H^{i+1}(C)[p^n] \rightarrow 0.
\]
\end{enumerate}
\end{proposition}

\begin{proof} This follows immediately from \cite[Section 7.4]{Rot} and \cite[Section 10.10]{Rot}. 

\end{proof}

We have the following consequence.  We will use this result to analyze $p$-power torsion in the module of K\"ahler differentials $\Omega^1_{A/\ZZ_p} \cong H^0(L_{A/\ZZ_p})$.

\begin{corollary} \label{surj to Tate}
Let $A$ denote a $\ZZ_p$-algebra.  Let $(L_{A/\ZZ_p})^{\wedge}$ denote the derived $p$-completion of $L_{A/\ZZ_p}$, and let $T_p$ denote the $p$-adic Tate module.  There is a surjective $A$-module map
\[
H^{-1}\left((L_{A/\ZZ_p})^{\wedge}\right) \rightarrow T_p(H^0(L_{A/\ZZ_p})).
\]
It is natural in the sense that, if $A \rightarrow B$ is a $\ZZ_p$-algebra map, then the following diagram commutes:
\[
\xymatrix{
H^{-1}\left((L_{B/\ZZ_p})^{\wedge}\right) \ar[r] & T_p(H^0(L_{B/\ZZ_p}))\\
H^{-1}\left((L_{A/\ZZ_p})^{\wedge}\right) \ar[r] \ar[u] & T_p(H^0(L_{A/\ZZ_p})). \ar[u]
}
\]
\end{corollary}

\begin{proof}
Note that for any chain complex $C$ and any $i \in \ZZ$, we have that $\Hom(\QQ_p/\ZZ_p, H^{i+1}(C))$ is naturally isomorphic to the p-adic Tate module $T_p(H^{i+1}(C))$.  Now if we take $C=L_{A/\ZZ_p}$, then by Proposition~\ref{universal coefficient theorem}, Part~(\ref{ext part}), we in particular get a surjective map 
\[
H^{-1}\left((L_{A/\ZZ_p})^{\wedge}\right) \rightarrow T_p(H^0(L_{A/\ZZ_p}))
\]
satisfying the stated naturality.  Moreover, this map is in fact an $A$-module map. This also follows by naturality, since the cotangent complex is a chain complex of $A$-modules. 
\end{proof}

When $R \rightarrow S$ is a smooth ring map, the cotangent complex is quasi-isomorphic to the module of K\"ahler differentials, concentrated in degree~zero \cite[Tag~08R5]{stacks-project}.  Proposition~\ref{Bhatt ex} below identifies two specific instances of ring maps $R \rightarrow S$ in which the cotangent complex is, again, quasi-isomorphic to the module of K\"ahler differentials, concentrated in degree~zero.

\begin{remark}
Bhargav Bhatt has recently communicated to us that, in fact, a significantly stronger version of Proposition~\ref{Bhatt ex}(\ref{Bhatt 12}) also holds.  Namely, if $V$ is a valuation ring and a flat (equivalently, $p$-torsion-free) $\ZZ_p$-algebra, then there is a quasi-isomorphism $L_{V/\ZZ_p} \cong \Omega^1_{V/\ZZ_p}$.  The proof of this stronger result relies on resolution of singularities in characteristic~zero and a result of Gabber and Ramero, \cite[Theorem~6.5.12]{GR03}.  Our proof of Proposition~\ref{Bhatt ex}(\ref{Bhatt 12}) follows the original strategy suggested by Bhatt in his \emph{Arizona winter school} lectures, \cite[Exercise~12]{Bha17}.  Closely related to our proof is an observation of Elmanto and Hoyois, relying on deep results of Temkin, that appears in the recent paper \cite[Proposition~4.2.1]{BD20}; we thank Elden Elmanto for calling this result to our attention.
\end{remark} 

\begin{proposition}[{\cite[Exercise~8 and Exercise~12]{Bha17}}] \label{Bhatt ex} \ 
\begin{enumerate}
\item \label{Bhatt 8} Let $L \supset K \supset \QQ_p$ denote a tower of algebraic extensions.  Then there is a quasi-isomorphism
\[
L_{\mathcal{O}_L/\mathcal{O}_K} \cong \Omega^1_{\mathcal{O}_L/\mathcal{O}_K}.
\]
\item \label{Bhatt 12} Let $\overline{V}$ denote a $p$-torsion-free $\ZZ_p$-algebra which is $p$-adically separated.  Assume that $\overline{V}$ is a valuation ring and that $\Frac \overline{V}$ is algebraically closed.  Then there is a quasi-isomorphism
\[
L_{\overline{V}/\ZZ_p} \cong \Omega^1_{\overline{V}/\ZZ_p}.
\]
\end{enumerate}
\end{proposition}

\begin{proof}
These both appear as exercises in the notes accompanying Bhatt's \emph{Arizona winter school} lectures, \cite[Exercise~8 and Exercise~12]{Bha17}.  We provide additional details for the steps suggested in those exercises.

Proof of (\ref{Bhatt 8}).  The cotangent complex commutes with filtered colimits by \cite[Tag~08S9]{stacks-project}, so $\colim L_{B_i/A_i} \cong L_{\colim(B_i)/\colim(A_i)}$, so it suffices to assume that $L/K/\QQ_p$ are finite extensions.  In this finite case, there exists a polynomial $f(x) \in \mathcal{O}_K[x]$ such that $\mathcal{O}_L$ is isomorphic to the quotient ring $\mathcal{O}_K[x]/(f(x))$ (see for example \cite[Chapter~II, Lemma~10.4]{Neu99}).  Notice that $L_{\mathcal{O}_K[x]/\mathcal{O}_K} \otimes_{\mathcal{O}_K[x]}^L \mathcal{O}_L$ has cohomology concentrated in degree zero; this follows from the fact that $L_{\mathcal{O}_K[x]/\mathcal{O}_K}$ has cohomology concentrated in degree zero, and that $H^0(L_{\mathcal{O}_K[x]/\mathcal{O}_K})$ is a flat $\mathcal{O}_K[x]$-module.
Then considering the Jacobi-Zariski sequence (\ref{Jacobi-Zariski}) associated to $\mathcal{O}_K \rightarrow \mathcal{O}_K[x] \rightarrow \mathcal{O}_L$, 
we find immediately that $H^{j}(L_{\mathcal{O}_L/\mathcal{O}_K}) \cong 0$ for $j \neq -1,0$.  It remains to show that $H^{-1}(L_{\mathcal{O}_L/\mathcal{O}_K}) \cong 0$.  To prove this, we will show that this $\mathcal{O}_L$-module is simultaneously $p$-torsion and $p$-torsion-free.

Let $S$ be the multiplicative set $\{1,p,p^2,\ldots\}$.  By \cite[Tag~08SF]{stacks-project}, we have a quasi-isomorphism
\[
L_{L/K} = L_{S^{-1} \mathcal{O}_L/S^{-1} \mathcal{O}_K} \cong L_{\mathcal{O}_L/ \mathcal{O}_K} \otimes^L_{\mathcal{O}_L} S^{-1} \mathcal{O}_L
\] 
in $D(L)$, the derived category of $L$-modules.  The extension $L/K$ is smooth (because it is a finite extension of characteristic zero fields, see for example \cite[Tag~07ND]{stacks-project}) and hence $H^{i}(L_{L/K}) \cong 0$ for all $i \neq 0$ (see for example \cite[Tag~08R5]{stacks-project}). In particular, $H^{-1}(L_{L/K}) \cong 0$, so $H^{-1}(L_{\mathcal{O}_L/ \mathcal{O}_K} \otimes^L_{\mathcal{O}_L} L) \cong 0$ by the above quasi-isomorphism, and hence $H^{-1}(L_{\mathcal{O}_L/ \mathcal{O}_K}) \otimes_{\mathcal{O}_L} L \cong 0$ (by the universal coefficient theorem Proposition~\ref{universal coefficient theorem}, using that $L$ is a flat $\mathcal{O}_L$-module).  We deduce that $H^{-1}(L_{\mathcal{O}_L/ \mathcal{O}_K})$ is $p$-torsion.

Using again the Jacobi-Zariski sequence (\ref{Jacobi-Zariski}) associated to $\mathcal{O}_K \rightarrow \mathcal{O}_K[x] \rightarrow \mathcal{O}_L$ and the fact that $H^{-1}(L_{\mathcal{O}_K[x]/\mathcal{O}_K} \otimes_{\mathcal{O}_K[x]}^L \mathcal{O}_L) \cong 0$ as explained above, we find an exact sequence 
\[
0 \rightarrow H^{-1}(L_{\mathcal{O}_L/\mathcal{O}_K}) \rightarrow H^{-1}(L_{\mathcal{O}_L/\mathcal{O}_K[x]}) \cong \mathcal{O}_L,
\]
which shows that $H^{-1}(L_{\mathcal{O}_L/\mathcal{O}_K})$ is $p$-torsion-free.  This completes the proof that there is a quasi-isomorphism $L_{\mathcal{O}_L/\mathcal{O}_K} \cong \Omega^1_{\mathcal{O}_L/\mathcal{O}_K}.$

Proof of (\ref{Bhatt 12}).  The strategy of this proof is as follows.  We use de~Jong's alterations theorem to show that every finitely generated $\ZZ_p$-subalgebra $R$ of $\overline{V}$ is contained in a \emph{regular} $\ZZ_p$-subalgebra $R_1$ of $\overline{V}$.  Thus $\overline{V}$ is a filtered colimit of such $\ZZ_p$-subalgebras $R_1$, and the desired result will follow from quasi-isomorphisms $L_{R_1/\ZZ_p} \cong \Omega^1_{R_1/\ZZ_p}$.

Let $R$ denote a finitely generated $\ZZ_p$-subalgebra of $\overline{V}$.  By de~Jong's alterations theorem \cite[Theorem~6.5]{dJ96}, there exists a regular $\ZZ_p$-scheme $X_1$ and a dominant, proper map $X_1 \rightarrow \Spec R$ such that the function field of $X_1$ is a finite extension of $\Frac R$.    Because $\Frac R \subseteq \Frac \overline{V}$ and $\Frac \overline{V}$ is algebraically closed, we deduce that the function field of $X_1$ embeds into $\Frac \overline{V}$.  These maps fit into a solid commutative diagram
\[
\xymatrix{
X_1 \ar[d] & \Spec \left( \Frac \overline{V}\right) \ar[l] \ar[d]\\
\Spec R  & \Spec \overline{V} \ar[l] \ar@{-->}[ul]
}
\]
By the valuative criterion for properness, there exists a map $\Spec \overline{V} \rightarrow X_1$ as in the dashed arrow in the diagram, for which the diagram remains commutative.  This map $\Spec \overline{V} \rightarrow X_1$ factors through some open affine subscheme $\Spec R_1 \subseteq X_1$ with $R_1$ regular and with the composite $R \rightarrow R_1 \rightarrow \overline{V}$ injective.

The rings $R_1$ as produced above (for various rings $R$) fit into a filtered diagram, with arrows corresponding to commutative diagrams
\[
\xymatrix{
R_2 \ar[r] & \overline{V} \\
R_1 \ar[u] \ar[ur]
}
\]
Because the ring $R$ above could be any finitely generated $\ZZ_p$-subalgebra of $\overline{V}$, we find that $\varinjlim R_\alpha$ is a filtered colimit, and that the corresponding map $\varinjlim R_\alpha \rightarrow \overline{V}$ is a ring isomorphism (where $R_{\alpha}$ ranges over the rings $R_1$ produced above).

The conclusion of the proof is very similar to the proof of part (\ref{Bhatt 8}).  Namely, by \cite[Tag~0E9J]{stacks-project}, we may assume each map $\ZZ_p \rightarrow R_{\alpha}$ can be factored as $\ZZ_p \rightarrow \ZZ_p[x_1, \ldots, x_n] \rightarrow \ZZ_p[x_1, \ldots, x_n]/I \cong R_{\alpha}$, where $I$ is generated by a regular sequence.  Thus $L_{R_{\alpha}/\ZZ_p[x_1, \ldots, x_n]}$ has cohomology concentrated in degree~-1, and $H^{-1}(L_{R_{\alpha}/\ZZ_p[x_1, \ldots, x_n]}) \cong I/I^2$, by \cite[Tag~08SJ]{stacks-project}.   Proceeding as in the proof of (\ref{Bhatt 8}), we find that $L_{R_{\alpha}/\ZZ_p} \cong \Omega^1_{R_{\alpha}/\ZZ_p}$ and that $L_{\overline{V}/\ZZ_p} \cong \Omega^1_{\overline{V}/\ZZ_p}.$
\end{proof}

The previous result had a hypothesis requiring a certain ring to be a valuation ring, and in fact, valuation rings play a special role at several points in this paper.  One reason is because, if $A$ is a $p$-torsion-free perfectoid ring containing a compatible sequence of $p$-power roots of unity, then $A$ contains an isomorphic copy of the valuation ring $\ZZ_p[\zeta_{p^{\infty}}]^{\wedge}$.  A second reason is because every $p$-torsion-free perfectoid ring embeds into a product of perfectoid valuation rings (see Lemma~\ref{embed into product} and its proof).  

We use \cite[Chapter~VI]{ZS60} as our basic reference on valuation rings, although we write our valuations multiplicatively.  (See \cite[Remark~2.3]{Sch12} for a remark on why Scholze writes his valuations multiplicatively.)  We emphasize that our valuation rings are not assumed to be rank~one.  (For an alternative to Zariski-Samuel, see \cite[Chapter~VI]{Bou06}.)

\begin{theorem} \label{valuations extend}
Let $K_1 \supseteq K_0$ denote two fields and assume there is a valuation on $K_0$.  Then there is a valuation on $K_1$ extending the valuation on $K_0$. 
\end{theorem}

\begin{proof}
This fact is well-known.  For example, see \cite[Chapter~VI, Theorem~$5'$]{ZS60} for the fact that places can be extended, and \cite[Chapter~VI, Section~9]{ZS60} for the correspondence between places and valuations.  
\end{proof}

\begin{proposition} \label{alg faithfully flat}
Let $K_0$ denote a field equipped with a valuation, and let $\mathcal{O}_{K_0}$ denote the corresponding valuation ring.  Let $K_1$ denote an algebraic extension of $K_0$.  By Theorem~\ref{valuations extend}, the field $K_1$ can be equipped with a valuation extending the valuation on $K_0$.  The corresponding map on valuation rings $\mathcal{O}_{K_0} \rightarrow \mathcal{O}_{K_1}$ is faithfully flat.
\end{proposition}

\begin{proof}
Every valuation ring is in particular a Pr\"ufer domain, and so $\mathcal{O}_{K_1}$ is flat over $\mathcal{O}_{K_0}$ because it is torsion-free.  We will now show that the induced map $\Spec \mathcal{O}_{K_1} \rightarrow \Spec \mathcal{O}_{K_0}$ is surjective, which will complete the proof that $\mathcal{O}_{K_0} \rightarrow \mathcal{O}_{K_1}$ is faithfully flat.  Let $v$ denote the valuation on $K_1$ extending the given valuation on $K_0$.  The most important preliminary result we will need is the fact that the quotient group $v(K_1)/v(K_0)$ is torsion or, equivalently, for every element $x \in K_1$, there exists some integer $n > 0$ such that $v(x^n) \in v(K_0)$.  Let us name this condition (*).  We do not prove it; it is stated explicitly as \cite[Chapter~VI, Section~11, Lemma~1]{ZS60}.

Consider any prime ideal $\mathfrak{p}_0 \subseteq \mathcal{O}_{K_0}$.  We are trying to show that there exists some prime ideal $\mathfrak{p}_1 \subseteq \mathcal{O}_{K_1}$ such that $\mathfrak{p}_1 \cap \mathcal{O}_{K_0} = \mathfrak{p}_0$.  We may assume $\mathfrak{p}_0$ is not the zero ideal.  Our construction is based on \cite[Chapter~VI, Theorem~15]{ZS60}.  Define 
\[
\mathfrak{p}_1 := \{y \in \mathcal{O}_{K_1} \colon v(y)^n = v(x) \text{ for some } n \in \ZZ_{>0}, x \in \mathfrak{p}_0\}.
\]
Using condition~(*) above, we see that $\mathfrak{p}_1$ is an ideal in $\mathcal{O}_{K_1}$.  It is a proper ideal.  Lastly, assume $ab \in \mathfrak{p}_1$.  Say $v(a^n b^n) = v(x)$, where $x \in \mathfrak{p}_0$.  After possibly raising to a higher power, using again condition~(*), we may assume that $a^n, b^n \in \mathcal{O}_{K_0}$.  Then $a^n b^n/x \in \mathcal{O}_{K_0}$, so $a^n b^n \in \mathfrak{p}_0$.  Then because $\mathfrak{p}_0$ is a prime ideal, we have that $a^n$ or $b^n$ is in $\mathfrak{p}_0$.  Thus $a$ or $b$ is in $\mathfrak{p}_1$.  This completes the proof that $\mathfrak{p}_1$ is a prime ideal.

It remains to check that $\mathfrak{p}_1 \cap \mathcal{O}_{K_0} = \mathfrak{p}_0$.  Clearly $\mathfrak{p}_1 \cap \mathcal{O}_{K_0} \supseteq \mathfrak{p}_0$.  Conversely $y \in \mathcal{O}_{K_0}$ and $y \in \mathfrak{p}_1$, then $v(y^n) = v(x)$ for some $x \in \mathfrak{p}_0$ and some $n \in \ZZ_{> 0}$, but then $y^n/x \in \mathcal{O}_{K_0}$, so $y^n \in \mathfrak{p}_0$, so $y \in \mathfrak{p}_0$.  This proves the reverse inclusion, $\mathfrak{p}_1 \cap \mathcal{O}_{K_0} \subseteq \mathfrak{p}_0$.  This completes the proof that $\mathcal{O}_{K_0} \rightarrow \mathcal{O}_{K_1}$ is faithfully flat.
\end{proof}

Later in this section we will need to embed a $p$-torsion-free perfectoid valuation ring into a larger perfectoid valuation ring for which its fraction field is algebraically closed.  In order to accomplish this, we will need the next few preliminary results.

\begin{lemma} \label{separated cofinal}
Let $K$ be a characteristic~zero field equipped with a valuation $v$ for which $\mathcal{O}_K$ is $p$-adically separated.  The values $v(p), v(p^2), \ldots, v(p^n), \ldots$ are cofinal in the set of all values of non-zero elements in $K$, in the sense that for every non-zero element $x \in K$, there exists an integer $n \geq 1$ such that $v(p^n) < v(x)$.
\end{lemma}

\begin{proof}
If $x \in K \setminus \mathcal{O}_K$, then any value of $n$ will work.  Otherwise, because $\mathcal{O}_K$ is $p$-adically separated, we may choose $n$ such that $x \not\in p^n \mathcal{O}_K$.
\end{proof}

We next refer to the \emph{completion} of a field $K$ with respect to a valuation.  There seems to be some subtlety involved in defining the notion of completion for general valued fields (see \cite[Section~2.4]{EP05}), but in our case it is easier, because Lemma~\ref{separated cofinal} shows that there exists a countable sequence of elements which are cofinal in the value group, so the elements in the completion will correspond to Cauchy sequences $(a_1, a_2, \ldots)$, where the index set is $\mathbb{N}$.

\begin{lemma} \label{Jan20a}
Let $K$ be a characteristic~zero field equipped with a valuation for which $\mathcal{O}_K$ is $p$-adically separated.  Let $L$ denote the valued field which is the completion of $K$, as in \cite[Theorem~2.4.3]{EP05}.  The corresponding valuation ring $\mathcal{O}_L$ is $p$-adically complete and separated.
\end{lemma}

\begin{proof}
We must show that every $p$-adic Cauchy sequence of elements in $\mathcal{O}_L$ has a unique limit in $\mathcal{O}_L$.  The existence of such a limit in $L$ follows immediately from Lemma~\ref{separated cofinal} and the construction of $L$ as equivalence classes of Cauchy sequences.  If the elements in the Cauchy sequence are all in $\mathcal{O}_L$, then the values of the elements in the Cauchy sequence are all at most $1 = v(1)$.  The value of the Cauchy sequence, viewed as an element of $L$, is either $0$ (in which case it is in $\mathcal{O}_L$), or is equal to the value of an entry in the sequence (in which case it is also in $\mathcal{O}_L$).  In either case, we see that every $p$-adic Cauchy sequence of elements in $\mathcal{O}_L$ converges in $\mathcal{O}_L$.

We next show that $\mathcal{O}_L$ is $p$-adically separated.  If the Cauchy sequence corresponds to $0$ in $L$, then there is nothing to check.  Otherwise, the value of the Cauchy sequence is equal to the value of one of its non-zero entries; call that entry $x$.  Because $\mathcal{O}_K$ is $p$-adically separated, we have that $v(p^n) < v(x)$ for some integer $n \geq 1$.  Thus the element corresponding to the Cauchy sequence is not in $p^n \mathcal{O}_L$.  This shows that $\mathcal{O}_L$ is $p$-adically separated.
\end{proof}

\begin{lemma} \label{Jan20b}
Let $L$ denote a characteristic~zero field equipped with a valuation for which $\mathcal{O}_L$ is $p$-adically separated.  Let $L'$ denote an algebraic extension of $L$.  Then $\mathcal{O}_{L'}$ is also $p$-adically separated.
\end{lemma}

\begin{proof}
This follows directly from condition~(*) that was named in the proof of Proposition~\ref{alg faithfully flat}.
\end{proof}

\begin{lemma} \label{still algebraically closed}
Let $K$ be a characteristic zero field equipped with a valuation for which $\mathcal{O}_K$ is $p$-adically separated, and let $L$ denote the completion of $K$, as in \cite[Theorem~2.4.3]{EP05}.  If $K$ is algebraically closed, then so is $L$.
\end{lemma}

\begin{proof}
Much of this proof is taken verbatim from notes written by Brian Conrad \cite{ConradNotes}.  (Those notes are phrased in terms of absolute values rather than valuations.  We have attempted to translate Conrad's proof into the setting of possibly higher rank valued fields.)  We first outline the proof strategy.  We must show that every non-constant polynomial $f(x)$ in $L[x]$ has a zero in $L$.  We can approximate $f(x)$ by polynomials $f_j(x) \in K[x]$, each of which has a zero $r_j \in K$, because $K$ is algebraically closed.  We will show that some subsequence of $(r_j)$ converges to a zero of $f(x)$.  Because $L$ is the completion of $K$, this will imply that some zero of $f(x)$ lies in $L$, as desired.

More precisely, the following are the key steps.  Let $\Gamma$ denote the value group of $K$.
\begin{enumerate}
\item There exists $\gamma \in \Gamma$ such that $v(r_j) \leq \gamma$ for all $j$.
\item The sequence $f(r_j) \in L$ converges to zero.
\item Let $\lambda_1, \ldots, \lambda_n$ denote the zeros of $f(x)$ in some field $L'$ which is an algebraic extension of $L$.  We have that the sequence $\big(\min_i v(r_j - \lambda_i)\big)$ approaches zero as $j$ approaches infinity, and hence, because there are only finitely many values of $\lambda_i$, some subsequence of $(r_j)$ is a Cauchy sequence converging to $\lambda_k$ for some $1 \leq k \leq n$.  Thus $\lambda_k \in L$, as required.
\end{enumerate}

We now carry out these steps.  The proof of \cite[Theorem~2.4.3]{EP05} describes the valuation on $L$, and by Theorem~\ref{valuations extend}, we can further extend the valuation on $L$ to a valuation on any field extension $L'$.  If $L'/L$ is algebraic, then $\mathcal{O}_{L'}$ is $p$-adically separated by Lemma~\ref{Jan20a} and Lemma~\ref{Jan20b}.

  We may assume our polynomial is separable.  Let $f(x) = x^n + a_{n-1}  x^{n-1} + \cdots + a_0$ denote an arbitrary separable, non-constant polynomial in $L[x]$.  It is also convenient for the proof below to assume that $a_0 \neq 0$ and $n \geq 2$.  Write $\lambda_1, \ldots, \lambda_n$ for the zeros of $f(x)$ in some field $L'$ which is algebraic over $L$.   We can approximate $f(x)$ by polynomials $f_j(x) = x^n + a_{n-1, j} x^{n-1} + \cdots + a_{0, j} \in K[x]$ satisfying the following two conditions:
\begin{enumerate}
\item If $a_i = 0$, then $a_{i,j} = 0$ for all $j$.
\item If $a_i \neq 0$, then $v(a_{i,j} - a_i) < \min(v(a_i), v(p^j))$ for all $j$.
\end{enumerate}
These conditions imply that $v(a_i) = v(a_{i,j})$ for all $i$ and all $j$.

For each $j$, let $r_j \in K$ denote a root of $f_j(x)$.  Because $f_j(r_j) = 0$, we have that
\[
v(r_j^n) = v\left( \sum_{i = 0}^{n-1} a_{i,j} r_j^i \right) \leq \max_i v(a_{i,j}) v(r_j)^i = \max_i v(a_i) v(r_j)^i.
\]
Hence, for each $j$, there exists some $i(j) \in \ZZ$ in the range $0 \leq i(j) \leq n-1$ such that
\[
v(r_j)^n \leq v(a_{i(j)}) v(r_j)^{i(j)}.
\]
Thus
\[
v(r_j) \leq \gamma := \max_{i} v(a_i)^{1/(n-i)}.
\]
(Note that, because $K$ is algebraically closed, the group element $v(a_i)^{1/(n-i)}$ makes sense.)

Because $f, f_j$ are monic polynomials of the same degree and because $f_j(r_j) = 0$ by our assumption, we have
\begin{align*}
v(f(r_j)) &= v(f(r_j) - f_j(r_j)) \\
&= v\left( \sum_{i = 0}^{n-1} (a_i - a_{i,j}) r_j^i \right) \\
&\leq \max_{0 \leq i \leq n-1} v(a_i - a_{i,j}) v(r_j)^i.\\
\intertext{We know $v(r_j) \leq \gamma$, so $v(r_j)^i \leq \gamma^i$, and in particular if $\gamma < 1$, then $v(r_j)^i \leq 1$.  If on the other hand $\gamma \geq 1$, then $v(r_j)^i \leq \gamma^{n-1}$.   In total, we deduce}
v(f(r_j)) &\leq \max_i v(a_i - a_{i,j}) \cdot \max \{ 1, \gamma^{n-1} \}.\\
\intertext{By our choice of the coefficients $a_{i,j}$, we have}
v(f(r_j)) &\leq v(p^j) \cdot \max \{ 1, \gamma^{n-1} \}.
\end{align*}
By Lemma~\ref{separated cofinal}, there exists some integer $m \geq 1$ such that $v(p^m) < \frac{1}{\max \{ 1, \gamma^{n-1} \}}$, and hence $\max \{ 1, \gamma^{n-1} \} < v(p^{-m})$ and therefore $v(f(r_j)) < v(p^{j-m})$ for all $j$.  Thus the sequence $f(r_j) \in L$ converges to 0.  

Recall that we denoted the roots of $f(x)$ by $\lambda_i \in L'$, so $f(x) = \prod_{i=1}^n (x - \lambda_i) \in L'[x]$, where $L'$ is some fixed algebraic extension of $L$.  Thus for all $j$, we have that
\begin{align*}
\prod_{i=1}^n v(r_j - \lambda_i) &\leq v(p^j) \cdot \max \{ 1, \gamma^{n-1} \},\\
\intertext{and so for all $j$, we have that}
\min_i v(r_j - \lambda_i) &\leq v(p^j)^{1/{n}} \cdot \max \{ 1, \gamma^{(n-1)/n}\}.\\
\intertext{By the pigeonhole principle, there is some $0 \leq i_0 \leq n-1$ for which the inequality}
v(r_j - \lambda_{i_0}) &\leq v(p^j)^{1/{n}} \cdot \max \{ 1, \gamma^{(n-1)/n}\}
\end{align*}
holds for infinitely many values of $j$.  As in the previous paragraph, some subsequence of $(r_j - \lambda_{i_0})$ converges to $0$, and hence some subsequence of $(r_j)$ converges to $\lambda_{i_0}$.  Because the elements $r_j$ are all in $K$, and because $L$ is the completion of $K$, it follows that some subsequence of $r_j$ converges to some element $\lambda \in L$.   Because $f(x) \in L[x]$ is a polynomial and $f(r_j)$ converges to $0$ in $L$, we have that $\lambda \in L$ satisfies $f(\lambda) = 0$, as desired.  (Alternatively, it could be shown that in fact $\lambda = \lambda_{i_0}$, using the fact that $\mathcal{O}_{L_1}$ is $p$-adically separated, but that argument requires consideration of the case $\lambda, \lambda_{i_0} \in L_1 \setminus \mathcal{O}_{L_1}$.)  This completes the proof that the completion of $K$ is algebraically closed.  
\end{proof}

We next discuss perfectoid rings.  We follow the presentation and notation of Bhatt-Morrow-Scholze's \cite[Section~3]{BMS16}.  Many of these next results also correspond to related results in Hesselholt's \cite[Section~1.2]{Hes06}.  For example, Lemma~\ref{inverse limit iso} is closely related to \cite[Proposition~1.2.3 and Addendum~1.2.4]{Hes06}.

\begin{lemma}[{\cite[Lemma~3.2(i)]{BMS16}}] \label{tilt}
Let $A$ denote a ring which is $p$-adically complete and separated.  Define the \emph{tilt} of $A$, denoted $A^{\flat}$, to be the ring $\displaystyle A^{\flat} := \varprojlim_{x \mapsto x^p} A/pA$.  The  map which reduces each element modulo~$p$, 
\[
\varprojlim_{x \mapsto x^p} A \rightarrow \varprojlim_{x \mapsto x^p} A/pA = A^{\flat},
\]
is an isomorphism of multiplicative monoids.
\end{lemma}

\begin{notation} \label{inv F notation}
For any ring $A$, we write 
$\varprojlim_F W_r(A)$ for the ring which is the inverse limit of the diagram 
\[
\cdots \xrightarrow{F} W_3(A) \xrightarrow{F} W_2(A) \xrightarrow{F} W_1(A),
\] where the transition maps are the finite-length Witt vector Frobenius.  For any integer $n \geq 1$, we let $\text{pr}_n: \varprojlim_F W_r(A) \rightarrow W_n(A)$ denote the projection.
\end{notation}

\begin{lemma} \label{inverse limit iso}
Let $A$ denote a ring which is $p$-adically complete and separated.  Let notation be as in Lemma~\ref{tilt} and Notation~\ref{inv F notation}.  There is a unique $p$-adically continuous ring homomorphism
\[
\pi: W(A^{\flat}) \rightarrow \varprojlim_F W_r(A)
\]
such that, for every integer $r \geq 1$ and every $ (t^{(0)}, t^{(1)}, \ldots) \in \varprojlim_{x \mapsto x^p} A$ (identified with an element $t \in A^{\flat}$ as in Lemma~\ref{tilt}), we have 
\[
(\text{pr}_r \circ \pi)([t]) = [t^{(r)}] \in W_r(A).
\]
The map $\pi$ is an isomorphism of rings.
\end{lemma}

\begin{proof}
See for example \cite[Lemma~3.2]{BMS16}, and the two paragraphs following that proof.  Uniqueness follows immediately from the fact that every element in $W(A^{\flat})$ is a $p$-adic combination $\sum p^i [t_i]$, where $t_i \in A^{\flat}$.
\end{proof}

\begin{definition} \label{theta def}
Let $A$ be a ring which is $p$-adically complete and $p$-adically separated.  For every integer $r \geq 1$, we define $\widetilde{\theta}_r$ to be the ring homomorphism
\[
\widetilde{\theta}_r := \text{pr}_r \circ \pi: W(A^{\flat}) \rightarrow W_r(A)
\] 
from Lemma~\ref{inverse limit iso}.  We define the ring homomorphism $\theta: W(A^{\flat}) \rightarrow A$ to be the composite $\widetilde{\theta}_1 \circ F$, where $F: W(A^{\flat}) \rightarrow W(A^{\flat})$ is the Witt vector Frobenius.
\end{definition}

\begin{remark}
Let $t \in A^{\flat}$ be arbitrary, and assume $t$ corresponds to $ (t^{(0)}, t^{(1)}, \ldots) \in \varprojlim_{x \mapsto x^p} A$. We have
\[
\theta([t]) = t^{(0)}.
\]
\end{remark}

\begin{definition}[{\cite[Definition~3.5]{BMS16}}] \label{perfectoid definition}
We say a ring $A$ is perfectoid if the following three conditions hold:
\begin{enumerate}
\item \label{pi condition} The ring $A$ is $\pi$-adically complete and separated for some element $\pi$ such that $\pi^p$ divides $p$.
\item The $p$-power Frobenius $A/pA \rightarrow A/pA$ is surjective.
\item The kernel of $\theta: W(A^{\flat}) \rightarrow A$ (with $\theta$ as in Definition~\ref{theta def}) is a principal ideal.
\end{enumerate}
\end{definition}

\begin{example}
Assume $A$ is a $p$-torsion-free ring containing a primitive $p$-th root of unity, $\zeta_p$.  Then Condition~(\ref{pi condition}) in the definition of \emph{perfectoid} can be replaced with 
\begin{enumerate}
\item[(\ref*{pi condition}$^{\prime}$)]  The ring $A$ is $p$-adically complete and separated.
\end{enumerate}
Indeed, on one hand, a perfectoid ring is $p$-adically complete and separated.  On the other hand, if the $p$-power map is surjective modulo~$p$ on a $p$-torsion-free ring $A$ containing $\zeta_p$, then there exists some $\pi, a \in A$ such that 
\[
\pi^p = (\zeta_p - 1) + pa.
\]
If $A$ is $p$-adically complete, then we deduce that $\pi^p$ divides $\zeta_p - 1$, and hence $\pi^p$ divides $p$.  On the other hand, $p$ divides $\pi^{p^2}$, so the ring $A$ is $\pi$-adically complete and separated.
\end{example}

\begin{remark}
We typically use \cite{BMS16} as our reference, but many of these properties were studied earlier.  For example, the significance of $\theta$ was recognized by Fontaine, and the isomorphism between $W(A^{\flat})$ and $\varprojlim_F W_r(A)$ appears (in a slightly different context) in Hesselholt's \cite[Addendum~1.2.4]{Hes06}.
An isomorphism $W(A^{\flat}) \cong \varprojlim_{F} W_r(A)$ was also studied by the first author and Kedlaya in \cite[Theorem~3.6]{DK15}, but that isomorphism differs from the isomorphisms of Hesselholt and Bhatt-Morrow-Scholze.  More precisely, the isomorphism from \cite[Theorem~3.6]{DK15} is attained from the isomorphism in \cite[Lemma~3.2]{BMS16} by first applying the Witt vector Frobenius automorphism on $W(A^{\flat})$, as indicated in the following commutative diagram
\[
\xymatrix{
W(A^{\flat}) \ar^{F}[rr] \ar_{\text{\cite{DK15}}}^{\sim}[dr]& & W(A^{\flat}) \ar^{\text{\cite{BMS16} or \cite{Hes06}}}_{\sim}[dl]\\
& \displaystyle \varprojlim_F W_r(A)
}
\]
The present paper uses the normalizations of Hesselholt and Bhatt-Morrow-Scholze.
\end{remark}

In general, a perfectoid ring $A$ and its tilt $A^{\flat}$ may contain zero divisors, and hence $W(A^{\flat})$ may also contain zero divisors.  On the other hand, every generator of the ideal $\ker \theta \subseteq W(A^{\flat})$ is a non-zero-divisor, as we recall in the next result.

\begin{lemma}[{\cite[Lemma~3.10 and Remark~3.11]{BMS16}}] \label{xi lemma}
Let $A$ denote a perfectoid ring, and let $\xi \in W(A^{\flat})$ be any generator of $\ker \theta$.
\begin{enumerate}
\item The element $\xi$ is a non-zero-divisor in $W(A^{\flat})$.
\item The generators $\xi$ are functorial in the following sense.  Let $B$ denote another perfectoid ring, with corresponding theta map $\theta_B: W(B^{\flat}) \rightarrow B$.  Let $f: A \rightarrow B$ denote a ring homomorphism.  By functoriality of tilts and Witt vectors, the map $f$ induces a map $W(f^{\flat}): W(A^{\flat}) \rightarrow W(B^{\flat})$, and the element $W(f^{\flat})(\xi) \in W(B^{\flat})$ is a generator of $\ker(\theta_B)$.
\end{enumerate}
\end{lemma}

One of the deeper results concerning perfectoid rings which we will need is the following.

\begin{lemma} \label{embed into product}
Every $p$-torsion-free perfectoid ring $A$ embeds into a product of $p$-torsion-free perfectoid valuation rings, $\prod \overline{V}_{\alpha},$ for which each $\Frac \overline{V}_{\alpha}$ is algebraically closed.
\end{lemma}

\begin{proof}
It is shown in \cite[Proof of Proposition~4.19]{BMS18} that every perfectoid ring embeds into a product of perfectoid valuation rings.  Every valuation ring is either $p$-torsion-free or is annihilated by $p$, and because our ring $A$ is $p$-torsion-free, if $A$ embeds into a product of perfectoid valuation rings $\prod V_{\alpha}$, it also embeds into the possibly smaller product consisting of only those $V_{\alpha}$ which are $p$-torsion-free.  (If $a$ maps to 0 in this smaller product, then $pa$ maps to 0 in the original product, and hence $pa = 0$, and hence $a = 0$.)  For each $\alpha$, let $K_\alpha := \Frac V_{\alpha}$ and let $L_{\alpha}$ denote the completion (as in \cite[Theorem~2.4.3]{EP05}) of an algebraic closure of $K_{\alpha}$.  Let $W_{\alpha}$ denote the valuation ring in $L_{\alpha}$.  By Lemma~\ref{still algebraically closed}, $L_{\alpha} = \Frac W_{\alpha}$ is algebraically closed, so it suffices to show that $W_{\alpha}$ is perfectoid.  The ring $W_{\alpha}$ is $p$-adically separated by Lemma~\ref{Jan20a} and Lemma~\ref{Jan20b}, and $W_{\alpha}$ is $p$-adically complete by construction.  Let $\pi \in W_{\alpha}$ denote an element for which $\pi^p = p$.  From the fact that $L_{\alpha}$ is algebraically closed, it's clear that such an element $\pi$ exists and that the $p$-power map $W_{\alpha}/\pi W_{\alpha} \rightarrow W_{\alpha}/\pi^p W_{\alpha}$ is surjective.  By \cite[Lemma~3.10(ii)]{BMS16}, to show that $W_{\alpha}$ is perfectoid, it suffices to show that this $p$-power map $W_{\alpha}/\pi W_{\alpha} \rightarrow W_{\alpha}/\pi^p W_{\alpha}$ is also injective, but this is obvious from the fact that $W_{\alpha}$ is a valuation ring: if $\pi^p \mid w^p$, then $v(\pi)^p \geq v(w)^p$, and so $v(\pi) \geq v(w)$, and so $\pi \mid w$.  
\end{proof}

We next transition to our algebraic results concerning Witt vectors.  A foundational result is the following.

\begin{lemma}[{\cite[Lemma~1.1.1]{Hes06}}] \label{Witt complete} 
Let $A$ denote a $p$-adically complete and $p$-torsion-free ring.  Then for each integer $n \geq 1$, the ring $W_n(A)$ is also $p$-adically complete and $p$-torsion-free.
\end{lemma}

Most of our Witt vector results in this section are less foundational and more specialized.  We briefly indicate the significance of these technical results.  Theorem~\ref{Tate theorem} relates $p$-torsion in the de\thinspace Rham-Witt complex to $W_n(A)$-modules of the form $W_n(A)/pW_n(A)$.  In almost every situation, it is easier to prove results about $W_n(A)/pW_n(A)$ than to prove results about $W_n\Omega^1_A[p]$.  In this section we gather several algebraic results concerning $W_n(A)/pW_n(A)$.  Much of our strategy involves relating these modules for different values of $n$, and there are two maps we typically use, Witt vector Frobenius and restriction.  So we also gather here some algebraic results concerning Frobenius and restriction.

\begin{lemma} \label{general iso}
Let $A$ denote a perfectoid ring.  Assume $f: A/pA \rightarrow A/pA$ is a map of $W_{n+1}(A)$-modules, where $A/pA$ is equipped with a $W_{n+1}(A)$-module structure via $F^n$.  If there exists a unit $u \in A/pA$ which is in the image of $f$, then $f$ is an isomorphism.
\end{lemma}

\begin{proof}
Using that $F^n: W_{n+1}(A) \rightarrow W_1(A)$ is surjective, we find that $f(a) = af(1)$ for every $a \in A/pA$, so it suffices to prove that $f(1)$ is a unit in $A/pA$, and this follows by our assumption, which shows that some multiple of $f(1)$ is a unit.
\end{proof}

The ring $\ZZ_p[\zeta_{p^{\infty}}]^{\wedge}$ plays a special role in this paper, especially in the proof of Theorem~\ref{Tate theorem}.  The next few results concern this ring.

\begin{lemma} \label{flat valuation ring}
Let $A_0$ denote the valuation ring $\ZZ_p[\zeta_{p^{\infty}}]^{\wedge}$.  Then the ring $A_0^{\flat}$ is a valuation ring.
\end{lemma}

\begin{proof}
The authors noticed this argument in notes from a course of Bhargav Bhatt \cite{Bha17b}.  We first observe that $A_0^{\flat}$ is an integral domain; indeed, elements in $A_0^{\flat}$ are uniquely representable by $p$-power-compatible sequences of elements in $A_0$ (as opposed to $A_0/pA_0$, see for example \cite[Lemma~3.2(i)]{BMS16}), and because $A_0$ is an integral domain, it is then clear that $A_0^{\flat}$ does not have any zero divisors.  We next show that if $x,y \in A_0^{\flat}$, then either $x$ divides $y$ or $y$ divides $x$.  Write $x = (x_i)$ where $x_i \in A_0$ and $x_i^p = x_{i-1}$ and similarly for $y = (y_i)$.  Notice that in an integral domain which is a valuation ring, $x_{i}$ divides $y_{i}$ if and only if $x_i^p$ divides $y_i^p$: the forward direction is obvious.  For the other direction, assume $x_i^p a = y_i^p$ and $y_i b = x_i$.  Then we have $x_i^p a b^p = x_i^p$, and hence (unless $x_i = 0$) $b$ is a unit, and hence $x_i$ divides $y_i$.  The result now follows.
\end{proof}

\begin{lemma} \label{unit exists}
Let $A_0 = \ZZ_p[\zeta_{p^{\infty}}]^{\wedge}$ and let $n \geq 1$ be an integer.  
The ring $W_n(A_0)/pW_n(A_0)$ has the following property: If $N \subseteq W_n(A_0)/pW_n(A_0)$ is an ideal and $x,y \in W_n(A_0)/pW_n(A_0)$ are such that $x, y \not\in N$ and $x \equiv y \bmod N$, then there exists a unit $u \in W_n(A_0)/pW_n(A_0)$ such that $x = uy$.
\end{lemma}

\begin{proof}
We have a surjective ring homomorphism $\widetilde{\theta}_n: W(A_0^{\flat}) \rightarrow W_n(A_0)$ (because $A_0$ is perfectoid and every perfectoid ring is Witt-perfect; see for example \cite[Lemma~3.9(iv)]{BMS16}), and hence we have a surjective ring homomorphism $A_0^{\flat} \cong W(A_0^{\flat})/pW(A_0^{\flat}) \rightarrow W_n(A_0)/pW_n(A_0)$.  By Lemma~\ref{flat valuation ring}, this means $W_n(A_0)/pW_n(A_0)$ is a quotient of a valuation ring.

Write $x = y + z$, where $z \in N$.  Let $x', y' \in A_0^{\flat}$ denote elements mapping to $x, y$ (respectively) under this surjection $A_0^{\flat} \twoheadrightarrow W_n(A_0)/pW_n(A_0)$.  We find that $z' := x' - y' \in A_0^{\flat}$ maps to $z$.  We cannot have $\frac{y'}{z'} \in A_0^{\flat}$, because we know that $y \not\in N$, and hence $y$ is not a multiple of $z$ in $W_n(A_0)/pW_n(A_0)$, and hence $y'$ is not a multiple of $z'$ in $A_0^{\flat}$.  Then because $A_0^{\flat}$ is a valuation ring, we know $\frac{z'}{y'} \in A_0^{\flat}$, and $(1 + \frac{z'}{y'}) y' = x' \in A_0^{\flat}$.  In other words, $x'$ is a multiple of $y'$ in $A_0^{\flat}$.  Reversing the roles of $x$ and $y$ in the argument, we find that $y'$ is also a multiple of $x'$ in $A_0^{\flat}$.  Because $A_0^{\flat}$ is an integral domain, $x'$ is a unit multiple of $y'$, and hence $x$ is a unit multiple $y$ in $W_n(A_0)/pW_n(A_0)$, as required.
\end{proof}

Here is a more basic result.

\begin{lemma} \label{units in Witt}
Let $A$ denote a $p$-adically complete ring.  An element $x \in W_n(A)$ is a unit if and only if its projection $x \in W_n(A)/pW_n(A)$ is a unit.  Furthermore, $x \in W_n(A)$ is a unit if and only if its first Witt coordinate is a unit in $A$.
\end{lemma}

\begin{proof}
Certainly if $x$ is a unit, then its projection to any quotient ring is a unit.  On the other hand, if $x,y,z \in W_n(A)$ satisfy $xy = 1 + pz$, then because $W_n(A)$ is $p$-adically complete (Lemma~\ref{Witt complete}), there exists $u \in W_n(A)$ such that $(1+pz)u = 1 \in W_n(A)$.  Thus $x(yu) = 1 \in W_n(A)$.  This shows that if the image of $x$ is a unit in $W_n(A)/pW_n(A)$, then $x$ is a unit in $W_n(A)$.

The proof of the second assertion is similar.  If $x \in W_n(A)$ is a unit, then clearly the first Witt coordinate of $x$ is a unit in $A$.  Conversely, if the first Witt coordinate of $x$ is a unit, then
\[
x = [x_0] + V(y)
\]
for some $y \in W_{n-1}(A)$.  If we multiply by the unit $[x_0^{-1}] \in W_n(A)$, we reduce to showing that every Witt vector in $W_n(A)$ of the form $1 + V(y')$ is a unit.  This again follows because $W_n(A)$ is $p$-adically complete, using that $\big(V(y')\big)^m \in p^{m-1} W_{n}(A)$ for every integer $m \geq 1$.
\end{proof}

A similar result is the following.

\begin{lemma} \label{zeta_p congruence implies unit}
Assume $A$ is a $p$-torsion-free perfectoid ring containing a primitive $p^{n}$-th root of unity $\zeta_{p^{n}}$ for some integer $n \geq 1$.  (Recall that throughout this paper, $p \neq 2$.)
If $y \in W_{n}(A)$ satisfies $([\zeta_{p^{n}}] - 1) y = ([\zeta_{p^{n}}] - 1) \in W_{n}(A)/pW_{n}(A)$, then $y$ is a unit in $W_{n}(A)/pW_{n}(A)$. 
\end{lemma}

\begin{proof}
Let $y_0$ denote the first Witt component of $y$.
The condition implies 
\[
(\zeta_{p^n} - 1) (y_0 - 1) = pb
\]
for some $b \in A$.  Because $A$ is $p$-torsion-free, it is also $(\zeta_{p^n} - 1)$-torsion-free, so we have that
\[
y_0 \equiv 1 \bmod \frac{p}{\zeta_{p^n} - 1} A.
\]
Because $A$ is $p$-adically complete (since it is perfectoid), it is also $\frac{p}{\zeta_{p^n} - 1}$-adically complete.  Thus $y_0$ is a unit in $A$ and thus by Lemma~\ref{units in Witt}, $y$ is a unit in $W_n(A)$.
\end{proof}

We next identify the kernels of some maps related to Frobenius in the special case that the ring $A$ contains a compatible sequence of primitive $p$-power roots of unity.  The following results are all closely related to each other.  Many of these results have counterparts in \cite[Section~1.2]{Hes06} and \cite[Section~3]{BMS16}.

\begin{lemma} \label{theta lemma}
Let $A$ denote a $p$-torsion-free perfectoid ring, and assume that $A$ contains a compatible sequence $\zeta_p, \zeta_{p^2}, \ldots$ of primitive $p$-power roots of unity.  Fix an integer $n \geq 1$, and let $\widetilde{\theta}_n$ be as in Definition~\ref{theta def}. Define $\varepsilon \in A^{\flat}$ by the $p$-power compatible sequence $(1, \zeta_p, \zeta_{p^2}, \ldots)$.  The following properties hold:
 \begin{enumerate}
 \item \label{eps} We have $\tilde{\theta}_n([\varepsilon]) = [\zeta_{p^n}]$.  
 \item \label{theta surj} The map $\tilde{\theta}_n$ is surjective.
 \item \label{ker BMS} The kernel of $\tilde{\theta}_n$ is the principal ideal generated by $\sum_{i = 0}^{p^n - 1} [\varepsilon]^i.$  
 \end{enumerate}
\end{lemma}

\begin{proof}
Property~(\ref{eps}) is a special case of the required property stated in Lemma~\ref{inverse limit iso}.  Property~(\ref{theta surj}) holds because $F: W_{r+1}(A) \rightarrow W_r(A)$ is surjective for every integer $r \geq 1$; see \cite[Lemma~3.9(iv)]{BMS16}.  Property~(\ref{ker BMS}) is stated directly in \cite[Example~3.16]{BMS16}.
\end{proof}

\begin{lemma} \label{ker Fn}
Let $A$ denote a $p$-torsion-free perfectoid ring containing a compatible system of $p$-power roots of unity.  The kernel of $F^n: W_{n+1}(A) \rightarrow W_1(A)$ is the principal ideal generated by $z_{n+1} = 1 + [\zeta_{p^{n+1}}] + \cdots + [\zeta_{p^{n+1}}]^{p-1}$.
\end{lemma}

\begin{proof}
From the definition of the $\tilde{\theta}_n$ maps, we have that $F^n \circ \tilde{\theta}_{n+1} = \tilde{\theta}_1$.  Because $\tilde{\theta}_{n+1}$ is surjective by Lemma~\ref{theta lemma}(\ref{theta surj}), the kernel of $F^n: W_{n+1}(A) \rightarrow W_1(A)$ is exactly the image of the kernel of $\tilde{\theta}_1$ under the map $\tilde{\theta}_{n+1}$.  Hence again by Lemma~\ref{theta lemma}, the kernel of $F^n$ is the principal ideal generated by $\tilde{\theta}_{n+1}(1 + [\varepsilon] + \cdots + [\varepsilon]^{p-1}) = 1 + [\zeta_{p^{n+1}}] + \cdots + [\zeta_{p^{n+1}}]^{p-1}$.  
\end{proof}

We also include one result that is taken directly from \cite{BMS16}.

\begin{lemma}[{\cite[Corollary~3.18(i)]{BMS16}}] \label{kernel of F}
Let $A$ denote a $p$-torsion-free perfectoid ring containing a compatible system of $p$-power roots of unity.
The kernel of $F: W_{n+1}(A) \rightarrow W_n(A)$ is the principal ideal generated by 
\[
\sum_{i = 0}^{p^{n}-1} [\zeta_{p^{n+1}}]^i.
\]
\end{lemma}

Determining whether one Witt vector is divisible by another Witt vector is often difficult.  One result we will need in this direction is the following lemma.  It is required for our proof of Proposition~\ref{canonical in tate}.  Notice the similarity between the terms $\frac{[\zeta_{p^n}] - 1}{[\zeta_{p^{n+s}}] - 1}$ appearing in Lemma~\ref{Witt intersection} and the terms $z_{n+1} = \frac{[\zeta_{p^n}] - 1}{[\zeta_{p^{n+1}}] - 1}$ considered above.  The following result should be compared to \cite[Lemma~3.23]{BMS16}.

\begin{lemma} \label{Witt intersection}
Fix an integer $n \geq 1$.  Let $A_0 = \ZZ_p[\zeta_{p^{\infty}}]^{\wedge}$.  If $x \in W_n(A_0)$ is in the intersection
\[
x \in \bigcap_{s = 1}^{\infty} \frac{[\zeta_{p^n}] - 1}{[\zeta_{p^{n+s}}] - 1} W_n(A_0),
\]
then $x \in ([\zeta_{p^n}] - 1)W_n(A_0)$.
\end{lemma}

\begin{proof}
We prove this using induction on $n$.  The base case follows by considering valuations in $A_0$.  
Assume now the result has been shown for some fixed value $n-1 \geq 1$.  We will prove the result also for $n$.  Thus, assume $x \in W_{n}(A_0)$ is in the intersection.  Again considering valuations, we find that the first Witt component of $x$ must be divisible by $\zeta_{p^n} - 1$, say $x_0 = (\zeta_{p^n} - 1) y_0$.  Consider the element $x' := x - ([\zeta_{p^n}] - 1) [y_0] \in W_n(A_0)$.  If we can prove that $x'$ is divisible by $[\zeta_{p^n}] - 1$, then it will follow that $x$ is also divisible by $[\zeta_{p^n}] - 1$, and we will be done.  We have that $x'$ is in the same intersection, and also $x' = V(z)$ for some $z \in W_{n-1}(A_0)$.  

We claim that $z$ is divisible by $\frac{[\zeta_{p^{n-1}}] - 1}{[\zeta_{p^{n-1+s}}] - 1} $ in $W_{n-1}(A_0)$ for every integer $s \geq 1$.  Fix $s \geq 1$.  We know that
\[
V(z) = \frac{[\zeta_{p^n}] - 1}{[\zeta_{p^{n+s}}] - 1} \, y
\]
for some $y \in W_n(A_0)$, and considering the first Witt components, we see that $y = V(y')$ for some $y' \in W_{n-1}(A_0)$, so
\[
V(z) = \frac{[\zeta_{p^n}] - 1}{[\zeta_{p^{n+s}}] - 1} \, V(y') = V \left( F\left(\frac{[\zeta_{p^n}] - 1}{[\zeta_{p^{n+s}}] - 1} \right) y'\right) = V \left( \frac{[\zeta_{p^{n-1}}] - 1}{[\zeta_{p^{n-1+s}}] - 1}  \,y'\right). 
\]
This proves the claim that $z$ is divisible by $\frac{[\zeta_{p^{n-1}}] - 1}{[\zeta_{p^{n-1+s}}] - 1} $ in $W_{n-1}(A_0)$ for every integer $s \geq 1$.  Fix $s \geq 1$.

By our induction hypothesis, we have that $z = ([\zeta_{p^{n-1}}] - 1) w$ for some $w \in W_{n-1}(A_0)$.  Thus 
\[
x' = V(z) = V\left( ([\zeta_{p^{n-1}}] - 1) w\right) = ([\zeta_{p^n}] - 1) V(w).
\]
This proves that $x'$ is divisible by $[\zeta_{p^n}] - 1$ in $W_n(A_0)$, and as explained above, this finishes the proof.
\end{proof}

We will eventually prove that, when $A$ is a $p$-torsion-free perfectoid ring containing a compatible sequence of $p$-power roots of unity, then the $p$-adic Tate module $T_p(W_n\Omega^1_A)$ is a free $W_n(A)$-module of rank~one.  The restriction map from level~$n+1$ to level~$n$ on the Tate modules does \emph{not} correspond to the restriction map on Witt vectors.  Instead it will correspond to the map $Rz_{n+1}$ appearing in the following lemma.

\begin{lemma} \label{exact Rz}
Let $A$ denote a $p$-torsion-free perfectoid ring containing a compatible system of $p$-power roots of unity.  We have an exact sequence of $W_{n+1}(A)$-modules
\[
0 \rightarrow A \xrightarrow{V^n} W_{n+1}(A) \xrightarrow{Rz_{n+1}} W_n(A) \xrightarrow{F^n} A/p^nA \rightarrow 0,
\]
where the module structures and maps are defined as follows.  Both $A$ and $A/p^nA$ are considered as $W_{n+1}(A)$-modules via $F^n$, and $W_n(A)$ is considered as a $W_{n+1}(A)$-module via restriction.  The map $Rz_{n+1}$ denotes the composite 
\[
\psi_n\colon W_{n+1}(A) \xrightarrow{1 + [\zeta_{p^{n+1}}] + \cdots + [\zeta_{p^{n+1}}]^{p-1}} W_{n+1}(A) \stackrel{R}{\longrightarrow} W_n(A).
\]
The map $F^n: W_n(A) \rightarrow A/p^nA$ is defined using the isomorphism $W_n(A) \cong W_{n+1}(A)/V^n(A)$.
\end{lemma}

\begin{proof}
It's clear that $V^n(A)$ is in the kernel of $Rz_{n+1}$.  Conversely, if $y \in W_{n+1}(A)$ is in the kernel of $Rz_{n+1}$, then $\left( 1 + [\zeta_{p^{n+1}}] + \cdots + [\zeta_{p^{n+1}}]^{p-1}\right)\cdot y \in V^n(A)$.  Because the first $n$ ghost components of $1 + [\zeta_{p^{n+1}}] + \cdots + [\zeta_{p^{n+1}}]^{p-1}$ are not zero divisors (because $A$ is $p$-torsion-free), we see that the first $n$ ghost components of $y$ are zero, and hence $y \in V^n(A)$, proving the other inclusion.

Because $A$ is Witt-perfect, the map $W_n(A) \cong W_{n+1}(A)/V^n(A) \rightarrow A/p^nA$ induced by $F^{n}$ is surjective.  It remains to show that its kernel is equal to the image of the map $Rz_{n+1}$.  Because $F^n(1 + [\zeta_{p^{n+1}}] + \cdots + [\zeta_{p^{n+1}}]^{p-1}) = 0 \in W_1(A)$, we see that the image of $Rz_{n+1}$ is contained in the kernel of $W_n(A) \rightarrow A/p^nA$.  For the reverse inclusion, say $x \in W_{n+1}(A)$ is the lift (under restriction) of some element in the kernel of $W_n(A) \rightarrow A/p^nA$.  Write $F^n(x) = p^n a$, where $a \in A$.  Then $F^n(x - V^n(a)) = 0$, and so by Lemma~\ref{ker Fn}, we can find an element $y \in W_{n+1}(A)$ such that 
\[
\left( 1 + [\zeta_{p^{n+1}}] + \cdots + [\zeta_{p^{n+1}}]^{p-1}\right) y = x - V^n(a).
\]
Applying $R$ to both sides, we get that $R(x)$ is in the image of $Rz_{n+1}$, as required.
\end{proof}

\section{On $p$-torsion in the module of K\"ahler differentials} \label{Kahler section}

Our goal in this section is to determine the $p^n$-torsion in the module of absolute K\"ahler differentials $\Omega^1_A := \Omega^1_{A/\ZZ}$  for a $p$-torsion-free perfectoid ring $A$.  Our strategy is first to consider the cotangent complex $L_{A/\ZZ_p}$, then to consider the relative K\"ahler differentials $\Omega^1_{A/\ZZ_p}$, and then finally (and this is the easiest part) to consider the absolute K\"ahler differentials $\Omega^1_A$.  

Our starting point is a result of Bhatt-Morrow-Scholze which states that when $A$ is a perfectoid ring, the derived $p$-completion (as defined in Proposition~\ref{universal coefficient theorem}) of $L_{A/\ZZ_p}$ is quasi-isomorphic to $A[1]$.   We recall this result in Proposition~\ref{derived p-completion}. (It does not require our usual $p$-torsion-free assumption.)

\begin{proposition}[{\cite[Proposition~4.19(2)]{BMS18}}] \label{derived p-completion}
Let $A$ denote a perfectoid ring.  The derived $p$-completion of $L_{A/\ZZ_p}$ is quasi-isomorphic to $A[1]$.
\end{proposition}

An immediate corollary of this is the following.

\begin{proposition} \label{L tensor Fp}
Let $A$ denote a $p$-torsion-free perfectoid ring and let $n \geq 1$ denote an integer.  Then we have a quasi-isomorphism
\[
L_{A/\ZZ_p} \otimes_{\ZZ_p}^{L} \ZZ/p^n\ZZ \simeq (A/p^nA)[1].
\]
\end{proposition}

\begin{proof}
Let $M$ denote the derived $p$-completion of $L_{A/\ZZ_p}$.  On one hand, we have
\[
L_{A/\ZZ_p} \otimes_{\ZZ_p}^{L} \ZZ/p^n\ZZ \cong M \otimes_{\ZZ_p}^{L} \ZZ/p^n\ZZ.
\] 
On the other hand, by Proposition~\ref{derived p-completion}, we have 
\[
M \cong A[1].
\]
The claimed result now follows directly from the universal coefficient theorem (Proposition~\ref{universal coefficient theorem}).
\end{proof}

We immediately deduce the following consequence concerning $p$-power torsion in the module of relative K\"ahler differentials.

\begin{corollary} \label{H-1 short exact sequence}
Let $A$ denote a $p$-torsion free perfectoid ring and let $n \geq 1$ be an integer.  Then we have a short exact sequence of $A$-modules
\[
0 \rightarrow H^{-1}(L_{A/\ZZ_p})/p^nH^{-1}(L_{A/\ZZ_p}) \rightarrow A/p^nA \rightarrow \Omega^1_{A/\ZZ_p}[p^n] \rightarrow 0.
\]
\end{corollary}

\begin{proof}
Again from the universal coefficient theorem (Proposition~\ref{universal coefficient theorem}), we know there is a short exact sequence 
\[
0 \rightarrow H^{-1}(L_{A/\ZZ_p}) \otimes_{\ZZ_p} \ZZ/p^n\ZZ \rightarrow H^{-1}(L_{A/\ZZ_p} \otimes_{\ZZ_p}^L \ZZ/p^n\ZZ) \rightarrow \Tor^{\ZZ_p}_{1}(H^0(L_{A/\ZZ_p}), \ZZ/p^n\ZZ) \rightarrow 0.
\]
Using Proposition~\ref{L tensor Fp} to replace the middle term with $A/p^nA$, the result follows.
\end{proof}

\begin{remark}
The quasi-isomorphism $(L_{A/\ZZ_p})^{\wedge} \simeq A[1]$ in Proposition~\ref{derived p-completion} is not natural in $A$, so we must be careful regarding any desired functoriality properties of maps deduced from this quasi-isomorphism.  For example, the $A$-module homomorphism
\[
A/p^nA \rightarrow \Omega^1_{A/\ZZ_p}[p^n]
\]
appearing in Corollary~\ref{H-1 short exact sequence} is \emph{not} functorial in $A$.  However, we will see below that this map $A/p^nA \rightarrow \Omega^1_{A/\ZZ_p}[p^n]$ is an isomorphism and is determined uniquely if we moreover fix $\xi \in W(A^{\flat})$ such that $\xi$ is a generator for the kernel of Fontaine's map $\theta: W(A^{\flat}) \rightarrow A$.  Notice that, because $\xi$ is a non-zero-divisor, any two such generators differ by a unit $u \in W(A^{\flat})$, and it will follow that the two corresponding isomorphisms $A/p^nA \rightarrow \Omega^1_{A/\ZZ_p}[p^n]$ will differ by the unit $\theta(u) \in A$.
\end{remark}

The most important result in this section is the following.

\begin{theorem} \label{main relative Omega theorem}
Let $A$ denote a $p$-torsion-free perfectoid ring. 
\begin{enumerate}
\item \label{Kahler p-divisible} Multiplication by $p$ is surjective on $\Omega^1_{A/\ZZ_p}$.
\item \label{quotient of Tate} For every integer $n \geq 1$, we have an isomorphism of $A$-modules
\[
T_p(\Omega^1_{A/\ZZ_p})/p^nT_p(\Omega^1_{A/\ZZ_p}) \cong \Omega^1_{A/\ZZ_p}[p^n].
\]
\item \label{p^n torsion part} For every integer $n \geq 1$, the map from Corollary~\ref{H-1 short exact sequence}
\[
A/p^nA \rightarrow \Omega^1_{A/\ZZ_p}[p^n]
\]
is an isomorphism of $A$-modules.
\item \label{Tate part} We have an isomorphism of $A$-modules
\[
A \stackrel{\sim}{\rightarrow} T_p(\Omega^1_{A/\ZZ_p})
\]
that, when reduced modulo~$p^n$ for any integer $n \geq 1$, induces the isomorphism from Part~(\ref{p^n torsion part}).
\item \label{H^-1 p-divisible} We have an isomorphism of $A$-modules
\[
H^{-1}(L_{A/\ZZ_p})/pH^{-1}(L_{A/\ZZ_p}) \cong 0.
\]
\end{enumerate}
\end{theorem}

\begin{proof}
Part~(\ref{Kahler p-divisible}) is true more generally for any ring $A$ with the property that the $p$-power map is surjective modulo~$p$.  Namely, it suffices to prove that an element of the form $da$ is divisible by $p$, where $a \in A$ is arbitrary, and this follows after writing $a = a_0^p + pa_1$ and applying the Leibniz rule.

Part~(\ref{quotient of Tate}) is true more generally for any module~$M$ on which multiplication by $p$ is surjective.  Namely, the natural projection map
\[
T_p(M) \rightarrow M[p^n]
\]
is surjective because any $p^n$-torsion element can be extended to an element in $T_p(M)$, and so it suffices to show that the kernel of this projection consists precisely of those elements which are divisible by $p^n$.  Thus let $(m_1, m_2, \ldots) \in T_p(M)$, where $m_i \in M[p^i]$, and assume $m_n = 0$.  Consider the element $(m_{n+1}, m_{n+2}, \ldots)$; clearly $p^n(m_{n+1}, m_{n+2}, \ldots) = (m_1, m_2, \ldots)$, and moreover $(m_{n+1}, m_{n+2}, \ldots)$ is in the Tate module, because $m_n = 0$ shows that $m_{n+i} \in M[p^i]$ for all integers $i \geq 1$.

We next prove Part~(\ref{p^n torsion part}) in a special case. Let $\overline{V}$ denote a $p$-torsion-free perfectoid ring which is moreover a valuation ring with $\Frac \overline{V}$ algebraically closed.  By Proposition~\ref{Bhatt ex}, we have $H^{-1}(L_{\overline{V}/\ZZ_p}) \cong 0$, and hence from Corollary~\ref{H-1 short exact sequence}, we have a $\overline{V}$-module isomorphism $\overline{V}/p^n\overline{V} \xrightarrow{\sim} \Omega^1_{\overline{V}/\ZZ_p}[p^n]$.  This proves Part~(\ref{p^n torsion part}) in this special case.  Before proving this part in general, we consider Part~(\ref{Tate part}).

Let $A$ be an arbitrary $p$-torsion-free perfectoid ring.  By Corollary~\ref{surj to Tate} and Proposition~\ref{derived p-completion}, we have a surjective $A$-module homomorphism
\[
A \twoheadrightarrow T_p(\Omega^1_{A/\ZZ_p}),
\]
and for any integer $n \geq 1$, we have an induced map
\[
A/p^nA \twoheadrightarrow T_p(\Omega^1_{A/\ZZ_p})/p^nT_p(\Omega^1_{A/\ZZ_p}) \cong \Omega^1_{A/\ZZ_p}[p^n].
\]
Moreover, this map agrees with the map in Corollary \ref{H-1 short exact sequence}. The reason is that the universal coefficient sequence in Proposition \ref{universal coefficient theorem}(\ref{new ext part}) is natural in $G$, and hence the diagram
\[\xymatrix{ H^{-1}((L_{A/\ZZ_p})^{\wedge}) \ar[d] \ar[r] & T_p(\Omega^1_{A/\ZZ_p}) \ar[d] \\ H^{-1}((L_{A/\ZZ_p}) \otimes_{\ZZ}^L \ZZ/p^n\ZZ ) \ar[r] & \Omega^1_{A/\ZZ_p}[p^n] }\]
commutes. Here the vertical maps are induced by (derived and underived) mod $p^n$ reduction. We also note that the identification $H^{-1}((L_{A/\ZZ_p}) \otimes_{\ZZ}^L \ZZ/p^n ) \cong A/p^nA$ which is used to define the map $A/p^nA \rightarrow \Omega^1_{A/\ZZ_p}[p^n]$ in Corollary \ref{H-1 short exact sequence} is determined by the identification $H^{-1}((L_{A/\ZZ_p})^{\wedge}) \cong A$ which itself up to unit only depends on the choice the generator of $\ker \theta$ (see \cite[Proposition 4.19]{BMS18}). 
Now in the case $A = \overline{V}$, this map is the isomorphism of the previous paragraph. Thus the map $\overline{V} \rightarrow T_p(\Omega^1_{\overline{V}/\ZZ_p})$ is an inverse limit of isomorphisms, and hence is an isomorphism.  This proves Part~(\ref{Tate part}) in the case of $\overline{V}$.

We now return to the case that $A$ is an arbitrary $p$-torsion-free perfectoid ring, and consider an injective ring homomorphism $A \rightarrow \prod \overline{V}_{\alpha}$, as in Lemma~\ref{embed into product}.
We would like to construct a commutative diagram
\[
\xymatrix{
\prod \overline{V}_{\alpha} \ar[r] & \prod T_p(\Omega^1_{\overline{V}_{\alpha}/\ZZ_p})\\
A \ar[r] \ar[u] & T_p(\Omega^1_{A/\ZZ_p}) \ar[u]
}
\]
The problem here is that the map $A \rightarrow T_p(\Omega^1_{A/\ZZ_p})$ is only unique after certain choices and one needs to be careful with naturality statements. Let $f_\alpha : A \rightarrow \overline{V}_{\alpha}$ denote the given embedding to $\prod \overline{V}_{\alpha}$ composed with the projection onto the $\alpha$ factor. By Corollary \ref{surj to Tate} we have a commutative diagram
\[\xymatrix{H^{-1}((L_{\overline{V}_{\alpha}/\ZZ_p})^{\wedge}) \ar[r] & T_p(\Omega^1_{\overline{V}_{\alpha}/\ZZ_p}) \\ H^{-1}((L_{A/\ZZ_p})^{\wedge}) \ar[u]^{(f_{\alpha})_*} \ar[r] & T_p(\Omega^1_{A/\ZZ_p}) \ar[u]^{(f_{\alpha})_*}    }
\]
Further, following \cite[Proposition 4.19]{BMS18}, we also have a commutative diagram with horizontal arrows isomorphisms 
\[\xymatrix{ \ker \theta_{\overline{V}_{\alpha}} / (\ker \theta_{\overline{V}_{\alpha}})^2 \ar[r]^-{\cong} & H^{-1}((L_{\overline{V}_{\alpha}/W(\overline{V}_{\alpha}^{\flat})})^{\wedge}) \ar[r]^-{\cong} & H^{-1}((L_{\overline{V}_{\alpha}/\ZZ_p})^{\wedge}) \\ \ker \theta_{A} / (\ker \theta_{A})^2 \ar[r]^-{\cong}  \ar[u]^{(f_{\alpha})_*}  & H^{-1}((L_{A/W(A^{\flat})})^{\wedge})  \ar[u]^{(f_{\alpha})_*} \ar[r]^-{\cong} & H^{-1}((L_{A/\ZZ_p})^{\wedge})  \ar[u]^{(f_{\alpha})_*} 
}\]
Let $\xi \in W(A^{\flat})$ be a generator of the principal ideal $\ker \theta_A$. The key observation is that the image of $\xi$ under the map induced by functoriality, $W(f_{\alpha}^{\flat})(\xi),$ is a generator of $\ker \theta_{\overline{V}_{\alpha}}$ in $W(\overline{V}_{\alpha}^{\flat})$; see Lemma~\ref{xi lemma}.  The elements $\xi$ and $W(f_{\alpha}^{\flat})(\xi)$ are non-zero-divisors (see again Lemma~\ref{xi lemma}) and hence determine isomorphisms $\ker \theta_{\overline{V}_{\alpha}} / (\ker \theta_{\overline{V}_{\alpha}})^2 \cong \overline{V}_{\alpha}$ and $\ker \theta_{A} / (\ker \theta_{A})^2 \cong A$.  This compatible choice of generators allows us to fit the latter isomorphisms in a commutative diagram
\[\xymatrix{\overline{V}_{\alpha} \ar[r]^-{\cong} & \ker \theta_{\overline{V}_{\alpha}} / (\ker \theta_{\overline{V}_{\alpha}})^2  \\ A \ar[r]^-{\cong} \ar[u]^{f_{\alpha}} & \ker \theta_{A} / (\ker \theta_{A})^2 \ar[u]^{(f_{\alpha})_*}}\]
By combining the last three commutative diagrams, we obtain, for every $\alpha$, a commutative diagram
\[
\xymatrix{
\overline{V}_{\alpha} \ar[r] & T_p(\Omega^1_{\overline{V}_{\alpha}/\ZZ_p})\\
A \ar[r] \ar[u]^{f_{\alpha}} & T_p(\Omega^1_{A/\ZZ_p}) \ar[u]^{(f_{\alpha})_*}
}
\]
Here the horizontal arrows are the morphisms obtained using Corollary~\ref{surj to Tate} and Proposition~\ref{derived p-completion} (and, which we again emphasize, depend up to a unit on the choice of a generator of $\ker \theta$ in $W(A^{\flat})$). Now passing to the products we attain the desired commutative diagram 
\[
\xymatrix{
\prod \overline{V}_{\alpha} \ar[r] & \prod T_p(\Omega^1_{\overline{V}_{\alpha}/\ZZ_p})\\
A \ar[r] \ar[u] & T_p(\Omega^1_{A/\ZZ_p}) \ar[u]
}
\]
The left vertical map and the top horizontal map are injective, so the  map $A \rightarrow T_p(\Omega^1_{A/\ZZ_p})$ is injective, and we have already remarked that it is surjective.  This proves Part~(\ref{Tate part}) in general, and reducing this isomorphism modulo~$p^n$ proves Part~(\ref{p^n torsion part}) in general.  Finally, Part~(\ref{H^-1 p-divisible}) follows in general from Corollary~\ref{H-1 short exact sequence} and the fact that $A/p^nA \rightarrow \Omega^1_{A/\ZZ_p}[p^n]$ is an isomorphism.
\end{proof}

Theorem~\ref{main relative Omega theorem} indicates that, for every $p$-torsion-free perfectoid ring $A$, multiplication by $p$ on $\Omega^1_{A/\ZZ_p}$ is far from being injective.  We next give a complementary result, which indicates a situation where multiplication by $p$ on K\"ahler differentials is injective.

\begin{lemma} \label{W(k) Kahler}
Let $k$ denote a perfect ring of characteristic~$p$.  The multiplication-by-$p$ map
\[
p: \Omega^1_{W(k)/\ZZ_p} \rightarrow \Omega^1_{W(k)/\ZZ_p}
\]
is an isomorphism of $W(k)$-modules.
\end{lemma}

\begin{proof}
The proof is the same as the proof of \cite[Proposition~2.7]{D17} (which concerned absolute K\"ahler differentials), simply by replacing every occurrence of $\ZZ$ in that proof with $\ZZ_p$.  The result is also a consequence of \cite[Lemma~3.14]{BMS16}.
\end{proof}

When $A$ is a $p$-torsion-free perfectoid ring, we know from Theorem~\ref{main relative Omega theorem} that there exists $\alpha \in \Omega^1_{A/\ZZ_p}$ which freely generates $\Omega^1_{A/\ZZ_p}[p^n]$ as an $A/p^nA$-module.  The following two results concern identifying such a generator $\alpha$.  The more explicit of the two results, Corollary~\ref{p-torsion generator relative}, requires that $A$ contain a compatible system of $p$-power roots of unity.  The same condition appears in Theorem~\ref{Tate theorem}.

\begin{proposition} \label{image of 1}
Let $A$ denote a $p$-torsion-free perfectoid ring.
Let $\xi \in W(A^{\flat})$ denote a generator of $\ker \theta$.  Let $\alpha_n \in \Omega^1_{W(A^{\flat})/\ZZ_p}$ denote the unique element such that $p^n \alpha_n = d\xi$.  Then $\Omega^1_{A/\ZZ_p}[p^n]$ is generated as an $A/p^nA$-module by the image of $\alpha_n$ under the map $\theta: \Omega^1_{W(A^{\flat})/\ZZ_p} \rightarrow \Omega^1_{A/\ZZ_p}$.
\end{proposition}

\begin{proof}
Uniqueness of the element $\alpha_n$ follows from Lemma~\ref{W(k) Kahler}.  
From the ring homomorphisms $\ZZ_p \rightarrow W(A^{\flat}) \stackrel{\theta}{\rightarrow} A$, the Jacobi-Zariski sequence (\ref{Jacobi-Zariski}) associates an exact triangle in the derived category $D(A)$
\[
L_{W(A^{\flat})/\ZZ_p} \otimes_{W(A^{\flat})}^L A \rightarrow L_{A/\ZZ_p} \rightarrow L_{A/W(A^{\flat})}.
\]
From the associated long exact sequence in cohomology, using that $\theta$ is surjective with kernel generated by a non-zero divisor (so $\ker(\theta) / \ker(\theta)^2 \cong A$), there is an exact sequence 
\[
\xymatrix{
\cdots \ar[r] & A \ar[r] & \Omega^1_{W(A^{\flat})/\ZZ_p} \otimes_{W(A^{\flat})} A \ar[r] & \Omega^1_{A/\ZZ_p} \ar[r] & 0.
}
\]
We then form a double complex
\[
\xymatrix{
\cdots \ar[r] & A \ar[r] & \Omega^1_{W(A^{\flat})/\ZZ_p} \otimes_{W(A^{\flat})} A \ar[r] & \Omega^1_{A/\ZZ_p} \ar[r] & 0 \\
\cdots \ar[r] & A \ar[r] \ar^{p^n}[u]& \Omega^1_{W(A^{\flat})/\ZZ_p} \otimes_{W(A^{\flat})} A \ar[r] \ar^{-p^n}[u]& \Omega^1_{A/\ZZ_p} \ar[r] \ar^{p^n}[u] & 0.
}
\]
The horizontal maps $A \rightarrow \Omega^1_{W(A^{\flat})/\ZZ_p} \otimes_{W(A^{\flat})} A$ send $1 \mapsto d\xi \otimes 1$ (see for example \cite[Theorem~25.2]{Mat89}).
Multiplication by $p^n$ is an isomorphism on $\Omega^1_{W(A^{\flat})/\ZZ_p}$ (by Lemma~\ref{W(k) Kahler}) and hence also on $\Omega^1_{W(A^{\flat})/\ZZ_p} \otimes_{W(A^{\flat})} A$.  Considering the two spectral sequences associated to this double complex, we must have a surjective map $A/p^nA \rightarrow \Omega^1_{A/\ZZ_p}[p^n]$ given by $1 \mapsto \theta(\alpha_n)$.  This completes the proof.
\end{proof}

\begin{corollary} \label{p-torsion generator relative}
Let $A$ denote a $p$-torsion-free perfectoid ring and assume furthermore that $A$ contains a compatible system of $p$-power roots of unity $\zeta_{p^n}$.  There exists an element $\alpha \in \Omega^1_{A/\ZZ_p}$ such that $(\zeta_p - 1)\alpha = \dlog \zeta_p$, and any such element is $p$-torsion and freely generates $\Omega^1_{A/\ZZ_p}[p]$ as an $A/pA$-module. 
\end{corollary}

\begin{proof}
We first prove that there is at least one such element $\alpha$ satisfying all the listed conditions, and then prove that any element $\alpha \in \Omega^1_{A/\ZZ_p}$ satisfying $(\zeta_p - 1)\alpha = \dlog \zeta_p$ is automatically $p$-torsion and a free $A/pA$-module generator of $\Omega^1_{A/\ZZ_p}[p]$.

Let $\varepsilon \in A^{\flat}$ consist of the $p$-power compatible system $(\zeta_{p^n})_{n \geq 0}$.  One generator of $\ker \theta$ is $\xi := 1 + [\varepsilon^{1/p}] + [\varepsilon^{1/p}]^2 + \cdots + [\varepsilon^{1/p}]^{p-1}$ (see \cite[Example~3.16]{BMS16}).  Define 
\[
\alpha_1 := \sum_{m = 1}^{p-1} m [\varepsilon^{1/p}]^{m} \dlog [\varepsilon^{1/p^2}] \in \Omega^1_{W(A^{\flat})/\ZZ_p}.
\]
Note that $p\alpha_1 = d \xi$.
By Proposition~\ref{image of 1}, we know that $\theta(\alpha_1) \in \Omega^1_{A/\ZZ_p}$ freely generates $\Omega^1_{A/\ZZ_p}[p]$ as an $A/pA$-module.  On the other hand,
\[
\theta(\alpha_1) = \sum_{m = 1}^{p-1} m \zeta_p^{m} \dlog \zeta_{p^2} \in \Omega^1_{A/\ZZ_p}.
\]
Because 
\[
(\zeta_p - 1) \sum_{m = 1}^{p-1} m \zeta_p^{m}  = p,
\]
this shows that the element $\theta(\alpha_1)$ satisfies all the conditions of the statement.

Now let $\alpha \in \Omega^1_{A/\ZZ_p}$ denote an arbitrary element which satisfies $(\zeta_p - 1)\alpha = \dlog \zeta_p$.  From
\[
p \alpha = \frac{p}{\zeta_p - 1} (\zeta_p - 1) \alpha = \frac{p}{\zeta_p - 1} (\zeta_p - 1) \theta(\alpha_1) = p \theta(\alpha_1) = 0,
\]
we see that $\alpha$ is $p$-torsion.  Because $\theta(\alpha_1)$ generates the $p$-torsion, we have $\alpha = a\theta(\alpha_1)$ for some $a \in A$, and we must furthermore have 
\[
(\zeta_p - 1) a \equiv \zeta_p - 1 \bmod pA,
\]
because $(\zeta_p - 1) a \theta(\alpha_1) = (\zeta_p - 1)\theta(\alpha_1)$.  By Lemma~\ref{zeta_p congruence implies unit}, we then have that $a$ is a unit in $A/pA$, which completes the proof.
\end{proof}

In Theorem~\ref{sequence theorem}, in addition to considering $p$-torsion-free perfectoid rings, we also consider rings of integers in finite extensions of $\QQ_p$.   The following algebraic result enables us to relate such rings to the perfectoid ring $\OCp$.

\begin{proposition} \label{injective Kahler}
Let $K$ denote an algebraic extension of $\QQ_p$ and let $\overline{K}$ denote an algebraic closure of $K$.
\begin{enumerate}
\item The natural map
\[
\Omega^1_{\mathcal{O}_K/\ZZ_p} \rightarrow \Omega^1_{\mathcal{O}_{\overline{K}}/\ZZ_p}
\]
is injective.
\item For every integer $n \geq 1$, the natural map
\[
\Omega^1_{\mathcal{O}_{\overline{K}}/\ZZ_p}[p^n] \rightarrow \Omega^1_{\OCp/\ZZ_p}[p^n]
\]
is an isomorphism.
\end{enumerate}
\end{proposition}

\begin{proof}
We know $H^{-1}(L_{\mathcal{O}_{\overline{K}}/\mathcal{O}_K}) \cong 0$ by Proposition~\ref{Bhatt ex}, so the map
\[
\Omega^1_{\mathcal{O}_K/\ZZ_p} \otimes_{\mathcal{O}_K} \mathcal{O}_{\overline{K}} \hookrightarrow \Omega^1_{\mathcal{O}_{\overline{K}/\ZZ_p}}
\]
is injective by the Jacobi-Zariski sequence (\ref{Jacobi-Zariski}).  We know $\mathcal{O}_K \rightarrow \mathcal{O}_{\overline{K}}$ is faithfully flat by Proposition~\ref{alg faithfully flat}, so the map $\Omega^1_{\mathcal{O}_K/\ZZ_p} \rightarrow \Omega^1_{\mathcal{O}_K/\ZZ_p} \otimes_{\mathcal{O}_K} \mathcal{O}_{\overline{K}}$ is injective by \cite[Theorem~7.5]{Mat89}.  Thus the composite of these two maps $\Omega^1_{\mathcal{O}_K/\ZZ_p} \rightarrow \Omega^1_{\mathcal{O}_{\overline{K}}/\ZZ_p}$ is also injective.  This completes the proof of the first statement.

We now prove the second statement.  Let $R$ denote either of the rings $\mathcal{O}_{\overline{K}}$ or $\OCp$.  In both cases, $H^{-1}(L_{R/\ZZ_p}) \cong 0$, again by Proposition~\ref{Bhatt ex}.  By Proposition~\ref{universal coefficient theorem}, in both cases we have 
\[
H^{-1}\bigg(L_{R/\ZZ_p} \otimes_{\ZZ_p}^L (\ZZ/p^n\ZZ)\bigg) \cong \Omega^1_{R/\ZZ_p}[p^n].
\]
By flat base change, we have 
\[
H^{-1}\bigg(L_{R/\ZZ_p} \otimes_{\ZZ_p}^L (\ZZ/p^n\ZZ)\bigg) \cong H^{-1}(L_{(R/p^nR)/(\ZZ/p^n\ZZ)});
\]
see for example \cite[Tag~08QQ]{stacks-project}.
The second result now follows from the fact that
\[
\mathcal{O}_{\overline{K}}/p^n \mathcal{O}_{\overline{K}} \cong \OCp/p^n\OCp.
\]
\end{proof}

This concludes our treatment of the relative K\"ahler differentials, $\Omega^1_{A/\ZZ_p}$.  We are ultimately interested in the $p$-power torsion in the module of absolute K\"ahler differentials, $\Omega^1_A := \Omega^1_{A/\ZZ}$.  Although $\Omega^1_A$ is much larger than $\Omega^1_{A/\ZZ_p}$, their $p^n$-torsion modules are isomorphic.  

\begin{lemma} \label{absolute isomorphism}
Let $A$ denote a $p$-torsion-free $\ZZ_p$-algebra for which the multiplication-by-$p$ map 
\[
H^{-1}(L_{A/\ZZ_p}) \stackrel{p}{\rightarrow} H^{-1}(L_{A/\ZZ_p})
\] 
is surjective.  (For example, $A$ could be any $p$-torsion-free perfectoid ring; see Theorem~\ref{main relative Omega theorem}.)
Then for every integer $n \geq 1$, the natural map 
\[
\Omega^1_A[p^n] \rightarrow \Omega^1_{A/\ZZ_p}[p^n]
\]
is an isomorphism.  
\end{lemma}

\begin{proof}
Consider the double complex
\[
\xymatrix{
\cdots \ar[r] & H^{-1}(L_{A/\ZZ_p}) \ar[r] & \Omega^1_{\ZZ_p/\ZZ} \otimes_{\ZZ_p} A \ar[r] & \Omega^1_{A/\ZZ} \ar[r] & \Omega^1_{A/\ZZ_p} \ar[r] & 0 \\
\cdots \ar[r] & H^{-1}(L_{A/\ZZ_p}) \ar[r] \ar^{p^n}[u]& \Omega^1_{\ZZ_p/\ZZ} \otimes_{\ZZ_p} A \ar[r] \ar^{-p^n}[u]& \Omega^1_{A/\ZZ} \ar[r] \ar^{p^n}[u] & \Omega^1_{A/\ZZ_p} \ar[r] \ar^{-p^n}[u] & 0.
}
\]
Each of the four displayed vertical maps is surjective, and in fact, multiplication-by-$p$ is an isomorphism on $\Omega^1_{\ZZ_p/\ZZ} \otimes_{\ZZ_p} A$ (see for example \cite[Lemma~2.2.4]{HM03} or \cite[Proposition~2.7]{D17}). 
Because the rows of this complex are exact, the two spectral sequences associated to this double complex must both converge to zero.   Consider the spectral sequence attained by first taking cohomology along the columns. The $E_1$-page of this spectral sequence will be
\[
\xymatrix{
\cdots \ar[r] & 0 \ar[r] & 0 \ar[r] &  \Omega^1_{A/\ZZ}/p^n\Omega^1_{A/\ZZ}\ar[r] & \Omega^1_{A/\ZZ_p}/p^n\Omega^1_{A/\ZZ_p} \ar[r] & 0 \\
\cdots \ar[r] & H^{-1}(L_{A/\ZZ_p})[p^n] \ar[r] & 0 \ar[r] & \Omega^1_{A/\ZZ}[p^n] \ar[r] & \Omega^1_{A/\ZZ_p}[p^n] \ar[r]  & 0.
}
\]
Because of the two zeros at the top-left of the diagram, we deduce that $\Omega^1_{A/\ZZ}[p^n] \rightarrow \Omega^1_{A/\ZZ_p}[p^n]$ is an isomorphism, as required.  
\end{proof}

We end this section with the following corollary.  It is the key result used in the base case of our inductive proof of Theorem~\ref{Tate theorem}.

\begin{corollary} \label{base case absolute}
Assume $A$ is a $p$-torsion-free perfectoid ring containing a compatible system of $p$-power roots of unity.   There exists an element $\alpha \in \Omega^1_A$ satisfying $(\zeta_p - 1) \alpha = \dlog\zeta_p$, and given any such element, it is $p$-torsion and the map
$A/pA \stackrel{\alpha}{\rightarrow} \Omega^1_A[p]$
given by $a  \mapsto a\alpha$ is an isomorphism of $A$-modules.  Writing $A_0 := \ZZ_p[\zeta_{p^{\infty}}]^{\wedge}$, we may furthermore find one such $\alpha$ in the image of the natural map $\Omega^1_{A_0} \rightarrow \Omega^1_A$.
\end{corollary}

\begin{proof}
 Consider any $\alpha \in \Omega^1_A$ satisfying $(\zeta_p - 1)\alpha = \dlog \zeta_p$.  (For example, we can take
$\alpha = \sum_{m = 1}^{p-1} m \zeta_p^{m} \dlog \zeta_{p^2} = \frac{p}{\zeta_p - 1} \dlog \zeta_{p^2}.$)
We do not yet know that such an element is $p$-torsion, but $p\alpha$ is certainly $p$-torsion, since $p \dlog \zeta_{p} = \dlog 1 = 0.$  By Lemma~\ref{absolute isomorphism}, we know that the map $\Omega^1_A[p] \rightarrow \Omega^1_{A/\ZZ_p}[p]$ is injective.  Under this map, $p\alpha \mapsto 0 \in \Omega^1_{A/\ZZ_p}[p]$ by Corollary~\ref{p-torsion generator relative}, so $p \alpha = 0 \in \Omega^1_{A}[p]$, so $\alpha$ is $p$-torsion in $\Omega^1_A$.

The composition
\[
A/pA \stackrel{\alpha}{\rightarrow} \Omega^1_A[p] \rightarrow \Omega^1_{A/\ZZ_p}[p]
\]
is an isomorphism by Corollary~\ref{p-torsion generator relative}.  The second map in this composition is an isomorphism by Lemma~\ref{absolute isomorphism}, and hence the first map is also an isomorphism.

Our example of $\alpha$ at the beginning of the proof was in the image of $\Omega^1_{A_0} \rightarrow \Omega^1_A$, which justifies the final assertion of the corollary, i.e., we can find such an $\alpha$ which is in the image of $\Omega^1_{A_0}$.
\end{proof}

\section{Proof of Theorem~\ref{sequence theorem}} \label{proof of sequence theorem}

Most of the exactness asserted in Theorem~\ref{sequence theorem} holds in much more generality than the cases covered by Theorem~\ref{sequence theorem}, thanks to the following result of Hesselholt and Madsen.   (We state the result for the de\thinspace Rham-Witt complex, whereas their result concerns the logarithmic de\thinspace Rham-Witt complex.  Notice that the de\thinspace Rham-Witt complex is a special case of the logarithmic de\thinspace Rham-Witt complex, attained by taking the trivial monoid $M = \{1\}$ when specifying the log ring $(A,M)$.  We actually refer to the logarithmic version below in Proposition~\ref{log HM}, but it seemed less distracting to omit the logarithmic part here because we will not refer to it until much later in the paper.  The tradeoff is that Proposition~\ref{log HM} is nearly identical to Proposition~\ref{Theorem A HM}, and with the same reference to \cite{HM03}.)  

\begin{proposition}[{\cite[Proposition~3.2.6]{HM03}}] \label{Theorem A HM}
Let $A$ denote a $p$-torsion-free $\ZZ_{(p)}$-algebra.  For a fixed integer $n \geq 1$, recall the sequence~(\ref{sequence}) of $W_{n+1}(A)$-modules
\[
0 \rightarrow A \xrightarrow{(-d, p^n)}  \Omega^1_{A} \oplus A \xrightarrow{V^n \oplus dV^n} W_{n+1} \Omega^1_{A} \xrightarrow{R} W_n \Omega^1_{A} \rightarrow 0,
\]
where the module structure is the same as in Theorem~\ref{sequence theorem}.  This sequence is exact at all slots, except that possibly the segment $A \xrightarrow{(-d, p^n)}  \Omega^1_{A} \oplus A \xrightarrow{V^n \oplus dV^n} W_{n+1} \Omega^1_{A}$ is not exact.  Furthermore, we have $\im\, \big(-d, p^n\big) \subseteq \ker \big(V^n \oplus dV^n\big).$ 
\end{proposition}

As a consequence of Proposition~\ref{Theorem A HM}, to prove the exactness asserted in Theorem~\ref{sequence theorem}, we only need to show that $\ker \big(V^n \oplus dV^n\big) \subseteq \im\, \big(-d, p^n\big)$.  

The results in the previous section immediately imply Part~(\ref{perf A}) of Theorem~\ref{sequence theorem}.  With a little more effort, we will use Part~(\ref{perf A}) to deduce Part~(\ref{alg A}) of Theorem~\ref{sequence theorem}.

\begin{proof}[Proof of Theorem~\ref{sequence theorem}, Part~(\ref{perf A})]
It was shown in \cite[Section~6]{D17} that the sequence~(\ref{sequence}) is exact if $A$ is a $p$-torsion-free perfectoid ring and, moreover, there exists a $p$-torsion element $\alpha \in \Omega^1_A$ with annihilator equal to $pA$.  Theorem~\ref{main relative Omega theorem} and Lemma~\ref{absolute isomorphism} guarantee the existence of such an element $\alpha$.
\end{proof}

We now prove Part~(\ref{alg A}) of Theorem~\ref{sequence theorem}.  A version of this result for the log de\thinspace Rham-Witt complex can be found in \cite[Proposition~3.2.6 and Proof of Theorem~3.3.8]{HM03}.  

The sequence~(\ref{sequence}) is exact for $B = \OCp$, because $\OCp$ is a $p$-torsion-free perfectoid ring.   To prove exactness for $A = \mathcal{O}_K$, where $K$ is an algebraic extension of $\QQ_p$, we will use the following result.

\begin{proposition} \label{deduce exactness for A}
Let $A$ denote a $p$-torsion-free $\ZZ_{(p)}$-algebra such that there exists a $p$-torsion-free ring $B \supseteq A$ with the following properties.
\begin{enumerate}
\item The sequence~(\ref{sequence}) is exact for the ring~$B$ for all integers $n \geq 1$. 
\item For all integers $n \geq 1$, we have $A \cap p^nB = p^nA$.
\item \label{assume injective on p^n torsion} For all integers $n \geq 1$, the natural map $\Omega^1_{A}[p^n] \rightarrow \Omega^1_B[p^n]$ is injective.
\end{enumerate}
Then the sequence~(\ref{sequence}) is also exact for the ring~$A$ for all integers $n \geq 1$.
\end{proposition}

\begin{proof}
By Proposition~\ref{Theorem A HM}, it suffices to show that if $\alpha \in \Omega^1_A$ and $a \in A$ are such that $V^n(\alpha) + dV^n(a) = 0 \in W_{n+1}\Omega^1_A$, then there exists $a_0 \in A$ such that $p^n a_0 = a$ and $-da_0 = \alpha$.  By exactness of the sequence for~$B$, there exists an element $b_0 \in B$ such that $p^n b_0 = a$ and $-db_0 = \alpha \in \Omega^1_B$.  By our assumption that $A \cap p^nB = p^nA$, we deduce that there at least exists $a_1 \in A$ such that $p^n a_1 = a$; we will be finished after we show that $-da_1 = \alpha \in \Omega^1_A$.

We know 
\begin{align*}
V^n(\alpha) + dV^n(a) &= 0 \in W_{n+1}\Omega^1_A\\
\intertext{and}
V^n(-da_1) + dV^n(p^na_1) &= 0 \in W_{n+1}\Omega^1_A,\\
\intertext{and because $p^na_1 = a$, we have}
V^n(\alpha + da_1) &= 0 \in W_{n+1}\Omega^1_A.
\end{align*}
Applying $F^n$ to both sides of this last equation, we have that $\alpha + da_1 \in \Omega^1_A[p^n].$  We also have $V^n(\alpha + da_1) = 0 \in W_{n+1}\Omega^1_B$.  Because $V^n$ is injective on $\Omega^1_B$ by Proposition~\ref{V is injective}, we have that $\alpha + da_1 = 0 \in \Omega^1_B$.  Thus the element $\alpha + da_1$ is simultaneously $p^n$-torsion in $\Omega^1_A$ and is also in the kernel of $\Omega^1_A \rightarrow \Omega^1_B$.  By our assumption~(\ref{assume injective on p^n torsion}), we have that $\alpha + da_1 = 0 \in \Omega^1_A$, as required.
\end{proof}

We can now prove Part~(\ref{alg A}) of Theorem~\ref{sequence theorem}.

\begin{proof}[Proof of Theorem~\ref{sequence theorem}, Part~(\ref{alg A})]
Let $A = \mathcal{O}_K$ and $B = \OCp$.  It suffices show that the conditions of Proposition~\ref{deduce exactness for A} are satisfied for this choice of $A$ and $B$.  Because $B$ is a $p$-torsion-free perfectoid ring, we saw at the beginning of this section that the sequence~(\ref{sequence}) is exact for $B$.  We also have $\mathcal{O}_K \cap p^n \OCp = p^n \mathcal{O}_K$ (for example, by considering the valuations on $\mathcal{O}_K$ and $\OCp$).

We now verify Condition~(\ref{assume injective on p^n torsion}).  First note that $H^{-1}(L_{A/\ZZ_p}) \cong 0$ and $H^{-1}(L_{B/\ZZ_p}) \cong 0$ by Proposition~\ref{Bhatt ex}. Although our desired statement concerns absolute K\"ahler differentials, by Lemma~\ref{absolute isomorphism}, it suffices to show that for every integer $n \geq 1$, the map
\[
\Omega^1_{A/\ZZ_p}[p^n] \rightarrow \Omega^1_{B/\ZZ_p}[p^n]
\]
is injective.  This was proved in Proposition~\ref{injective Kahler}.
\end{proof}

We end this section with an application of exactness from Theorem~\ref{sequence theorem}.  (Two other applications were given in the introduction; see Proposition~\ref{dV1} and Proposition~\ref{V is injective}.)  The following result gives conditions on rings $A \rightarrow B$ under which the induced map $W_n\Omega^1_A \rightarrow W_n\Omega^1_B$ is injective.  

\begin{corollary} \label{injective dRW}
Let $A \subseteq B$ be $p$-torsion-free rings, and assume the following conditions are met. 
\begin{enumerate}
\item For all integers $n \geq 1$, we have $A \cap p^n B = p^nA$.
\item For all integers $n \geq 1$, the sequence~(\ref{sequence}) is exact for both $A$ and $B$.
\item \label{7.6.18 cond 4} The induced map $\Omega^1_A \rightarrow \Omega^1_B$ is injective.
\end{enumerate}
Then for all integers $n \geq 1$, the induced map
\[
W_n\Omega^1_A \rightarrow W_n\Omega^1_B
\]
is injective.
\end{corollary}

\begin{proof}
We prove this using induction on the level~$n$.  The base case $n = 1$ is precisely condition~(\ref{7.6.18 cond 4}).  Now assume the result holds for some fixed value of $n$.  Let $h_n: A \rightarrow \Omega^1_A \oplus A$ be given by $h_n(a) = (-da, p^n a)$, and similarly for the ring $B$.  Consider the double-complex of $W_{n+1}(A)$-modules arising from the sequences~(\ref{sequence}),
\[
\xymatrix{
0 \ar[r] & \left(\Omega^1_B \oplus B\right) / h_n(B) \ar[r] & W_{n+1} \Omega^1_B \ar[r] & W_n \Omega^1_B \ar[r] & 0 \\
0 \ar[r] & \left(\Omega^1_A \oplus A\right) / h_n(A) \ar[r] \ar[u] & W_{n+1} \Omega^1_A \ar[r] \ar[u] & W_n \Omega^1_A \ar[r] \ar[u] & 0.
}
\]
These are short exact sequences because we are assuming the sequence~(\ref{sequence}) is exact for both $A$ and $B$.
From our induction hypothesis, we know the right-hand vertical map is injective.  The horizontal sequences are exact, so by the snake lemma, it suffices to show that the left-hand vertical map is injective.  Assume $(\omega, a) \in \Omega^1_A \oplus A$ maps to an element $(-db, p^n b) \in \Omega^1_B \oplus B$.  Thus the element $a$ is in $A \cap p^nB$, so there exists $a_0 \in A$ such that $a = p^n a_0 = p^n b$.  Because $B$ is $p$-torsion-free, we know that in fact $a_0 = b$.  Thus the differentials $\omega, -da_0 \in \Omega^1_A$ are equal in $\Omega^1_B$.  Because the map $\Omega^1_A \rightarrow \Omega^1_B$ is injective by our assumption, in fact $\omega = -da_0 \in \Omega^1_A$.  Thus $(\omega, a) = 0 \in \left(\Omega^1_A \oplus A\right) / h_n(A)$, so the left-hand vertical map is injective, as required.
\end{proof}

\begin{example}
These conditions  of Corollary~\ref{injective dRW} are not so easy to verify in practice, but only because it is difficult to verify the ``base case" of level $n = 1$, i.e., that $\Omega^1_A \rightarrow \Omega^1_B$ is injective.  (That is clearly a necessary condition.)  The conditions are satisfied, for example, for the rings $A = \mathcal{O}_K$, $B = \mathcal{O}_L$, when $\QQ_p \subseteq K \subseteq L$ is a tower of algebraic extensions.  In particular, Condition~(\ref{7.6.18 cond 4}) in this case follows from Proposition~\ref{injective Kahler}.
\end{example}

\section{On $p$-power-torsion in the de\thinspace Rham-Witt complex} \label{main section}

This section contains the main results of the paper, including the proof of Theorem~\ref{Tate theorem} from the introduction.  These results are valid for a ring~$A$ satisfying the following assumptions.  

\begin{notation} \label{has roots of unity}
Let $A$ denote a $p$-torsion-free perfectoid ring which contains a sequence $(1, \zeta_p, \zeta_{p^2}, \ldots)$ of $p^n$-th roots of unity, compatible in the sense that $\zeta_{p^{n+1}}^p = \zeta_{p^{n}}$ for each integer $n \geq 1$, and where $\zeta_p$ is a primitive $p$-th root of unity in the sense that $\zeta_p$ satisfies $1 + \zeta_p + \cdots + \zeta_p^{p-1} = 0$.  We fix a choice of such elements.  Let $z_{n} := 1 + [\zeta_{p^{n}}] + \cdots + [\zeta_{p^{n}}^{p-1}] \in W_{n}(A)$. 
\end{notation}

\begin{remark}
We list here a few observations related to Notation~\ref{has roots of unity}.  First of all, for $A$ as in Notation~\ref{has roots of unity}, the choices of $\zeta_{p^n}$ for $n \geq 1$ determine a preferred ring homomorphism $A_0 \rightarrow A$, where $A_0 := \ZZ_p[\zeta_{p^{\infty}}]^{\wedge}$.  Because $A$ is $p$-torsion-free, that map $A_0 \rightarrow A$ is injective, and so $A_0$ may be considered as a subring of $A$. Our formulation of ``primitive" $p$-th root of unity is taken from \cite[Example~3.16]{BMS16}; note that it is a strictly stronger requirement than requiring $\zeta_p^p = 1$ and $\zeta_p \neq 1$.  For example, $(\zeta_p, 1) \in A_0 \times A_0$ satisfies this latter condition but not the condition of Notation~\ref{has roots of unity}.  Lastly, it is convenient to notice that $z_n = \frac{[\zeta_{p^{n-1}}] - 1}{[\zeta_{p^n}] - 1} \in W_n(A)$.  This latter formulation is well-defined because $[\zeta_{p^n}] - 1$ is a non-zero-divisor in $W_n(A)$, as can be seen by considering ghost components and using that $A$ is $p$-torsion-free.
\end{remark}

Our arguments in this section analyze the structure of $W_n\Omega^1_A[p]$ using induction on the level, $n$.  We have $W_1\Omega^1_A \cong \Omega^1_A$ (for every $\ZZ_{(p)}$-algebra $A$, see \cite[Theorem~D and the first sentence of the proof of Proposition~5.1.1]{HM04}), and so the base case of our induction will rely heavily on the results from Section~\ref{Kahler section}. To relate levels $n$ and $n+1$, we use Theorem~\ref{sequence theorem}, as in the proof of the following result.  Our original motivation for considering Theorem~\ref{sequence theorem} was to enable these sorts of arguments.  

\begin{proposition} \label{from spectral sequence}
Let $A$ denote any $p$-torsion-free perfectoid ring.
\begin{enumerate}
\item \label{ss pt1} We have an exact sequence of $W_{n+1}(A)$-modules,
\[
W_{n+1}\Omega^1_A[p] \stackrel{R}{\rightarrow} W_n\Omega^1_A[p] \rightarrow A/pA \rightarrow 0,
\]
where the $W_{n+1}(A)$-module structure on $A/pA$ is induced by $F^n$ and where the $W_{n+1}(A)$-module structure on $W_n\Omega^1_A[p]$ is induced by restriction.
\item Set $N:= \ker \left(R: W_{n+1} \Omega^1_A[p] \rightarrow W_n \Omega^1_A[p]\right)$.  We have an exact sequence of $W_{n+1}(A)$-modules,
\[
0 \rightarrow \Omega^1_A[p] \stackrel{V^n}{\rightarrow} N \rightarrow A/pA \rightarrow 0,
\]
where the $W_{n+1}(A)$-module structures on $\Omega^1_A$ and on $A/pA$ are induced by $F^n$.
\end{enumerate}
\end{proposition}

\begin{proof}
Consider the double complex of $W_{n+1}(A)$-modules
\[
\xymatrixcolsep{4pc} 
\xymatrix{
0 \ar[r] &A \ar[r]^{(-d, p^n)} & \Omega^1_A \oplus A \ar[r]^{(V^n ,dV^n)} & W_{n+1} \Omega^1_A \ar[r]^{R} & W_n \Omega^1_A \ar[r] & 0\\
0 \ar[r] &A \ar[r]^{(-d, p^n)} \ar[u]^{-p} & \Omega^1_A \oplus A \ar[u]^p \ar[r]^{(V^n ,dV^n)} & W_{n+1} \Omega^1_A \ar[u]^{-p} \ar[r]^{R} & W_n \Omega^1_A \ar[u]^p \ar[r] & 0.
}
\]
Because the rows are exact by Theorem~\ref{sequence theorem}, both spectral sequences associated to this double complex must converge to 0.  

Consider the spectral sequence with $E_1$ page attained by taking cohomology along the columns.  The $E_2$ page of this spectral sequence has the following form:
\[
\xymatrix{
0 \ar[rrd]  &A/pA \ar[rrd] & A/pA \ar[rrd] & 0  & 0  & 0\\
0  & 0  & \ker (V^n)  & N/(\im V^n)  & W_n \Omega^1_A[p]/R(W_{n+1}\Omega^1_A[p])  & 0.
}
\]
All these $d_2$ maps must be isomorphisms of $W_{n+1}(A)$-modules, and so the results follow.
\end{proof}

A consequence of the last proof is the following.  A key observation is the similarity between the exact sequence appearing in Proposition~\ref{zeta cor} and the exact sequence from Lemma~\ref{exact Rz}.

\begin{proposition} \label{zeta cor} Let $A$ denote a $p$-torsion-free perfectoid ring. 
For every integer $n \geq 1$, we have an exact sequence of $W_{n+1}(A)$-modules
\[
0 \rightarrow T_p(\Omega^1_A) \xrightarrow{V^n} T_p(W_{n+1}\Omega^1_{A}) \xrightarrow{R} T_p(W_n\Omega^1_{A}) \xrightarrow{g} A/p^nA \to 0,
\]
where the $W_{n+1}(A)$-module structure on $T_p(W_n\Omega^1_{A})$ is induced by restriction, and where the $W_{n+1}(A)$-module structures on $T_p(\Omega^1_A)$ and $A/p^nA$ are induced by $F^n$.  (We do not specify the map $g$.)
\end{proposition}

\begin{proof} Using verbatim the same spectral sequence argument as in the proof of Proposition~\ref{from spectral sequence}, for any $r \geq n$ we obtain exact sequences for higher torsion
\[W_{n+1}\Omega^1_A[p^r] \stackrel{R}{\rightarrow} W_n\Omega^1_A[p^r] \rightarrow A/p^n A \rightarrow 0\]
and 
\[
0 \rightarrow \Omega^1_A[p^r] \stackrel{V^n}{\rightarrow} N^r \rightarrow A/p^nA \rightarrow 0,
\]
where $N^r:= \ker \left(R: W_{n+1} \Omega^1_A[p^r] \rightarrow W_n \Omega^1_A[p^r]\right)$. (Here and throughout this proof, the exact sequences are all exact sequences of $W_{n+1}(A)$-modules, with the module structures as in the statement of this proposition.)  The double complex used to obtain the exact sequences for $p^{r+1}$ maps to the double complex for $p^r$ via the identity on the top row and $p$ on the bottom row. This induces a map of spectral sequences and hence we get commutative diagrams with exact rows:
\[ \xymatrix{ 0 \ar[r] & N^r \ar[r] & W_{n+1}\Omega^1_A[p^r]  \ar[r]^R & W_n\Omega^1_A[p^r] \ar[r] & A/p^nA \ar[r] & 0 \\ 0 \ar[r] & N^{r+1} \ar[u]^p \ar[r] & W_{n+1}\Omega^1_A[p^{r+1}]  \ar[u]^{p} \ar[r]^R & W_n\Omega^1_A[p^{r+1}] \ar[u]^{p} \ar[r] & A/p^{n} A \ar@{=}[u] \ar[r] & 0 } \]
and 
\[\xymatrix{ 0 \ar[r] & \Omega^1_A[p^r] \ar[r]^{V^n} & N^r \ar[r] & A/p^nA \ar[r] & 0 \\ 0 \ar[r] & \Omega^1_A[p^{r+1}]  \ar[u]^p  \ar[r]^{V^n} & N^{r+1} \ar[u]^p \ar[r] & A/p^{n}A \ar[r] \ar[u]^p & 0.} \]
From this last diagram we obtain an isomorphism
\[
T_p(\Omega^1_A) \xrightarrow{V^n} \varprojlim_r N^r.
\]
Furthermore, the maps $\Omega^1_A[p^{r+1}] \xrightarrow{p} \Omega^1_A[p^r]$ are surjective and the tower
\[
\cdots \xrightarrow{p} A/p^nA \xrightarrow{p} A/p^nA 
\]
satisfies the Mittag-Leffler condition, so using the $\varprojlim$-$\varprojlim^1$ exact sequence (see for example \cite[Section~3.5 and Proposition~3.5.7]{Wei94}), we deduce that $\varprojlim^1 N^r = 0$.  

Let $I^r$ denote the image 
\[\im \left(R: W_{n+1} \Omega^1_A[p^r] \rightarrow W_n \Omega^1_A[p^r]\right).\]
Because multiplication by $p$ maps $W_{n+1} \Omega^1_A[p^{r+1}]$ surjectively onto $W_{n+1} \Omega^1_A[p^r]$ (using Proposition~\ref{mult by p is surjective}), it follows immediately that multiplication by $p$ maps $I^{r+1}$ surjectively onto $I^r$.  We deduce for later that $\varprojlim^1_{r} I^r=0$.  These $\varprojlim^1$ computations will be used below.

We can split the first commutative diagram above into two commutative diagrams with exact rows:
\[ \xymatrix{ 0 \ar[r] & N^r \ar[r] & W_{n+1}\Omega^1_A[p^r]  \ar[r]^-R & I^r \ar[r] & 0 \\ 0 \ar[r] & N^{r+1} \ar[u]^p \ar[r] & W_{n+1}\Omega^1_A[p^{r+1}]  \ar[u]^{p} \ar[r]^-R & I^{r+1} \ar[u]^p \ar[r] & 0 } \]
and
\[ \xymatrix{ 0 \ar[r] & I^r  \ar[r] & W_n\Omega^1_A[p^r] \ar[r] & A/p^nA \ar[r] & 0 \\ 0 \ar[r] & I^{r+1} \ar[u]^p \ar[r] & W_n\Omega^1_A[p^{r+1}] \ar[u]^{p} \ar[r] & A/p^{n} A \ar@{=}[u] \ar[r] & 0. } \]
Altogether we obtain two short exact sequences of inverse systems. Taking into account that $\varprojlim^1_{r} N^r=0$ and $\varprojlim^1_{r} I^r=0$ as remarked above, again using the $\varprojlim$-$\varprojlim^1$ exact sequence, we obtain exact sequences
\[\xymatrix{0 \ar[r] &   \varprojlim_r N^r  \ar[r] & T_p W_{n+1} \Omega^1_A \ar[r]^-R  & \varprojlim_r I^r \ar[r] & 0 }\]
and 
\[\xymatrix{0 \ar[r] &   \varprojlim_r I^r  \ar[r] & T_p W_n \Omega^1_A \ar[r]^-g  & A/p^n A \ar[r] & 0. }\]
Now splicing these exact sequences together and using the isomorphism $T_p(\Omega^1_A) \xrightarrow{V^n} \varprojlim_r N^r$, we obtain the desired exact sequence
\[\xymatrix{0 \ar[r] &  T_p(\Omega^1_A)  \ar[r]^{V^n} & T_p W_{n+1} \Omega^1_A \ar[r]^R &  T_p W_n \Omega^1_A \ar[r]^-g  & A/p^n A \ar[r] & 0.} \]

\end{proof}

\begin{lemma} \label{kernel R versus kernel F}
Let $A$ denote any $p$-torsion-free $\ZZ_{(p)}$-algebra for which the sequence~(\ref{sequence}) is exact and let $n \geq 1$ be an integer.  We have an inclusion $W_{n+1}\Omega^1_A[p] \cap \ker R \subseteq \ker F.$  In particular, if $A$ is a $p$-torsion-free perfectoid ring and if $x,y \in T_p(W_{n+1}\Omega^1_A)$ are such that $R(x) \equiv R(y) \bmod pT_p(W_{n}\Omega^1_A)$, then $F(x) \equiv F(y) \bmod pT_p(W_{n}\Omega^1_A)$.
\end{lemma}

\begin{proof}
Let $x \in W_{n+1}\Omega^1_A[p] \cap \ker R$.  The fact that $x \in \ker R$ implies, using exactness of (\ref{sequence}), that there exist $\alpha \in \Omega^1_A$ and $a \in A$ such that
\[
x = V^n(\alpha) + dV^n(a).
\]
To prove $W_{n+1}\Omega^1_A[p] \cap \ker R \subseteq \ker F$, we wish to show that if $px = 0$, then $F(x) = 0$.  Using standard identities within the de\thinspace Rham-Witt complex, we wish to show that if $px = 0$, then $V^{n-1}(p\alpha) + dV^{n-1}(a)= 0$.  We compute 
\[
0 = px = V^n(p\alpha) + dV^n(pa) = V^n(p\alpha) + V\left(dV^{n-1}(a)\right) = V\bigg(V^{n-1}(p\alpha) + dV^{n-1}(a)\bigg).
\]
Because $V$ is injective by Proposition~\ref{V is injective}, 
\[
0 = V^{n-1}(p\alpha) + dV^{n-1}(a).
\]
This proves the first assertion.  The second assertion about Tate modules follows directly because, as in the proof of Theorem~\ref{main relative Omega theorem}, we have $T_p(W_{n+1}\Omega^1_A)/pT_p(W_{n+1}\Omega^1_A) \cong W_{n+1}\Omega^1_A[p]$.
\end{proof}

We will soon prove that $T_p(W_n\Omega^1_A)$ is a free $W_n(A)$-module of rank one for rings $A$ as in Notation~\ref{has roots of unity}.  It turns out there is essentially no difference between proving this and proving that the $p$-torsion $W_n\Omega^1_A[p]$ is a free $W_n(A)/pW_n(A)$-module of rank one.  This is the content of the following lemma.

\begin{lemma} \label{p-torsion vs Tate module}
Let $A$ be a $p$-torsion-free perfectoid ring, and let $n \geq 1$ be an integer.  If $\alpha \in T_p(W_n\Omega^1_A)$ freely generates $T_p(W_n\Omega^1_A)$ as a $W_n(A)$-module, then the projection of $\alpha$ to $W_n\Omega^1_A[p^r]$, written $\alpha^{(r)}$, freely generates $W_n\Omega^1_A[p^r]$ as a $W_n(A)/p^rW_n(A)$-module.  Conversely, if $\alpha \in T_p(W_n\Omega^1_A)$ is such that $\alpha^{(1)}$ freely generates $W_n\Omega^1_A[p]$ as a $W_n(A)/pW_n(A)$-module, then $\alpha$ freely generates $T_p(W_n\Omega^1_A)$ as a $W_n(A)$-module.
\end{lemma}

\begin{proof}
The first statement follows from the short exact sequence 
\[
0 \rightarrow T_p(W_n\Omega^1_A) \xrightarrow{p^r} T_p(W_n\Omega^1_A) \xrightarrow{} W_n\Omega^1_A[p^r] \rightarrow 0.
\]
(See the proof of Theorem~\ref{main relative Omega theorem}.)  The second statement follows by showing, level-by-level, that multiplication by $\alpha^{(r)}$ induces an isomorphism of $W_n(A)/p^rW_n(A)$-modules $W_n(A)/p^rW_n(A) \rightarrow W_n\Omega^1_A[p^r]$, using induction and the following diagram: 
\[
\xymatrix{
0 \ar[r] & W_n\Omega^1_A[p]\ar[r] & W_n\Omega^1_A[p^r] \ar[r]^{p} & W_n\Omega^1_A[p^{r-1}]\ar[r] & 0 \\
0 \ar[r] & W_n(A)/pW_n(A) \ar[u]^{\alpha^{(1)}} \ar[r]^{p^{r-1}} & W_n(A)/p^rW_n(A) \ar[r] \ar[u]^{\alpha^{(r)}}  & W_n(A)/p^{r-1}W_n(A) \ar[u]^{\alpha^{(r-1)}} \ar[r] & 0.
}
\]
(Notice that $W_n(A)$ is $p$-torsion-free because $A$ is $p$-torsion-free, and so the lower-left map given by multiplication by $p^{r-1}$ is indeed injective.)
Using that $W_n(A)$ is $p$-adically complete (Lemma~\ref{Witt complete}), it follows that multiplication by $\alpha$ induces an isomorphism of $W_n(A)$-modules
\[
\varprojlim_r W_n(A)/p^r W_n(A) \xrightarrow{\alpha} T_p(W_n\Omega^1_A).
\]
This completes the proof of the second statement.
\end{proof}

\begin{theorem} \label{generator of Tate}
Let $A$ be a ring as in Notation~\ref{has roots of unity}.  For every integer $n \geq 1$, the $p$-adic Tate module $T_p(W_n\Omega^1_A)$ is a free $W_n(A)$-module of rank~one.  Furthermore, there exists $\alpha_n \in T_p(W_n\Omega^1_A)$ which is a generator and such that the projection of $\alpha_n$ to $W_n\Omega^1_A[p]$, written $\alpha_n^{(1)}$, satisfies 
\[
([\zeta_{p^n}] - 1) \alpha_n^{(1)} = \dlog [\zeta_p] \in W_n\Omega^1_A[p].
\]
\end{theorem}

\begin{proof}
We prove a stronger result using induction on $n$.    Let $A_0 = \ZZ_p[\zeta_{p^{\infty}}]^{\wedge}$.
\begin{itemize}
\item For every integer $n \geq 1$, there exists $\alpha_{0,n} \in T_p(W_n\Omega^1_{A_0})$ such that
\[
([\zeta_{p^n}] - 1) \alpha_{0,n}^{(1)} = \dlog [\zeta_p] \in W_n\Omega^1_{A_0}[p],
\]
and such that, under the map induced by functoriality, $T_p(W_n\Omega^1_{A_0}) \rightarrow T_p(W_n\Omega^1_{A})$,  the image of $\alpha_{0,n}$ freely generates $T_p(W_n\Omega^1_A)$ as a $W_n(A)$-module.  We write $\alpha_n$ for this generator of $T_p(W_n\Omega^1_A)$.
\end{itemize}

We prove the base case.  Because multiplication by $p$ is surjective on $\Omega^1_{A_0}$, given any element $x \in \Omega^1_{A_0}[p]$, there exists $\alpha_{0,1} \in T_p(\Omega^1_{A_0})$ with $\alpha_{0,1}^{(1)} = x$.  Thus, by Corollary~\ref{base case absolute}, there exists $\alpha_{0,1} \in T_p(\Omega^1_{A_0})$ satifying all the listed properties, with the exception that we do not yet know its image $\alpha_1 \in T_p(\Omega^1_A)$ is a generator for $T_p(\Omega^1_A)$ as an $A$-module.  As yet, we only know (again by Corollary~\ref{base case absolute}), that $\alpha_1^{(1)}$ is a generator for the $p$-torsion.  But then $\alpha_1$ is a generator for the Tate module $T_p(\Omega^1_A)$ by Lemma~\ref{p-torsion vs Tate module}.

Now inductively assume the result has been proved for some fixed value of $n$, and let $\alpha_{0,n}$ and $\alpha_n$ denote the corresponding elements.  Considering the $W_{n+1}(A_0)$-module structure on the terms in the exact sequence from Proposition~\ref{zeta cor}, we see that $g(R(z_{n+1})\alpha_{0,n}) = F^n(z_{n+1}) g(\alpha_{0,n}) = 0 \in A_0/p^nA_0$, so there must exist $\alpha_{0,n+1}^{\prime} \in T_p(W_{n+1}\Omega^1_{A_0})$ such that $R(\alpha_{0,n+1}^{\prime}) = R(z_{n+1}) \alpha_{0,n}$.  Fix one such element $\alpha_{0,n+1}^{\prime}$ and its image $\alpha_{n+1}^{\prime} \in T_p(W_{n+1}\Omega^1_A)$.  We have the following commutative diagram: 
\[
\xymatrix{T_p(W_{n+1}\Omega^1_{A_0}) \ar[r] \ar_{R}[d]  & T_p(W_{n+1}\Omega^1_{A}) \ar_{R}[d] \\ T_p(W_{n}\Omega^1_{A_0})  \ar[r] & T_p(W_{n}\Omega^1_{A}) } \hspace*{1in} 
\xymatrix{\alpha_{0,n+1}^{\prime} \ar@{|->}[r] \ar@{|->}[d] & \alpha_{n+1}^{\prime} \ar@{|->}[d] \\ R(z_{n+1})\alpha_{0,n} \ar@{|->}[r]  & R(z_{n+1})\alpha_{n}  }
\]
In particular, $R(\alpha_{n+1}^{\prime}) = R(z_{n+1}) \alpha_{n}$.

By patching together the exact sequences from Proposition~\ref{zeta cor} and Lemma~\ref{exact Rz}, we form a commutative diagram in which the rows are exact sequences of $W_{n+1}(A)$-modules:
\begin{equation} \label{Tate diagram}
\begin{gathered}
\xymatrix{
0 \ar[r] &  T_p(\Omega^1_A)  \ar[r]^{V^n} & T_p W_{n+1} \Omega^1_A \ar[r]^R &  T_p W_n \Omega^1_A \ar[r]^-g  & A/p^n A \ar[r] & 0\\
0 \ar[r] &  A \ar[u]^{F^n(\alpha_{n+1}^{\prime})} \ar[r]^{V^n} & W_{n+1}(A) \ar[u]^{\alpha_{n+1}^{\prime}} \ar[r]^{Rz_{n+1}} &  W_n(A) \ar[u]^{\alpha_n} \ar[r]^{F^n}  & A/p^n A \ar@{-->}[u]^{?}  \ar[r] & 0
}
\end{gathered}
\end{equation}
By our induction hypothesis, the vertical map given by $\alpha_n$ is an isomorphism.  There exists a unique vertical map as in the dashed arrow.  By a diagram chase, that vertical map is surjective, and viewing it as a map of $A/p^nA$-modules, we see that it is also an isomorphism.

Notice that $([\zeta_{p^{n+1}}] - 1)z_{n+1}  = [\zeta_{p^n}] - 1$, and so 
\[
R(\dlog[\zeta_{p}]) = R(([\zeta_{p^{n+1}}] - 1) \alpha_{0,n+1}^{\prime\,(1)}) \in W_n\Omega^1_{A_0}[p].
\] 
Then by Lemma~\ref{kernel R versus kernel F}, we have
\[
F(\dlog [\zeta_p]) = F\left(([\zeta_{p^{n+1}}] - 1) \alpha_{0,n+1}^{\prime\,(1)}\right) \in W_n\Omega^1_{A_0}[p],
\] 
and therefore, applying $F^{n-1}$ to both sides, we have
\[
\dlog \zeta_p = (\zeta_p - 1) \,F^n\left( \alpha_{0,n+1}^{\prime\,(1)}\right) \in \Omega^1_{A_0}[p],
\]
and by functoriality, we have 
\[
\dlog \zeta_p = (\zeta_p - 1) \,F^n\left( \alpha_{n+1}^{\prime\,(1)}\right) \in \Omega^1_{A}[p].
\]
Again using Corollary~\ref{base case absolute}, we know that $F^n(\alpha_{n+1}^{\prime\,(1)})$ freely generates $\Omega^1_A[p]$ as an $A/pA$-module, and hence $F^n(\alpha_{n+1}^{\prime})$ freely generates $T_p(\Omega^1_A)$ as an $A$-module by Lemma~\ref{p-torsion vs Tate module}.  
Thus the left-hand vertical map in the diagram (\ref{Tate diagram}) is also an isomorphism.  It follows by the five lemma that the remaining vertical map, given by multiplication by $\alpha_{n+1}^{\prime}$, is also an isomorphism.  This shows that $T_p(W_{n+1}\Omega^1_A)$ is a free $W_{n+1}(A)$-module of rank~one.  We must still construct the element $\alpha_{0,n+1}$ from $\alpha_{0,n+1}^{\prime}$.

We recall that we have elements $\alpha_{0,n}$ and $\alpha_{0,n+1}^{\prime}$ satisfying
\begin{align*}
([\zeta_{p^n}] - 1) \alpha_{0,n}^{(1)} &= \dlog [\zeta_p] \in W_{n}\Omega^1_{A_0}[p].\\
\intertext{and}
R(\alpha_{0,n+1}^{\prime}) &= R(z_{n+1}) \alpha_{0,n}.\\
\intertext{We claim that there exists a unit $u \in W_{n+1}(A_0)/pW_{n+1}(A_0)$ such that}
([\zeta_{p^{n+1}}] - 1) u \alpha_{0,n+1}^{\prime\,(1)}  &= \dlog [\zeta_p] \in W_{n+1}\Omega^1_{A_0}[p].
\end{align*}
Because $\dlog [\zeta_p]$ is $p$-torsion and because $\alpha_{0,n+1}^{\prime\,(1)}$ freely generates the $p$-torsion in $W_{n+1}\Omega^1_{A_0}$ by Lemma~\ref{p-torsion vs Tate module}, we know there exists a (unique) element $x \in W_{n+1}(A_0)/pW_{n+1}(A_0)$ such that $\dlog [\zeta_p] = x\alpha_{0,n+1}^{\prime\,(1)} $.   Let 
\[
J := \ker \left(W_{n+1}(A_0)/pW_{n+1}(A_0) \stackrel{Rz_{n+1}}{\longrightarrow} W_n(A_0)/pW_n(A_0)\right).
\] 
Note that $J \subseteq W_{n+1}(A_0)/pW_{n+1}(A_0)$ is an ideal.  
Because 
\[
R(x \alpha_{0,n+1}^{\prime\,(1)}) = R(([\zeta_{p^{n+1}}] - 1) \alpha_{0,n+1}^{\prime\,(1)}) = \dlog [\zeta_{p}] \in W_n\Omega^1_A[p],
\] 
we have 
\[
x \equiv ([\zeta_{p^{n+1}}] - 1) \bmod J.
\] 
These elements $x$ and $[\zeta_{p^{n+1}}]-1$ are not themselves in the ideal $J$, because $\dlog [\zeta_p] \neq 0 \in W_{n}\Omega^1_{A_0}[p]$.  Thus, by Lemma~\ref{unit exists} (which applies because we are considering the subring $A_0 \subseteq A$), there exists a unit $u \in W_{n+1}(A_0)/pW_{n+1}(A_0)$ such that $x = u ([\zeta_{p^{n+1}}] - 1)$.   Write $u$ also for a unit $u \in W_{n+1}(A_0)$ which lifts $u \in W_{n+1}(A_0)/pW_{n+1}(A_0)$; any lift to $W_{n+1}(A_0)$ is in fact a unit by Lemma~\ref{units in Witt}.  Let $\alpha_{0,n+1} := u\alpha_{0,n+1}^{\prime} \in T_p(W_{n+1}\Omega^1_{A_0})$ and write $\alpha_{n+1} := u\alpha_{n+1}^{\prime} \in T_p(W_{n+1}\Omega^1_{A})$ for its image.  Because $\alpha_{n+1}^{\prime}$ freely generates $T_p(W_{n+1}\Omega^1_{A})$, the same is true for this unit multiple $\alpha_{n+1}$, and by construction it is the image of an element $\alpha_{0,n+1}$ satisfying $([\zeta_{p^{n+1}}] - 1) \alpha_{0,n+1}^{(1)} = \dlog [\zeta_p]$.  This completes the induction.
\end{proof}

We proved in Theorem~\ref{generator of Tate} the existence of a generator of $T_p(W_n\Omega^1_A)$ as a free rank~one $W_n(A)$-module.  The next result gives a condition for identifying generators.

\begin{corollary} \label{identify generator}
Let $\alpha_n \in T_p(W_{n}\Omega^1_{A})$ be any element satisfying $([\zeta_{p^{n}}] - 1) \alpha_n^{(1)} = \dlog [\zeta_p] \in W_n\Omega^1_A[p]$.  Then $\alpha_n$ freely generates $T_p(W_n\Omega^1_A)$ as a $W_n(A)$-module.
\end{corollary}

\begin{proof}
Fix $\alpha \in T_p(W_{n}\Omega^1_{A})$ such that $([\zeta_{p^{n}}] - 1) \alpha^{(1)} = \dlog [\zeta_p] \in W_n\Omega^1_A[p]$ and such that $\alpha$ freely generates $T_p(W_n\Omega^1_A)$ as a $W_n(A)$-module; such an element $\alpha$ exists by Theorem~\ref{generator of Tate}.  Now let $\alpha_n$ be as in the statement of the corollary.  There exists a unique $y \in W_n(A)$ such that $\alpha_n = y \alpha$, and we know that 
\[
([\zeta_{p^{n}}] - 1) y \alpha \equiv ([\zeta_{p^{n}}] - 1) \alpha \bmod p T_p(W_n\Omega^1_A).
\]
Therefore
\[
([\zeta_{p^{n}}] - 1) y \equiv ([\zeta_{p^{n}}] - 1)  \bmod p W_n(A).
\]
So $y$ projects to a unit  in $W_{n}(A)/pW_{n}(A)$ by Lemma~\ref{zeta_p congruence implies unit}.  Thus $y$ is a unit in $W_n(A)$ by Lemma~\ref{units in Witt}.  Because $\alpha$ was a generator, then $\alpha_n = y \alpha$ is also a generator.  This completes the proof.
\end{proof}

The elements produced in Theorem~\ref{generator of Tate} are not canonical. We next describe canonical generators.  Our description of these generators is modeled after Hesselholt's \cite[Theorem~B]{Hes06}.
For rings $A$ as in Notation~\ref{has roots of unity}, there is an obvious element in $T_p(W_n\Omega^1_A)$; namely, for the component in $W_n\Omega^1_A[p^r]$, we take the element $\dlog [\zeta_{p^r}]$.  We refer to the corresponding element as $\dlog [\zeta_p^{(\infty)}] \in T_p(W_n\Omega^1_A)$.  Notice that our way of writing this element does not indicate the level~$n$.  These elements, for various $n$, are all compatible under the restriction map (as well as the Frobenius map).   This compatibility under restriction is the key observation to identifying a canonical generator of $T_p(W_n\Omega^1_A)$.  This generator $\alpha_n^{(\infty)}$ will satisfy
\[
([\zeta_{p^n}] - 1) \alpha_n^{(\infty)} = \dlog[\zeta_p^{(\infty)}] \in T_p(W_n\Omega^1_A);
\] 
notice that this is similar to the condition of Theorem~\ref{generator of Tate}, but the condition of Theorem~\ref{generator of Tate} was only a condition modulo~$p$.
Proving the existence of such an element $\alpha_n^{(\infty)}$ seems rather delicate and will require several preliminary results.

\begin{lemma} \label{div Rz}
Fix an integer $n \geq 1$.  Let $A$ and $z_{n+1}$ be as in Notation~\ref{has roots of unity}.
  We have that $\beta \in T_p(W_n\Omega^1_{A})$ is in the image of the restriction map
\[
R: T_p(W_{n+1}\Omega^1_{A}) \rightarrow T_p(W_n\Omega^1_{A}),
\]
if and only if there exists $\alpha \in T_p(W_n\Omega^1_{A})$ such that $R(z_{n+1})\alpha = \beta$.
\end{lemma}

\begin{proof}
The ``if" direction was shown in the proof of Theorem~\ref{generator of Tate}; it was used to produce the element $\alpha_{0,n+1}^{\prime}$.  We now prove the ``only if" direction.  We know $T_p(W_{n}\Omega^1_{A}) \cong W_{n}(A)$ as $W_{n}({A})$-modules or, equivalently, as $W_{n+1}({A})$-modules via restriction.  Fix one such isomorphism, and let $\alpha_n \in T_p(W_n\Omega^1_{A})$ correspond to $1 \in W_n({A})$.  Let $g: T_p(W_n\Omega^1_A) \rightarrow A/p^nA$ be the $W_{n+1}(A)$-module map indicated in Proposition~\ref{zeta cor}.  Because $g$ is surjective, we know $g(\alpha_n)$ is a unit in $A/p^nA$.

Return now to our element $\beta \in T_p(W_n\Omega^1_{A})$ which is in the image of restriction.  We have $\beta = R(x) \alpha_n$ for some $x \in W_{n+1}({A})$.  Because $\beta$ is in the image of restriction, by exactness of the sequence from Proposition~\ref{zeta cor}, we have $g(\beta) = 0 \in {A/p^nA}$.  Thus, considering the $W_{n+1}(A)$-module structure, we find
\[
g\left(R(x) \alpha_n\right) = F^n(x) g(\alpha_n) = 0 \in A/p^n A.
\]
We have already observed that $g(\alpha_n)$ is a unit in $A/p^nA$, so $F^n(x) = 0 \in A/p^n A$.  By Lemma~\ref{exact Rz}, this means precisely that $R(x) = R(z_{n+1}y)$ for some $y \in W_{n+1}(A)$.  For our desired element $\alpha$, we may then take $R(y) \alpha_n$.
\end{proof}

\begin{lemma} \label{R^s}
Let $A$ denote a ring as in Notation~\ref{has roots of unity}.
Fix integers $n,s \geq 1$.  If $\beta \in T_p(W_n\Omega^1_{A})$ is in the image of the restriction map
\[
R^s: T_p(W_{n+s}\Omega^1_{A}) \rightarrow T_p(W_n\Omega^1_{A}),
\]
then there exists $\alpha \in T_p(W_n\Omega^1_{A})$ such that 
\[
R^s(z_{n+s})\cdots R^2(z_{n+2})R(z_{n+1})\alpha = \beta.
\]
\end{lemma}

\begin{proof}
We prove this result using induction on $s$. 
The base case $s = 1$ is precisely Lemma~\ref{div Rz}.  Now assume the result has been shown for some fixed value of $s$, i.e., assume the claimed result holds for that value of $s$ and for all values of $n \geq 1$.  We will prove that the claimed result holds also for $s + 1$ and all values of $n \geq 1$.

We know that for all integers $n \geq 1$, there exists $\alpha_n \in T_p(W_n\Omega^1_{A})$ such that $\alpha_n$ freely generates $T_p(W_n\Omega^1_{A})$ as a $W_n(A)$-module.  It clearly suffices to show 
\[
R^{s+1}(\alpha_{n+s+1}) = R^{s+1}(z_{n+s+1})\cdots R^2(z_{n+2})R(z_{n+1})\alpha
\] 
for some $\alpha \in T_p(W_n\Omega^1_{A})$.  By the induction hypothesis (applied to the values $s$ and $n+1$), we know
\[
R^{s}(\alpha_{n+s+1}) = R^{s}(z_{n+s+1})\cdots R(z_{n+2})\alpha'
\] 
for some $\alpha' \in T_p(W_{n+1}\Omega^1_{A})$.  We also know 
\[
R(\alpha') = R(z_{n+1}) \alpha
\]
for some $\alpha \in T_p(W_n\Omega^1_{A})$ by Lemma~\ref{div Rz}.  The desired result follows immediately from the two previous displayed equations.
\end{proof}

\begin{proposition} \label{divided log}
Let $A$ denote a ring as in Notation~\ref{has roots of unity}.
Fix an integer $n \geq 1$.  There exists $\alpha \in T_p(W_n\Omega^1_{A})$ such that 
\[
([\zeta_{p^n}] - 1) \alpha = \dlog [\zeta_{p}^{(\infty)}].
\]
\end{proposition}

\begin{proof}
By functoriality, it suffices to prove this result for the ring $A_0 = \ZZ_p[\zeta_{p^{\infty}}]^{\wedge}$.
Choose $\alpha_n \in T_p(W_n\Omega^1_{A_0})$ which freely generates $T_p(W_n\Omega^1_{A_0})$ as a $W_n(A_0)$-module.  Then there exists some unique $x \in W_n(A_0)$ such that 
\[
x \alpha_n = \dlog [\zeta_{p}^{(\infty)}].
\]
By Lemma~\ref{Witt intersection}, it suffices to show, for every integer $s \geq 1$, that $x$ is a multiple of $\frac{[\zeta_{p^n}] - 1}{[\zeta_{p^{n+s}}] - 1}$.  Notice that for every integer $s \geq 1$, we have 
\[
R^s \left( \dlog [\zeta_{p}^{(\infty)}] \right) = \dlog [\zeta_{p}^{(\infty)}],
\]
where the $\dlog [\zeta_{p}^{(\infty)}]$ on the left is considered as an element in $T_p(W_{n+s}\Omega^1_{A_0})$, and where the $\dlog [\zeta_{p}^{(\infty)}]$ on the right is considered as an element of $T_p(W_n\Omega^1_{A_0})$.  Thus our desired result follows immediately from Lemma~\ref{R^s} and the fact that
\[
R^s(z_{n+s})\cdots R^2(z_{n+2})R(z_{n+1}) = \frac{[\zeta_{p^n}] - 1}{[\zeta_{p^{n+s}}] - 1} \in W_n(A_0).
\]
This completes the proof.
\end{proof}

The following result is one of the most important results of this section.  Its proof contains no new ideas; it simply requires assembling the results attained earlier.  This result should be compared to \cite[Theorem~B]{Hes06}; notice that Hesselholt's proof uses techniques from topology (see especially \cite[Section~3]{Hes06}).

\begin{theorem} \label{canonical in tate}
Let $A$ denote a ring as in Notation~\ref{has roots of unity}.  
For each integer $n \geq 1$, there exists a unique element $\alpha_n^{(\infty)} \in T_p(W_n\Omega^1_A)$ with $([\zeta_{p^n}] - 1) \alpha_n^{(\infty)} = \dlog [\zeta_p^{(\infty)}]$.  This element freely generates $T_p(W_n\Omega^1_A)$ as a $W_n(A)$-module.  For each~$n \geq 1$, these elements  satisfy $F(\alpha_{n+1}^{(\infty)}) = \alpha_{n}^{(\infty)}$ and $R(\alpha_{n+1}^{(\infty)}) = R(z_{n+1}) \alpha_{n}^{(\infty)}$.
\end{theorem}

\begin{proof}
We first claim that if $\alpha$ and $\alpha' \in T_p(W_n\Omega^1_A)$ are such that 
\[
([\zeta_{p^n}] - 1) \alpha  = \dlog [\zeta_p^{(\infty)}] \qquad \text{ and } \qquad ([\zeta_{p^n}] - 1) \alpha' = \dlog [\zeta_p^{(\infty)}],
\]
then $\alpha = \alpha'$.
Choose an isomorphism of $W_n(A)$-modules $T_p(W_n\Omega^1_A) \cong W_n(A)$; this is possible by Theorem~\ref{generator of Tate}.  Let $x, y \in W_n(A)$ correspond to $\alpha, \alpha'$, respectively, under our chosen isomorphism.  Our assumption implies
\[
([\zeta_{p^n}] - 1) x = ([\zeta_{p^n}] - 1) y \in W_n(A),
\]
but we then find $x = y$ because $[\zeta_{p^n}] - 1$ is not a zero divisor in $W_n(A)$.  We already know such an element $\alpha \in T_p(W_n\Omega^1_A)$ exists by Proposition~\ref{divided log}, and now that we know it is unique, we name it $\alpha_n^{(\infty)}$.  It freely generates $T_p(W_n\Omega^1_A)$ as a $W_n(A)$-module by Corollary~\ref{identify generator}.

It remains to check the stated compatibilities under Frobenius and restriction.  Applying Frobenius to both sides of
\[
([\zeta_{p^{n+1}}] - 1) \alpha_{n+1}^{(\infty)}  = \dlog [\zeta_p^{(\infty)}],
\]
we find that $F(\alpha_{n+1}^{(\infty)})$ satisfies 
\[
F([\zeta_{p^{n+1}}] - 1) F(\alpha_{n+1}^{(\infty)})  = ([\zeta_{p^{n}}] - 1) F(\alpha_{n+1}^{(\infty)}) = \dlog [\zeta_p^{(\infty)}].
\]
Thus by uniqueness, we have $F(\alpha_{n+1}^{(\infty)}) = \alpha_{n}^{(\infty)}$.  

Similarly, applying restriction to both sides of 
\begin{align*}
([\zeta_{p^{n+1}}] - 1) \alpha_{n+1}^{(\infty)}  &= \dlog [\zeta_p^{(\infty)}] \in T_p(W_{n+1}\Omega^1_A),\\
\intertext{we find that $R(\alpha_{n+1}^{(\infty)})$ satisfies}
([\zeta_{p^{n+1}}] - 1) R(\alpha_{n+1}^{(\infty)})  &= \dlog [\zeta_p^{(\infty)}] \in T_p(W_n\Omega^1_A).\\
\intertext{On the other hand, $\alpha_n^{(\infty)}$ satisfies}
([\zeta_{p^{n}}] - 1) \alpha_{n}^{(\infty)}  &= \dlog [\zeta_p^{(\infty)}],\\
\intertext{so} 
([\zeta_{p^{n+1}}] - 1) R(z_{n+1}) \alpha_{n}^{(\infty)}  &= \dlog [\zeta_p^{(\infty)}].
\end{align*}
The fact that $R(\alpha_{n+1}^{(\infty)}) = R(z_{n+1}) \alpha_{n}^{(\infty)}$ follows because $[\zeta_{p^{n+1}}] - 1$ is not a zero divisor in $W_n(A)$.
\end{proof}

We next record a few easy consequences of Theorem~\ref{generator of Tate} and Theorem~\ref{canonical in tate}.  

\begin{corollary}
Let $A$ denote a ring as in Notation~\ref{has roots of unity}.  Then for all integers $n,r \geq 1$, we have that $W_n\Omega^1_A[p^r]$ is a free $W_n(A)/p^rW_n(A)$-module of rank~one
\end{corollary}

\begin{proof}
This follows immediately from Theorem~\ref{generator of Tate} and Lemma~\ref{p-torsion vs Tate module}.
\end{proof}

\begin{corollary} \label{F surjective on p-torsion}
Let $A$ denote a ring as in Notation~\ref{has roots of unity} and let $r,n \geq 1$ be integers.  
The Frobenius map $F: W_{n+1}\Omega^1_A[p^r] \rightarrow W_{n}\Omega^1_A[p^r]$ is surjective.  Also the Frobenius map $F: T_p(W_{n+1}\Omega^1_A) \rightarrow T_p(W_n\Omega^1_A)$ is surjective.
\end{corollary}

\begin{proof}
These follow immediately from the fact that $F: W_{n+1}(A) \rightarrow W_n(A)$ is surjective, and the fact from Proposition~\ref{canonical in tate} that Frobenius maps a generator of $T_p(W_{n+1}\Omega^1_A)$ to a generator of $T_p(W_n\Omega^1_A)$.
\end{proof}

The following corollary is analogous to \cite[Proposition~2.4.2]{Hes06}; see also the paragraph immediately preceding that proposition.  The proof of \cite[Proposition~2.4.2]{Hes06} uses topology.

\begin{corollary} \label{inverse F}
Let $A$ denote a ring as in Notation~\ref{has roots of unity}.  Consider $\varprojlim_F T_p \left( W_n \Omega^1_A\right)$ as a $W(A^{\flat})$-module via the ring isomorphism $\varprojlim_F {W_n(A)} \cong W(A^{\flat})$ from Lemma~\ref{inverse limit iso}.  For each integer $n \geq 1$, let $\alpha_n^{(\infty)} \in T_p(W_n\Omega^1_A)$ be the element specified in Theorem~\ref{canonical in tate}, and let $\alpha \in \varprojlim_F T_p \left( W_n \Omega^1_A\right)$ be the sequence of elements 
\[
\alpha = (\alpha_1^{(\infty)}, \alpha_2^{(\infty)}, \ldots) \in \varprojlim_F T_p \left( W_n \Omega^1_A\right).
\] 
Let $\varepsilon = (1, \zeta_p, \zeta_{p^2}, \ldots) \in A^{\flat}$.  The following properties hold.
\begin{enumerate}
\item The element $\alpha$ is the unique element satisfying $([\varepsilon] - 1) \alpha = (\dlog[\zeta_p^{(\infty)}])$.
\item The $W(A^{\flat})$-module $\varprojlim_F T_p \left( W_n \Omega^1_A\right)$ is a free module of rank~one, generated by $\alpha$.  
\end{enumerate}
\end{corollary}

\begin{proof}
Notice that $\widetilde{\theta}_r([\varepsilon]) = [\zeta_{p^r}] \in W_r(A)$.  The stated properties follow immediately from Theorem~\ref{canonical in tate}.
\end{proof}

\begin{corollary}
Let $A$ denote a ring as in Notation~\ref{has roots of unity} and let $r \geq 1$ denote an integer.
The projection map
\[
\varprojlim_F T_p \left( W_n \Omega^1_A\right) \rightarrow T_p \left(W_r \Omega^1_A\right)
\]
is a surjective map of $W(A^{\flat})$-modules, where the $W(A^{\flat})$-module structure on $T_p \left(W_r \Omega^1_A\right)$ is defined via $\widetilde{\theta}_r$.
\end{corollary}

\begin{proof}
This follows immediately from Corollary~\ref{inverse F} and the fact that, for each integer $r \geq 1$, the projection map $\widetilde{\theta}_r: W(A^{\flat}) \cong \varprojlim_F W_n(A) \rightarrow W_r(A)$ is surjective.
\end{proof}

Let $(x_1, x_2, \ldots)$ denote an arbitrary element in $\varprojlim_F T_p(W_n\Omega^1_A)$, where for each integer $n \geq 1$, we have $x_n \in T_p(W_n \Omega^1_A)$.  Let $R: T_p(W_{n+1}\Omega^1_A) \rightarrow T_p(W_n\Omega^1_A)$ denote the restriction map.  The sequence $(R(x_2), R(x_3), \ldots)$ is still Frobenius-compatible, and we again use $R$ to denote the corresponding map $(x_1, x_2, \ldots) \mapsto (R(x_2), R(x_3), \ldots)$.  Over the next few results, we seek to describe the elements which are fixed by this map $R$.

The map $R$ is not a $W(A^{\flat})$-module map for the module structure defined in Corollary~\ref{inverse F}, but we do have the following structure.

\begin{lemma} \label{R inv Tate}
Let $A$ denote a ring as in Notation~\ref{has roots of unity}.  Let $R$ denote the map
\[
R: \varprojlim_F T_p(W_n\Omega^1_A) \rightarrow \varprojlim_F T_p(W_n\Omega^1_A),
\]
defined by sending $(x_1, x_2, x_3, \ldots) \mapsto (R(x_2), R(x_3), \ldots).$  Let $\varphi: W(A^{\flat}) \rightarrow W(A^{\flat})$ denote the Witt vector Frobenius ring automorphism, and let $\varphi^{-1}$ denote its inverse.  For an arbitrary element $t \in W(A^{\flat})$ and $x \in \varprojlim_F T_p(W_n\Omega^1_A)$, define the product $tx \in \varprojlim_F T_p(W_n\Omega^1_A)$ using the isomorphism $W(A^{\flat}) \cong \varprojlim_F W_n(A)$.  We have
\[
R(tx) = \varphi^{-1}(t) R(x).
\]
\end{lemma}

\begin{proof}
Let $t \in W(A^{\flat})$ correspond to $(t_1, t_2, \ldots) \in \varprojlim_F W_n(A)$ under the isomorphism from Lemma~\ref{inverse limit iso}.  Then
\[
R(tx) = (R(t_2 x_2), R(t_3, x_3), \ldots) = \big(R(t_2), R(t_3), \ldots\big) R(x),
\]
so we reduce to checking that $ (R(t_2), R(t_3), \ldots) \in \varprojlim_F W_n(A)$ corresponds to $\varphi^{-1}(t) \in W(A^{\flat})$.  For this last claim, see the last sentence of \cite[Lemma~3.2]{BMS16}.
\end{proof}
  
\begin{proposition} \label{equal sets}
Let $A$ denote a ring as in Notation~\ref{has roots of unity}.  Let $\alpha \in \varprojlim_F T_p(W_n\Omega^1_A) $ denote the $W(A^{\flat})$-module generator described in Corollary~\ref{inverse F}.  Let $R: \varprojlim_F T_p(W_n\Omega^1_A) \rightarrow \varprojlim_F T_p(W_n\Omega^1_A)$ denote the map described in Lemma~\ref{R inv Tate}.  Finally, let $\varepsilon \in A^{\flat}$ denote the element $(1, \zeta_p, \zeta_{p^2}, \ldots)$.  We have the following equality of $\ZZ_p$-modules,
\[
\bigg\{ x \in \varprojlim_F T_p(W_n\Omega^1_A) \colon R(x) = x\bigg\} = \bigg\{ y\alpha \in \varprojlim_F T_p(W_n\Omega^1_A) \colon y \in W(A^{\flat}) \text{ satisfies } \varphi(y) = \frac{[\varepsilon^p] - 1}{[\varepsilon] - 1} \, y\bigg\}.
\]
\end{proposition}

\begin{proof}
We first find an element $t \in W(A^{\flat})$ such that $R(\alpha) = t \alpha$.  By definition of $\alpha$, using the notation of Corollary~\ref{inverse F}, we have
\begin{align*}
R(\alpha) &= (R(\alpha_2^{(\infty)}), R(\alpha_3^{(\infty)}), \ldots) \in \varprojlim_F T_p(W_n\Omega^1_A)\\
&= (R(z_2) \alpha_1^{(\infty)}, R(z_3) \alpha_2^{(\infty)}, \ldots)  \in \varprojlim_F T_p(W_n\Omega^1_A)\\
&= (\frac{[\zeta_p] - 1}{[\zeta_{p^2}] - 1} \alpha_1^{(\infty)}, \frac{[\zeta_{p^2}] - 1}{[\zeta_{p^3}] - 1} \alpha_2^{(\infty)}, \ldots)  \in \varprojlim_F T_p(W_n\Omega^1_A)\\
&= \left(\frac{[\varepsilon] - 1}{[\varepsilon^{1/p}] - 1} \right) \alpha \in \varprojlim_F T_p(W_n\Omega^1_A).
\end{align*}

Using this preliminary calculation, it is easy to complete the proof.  Namely, each element $x \in \varprojlim_F T_p(W_n\Omega^1_A)$ can be written uniquely in the form $y \alpha$ for some $y \in W(A^{\flat})$.  We have $R(x) = x$ if and only if 
\[
\varphi^{-1}(y) \left(\frac{[\varepsilon] - 1}{[\varepsilon^{1/p}] - 1} \right) \alpha = y \alpha,
\]
where we have used Lemma~\ref{R inv Tate} to express $R(y\alpha)$ in terms of $R(\alpha)$.  Because $\alpha$ is a free generator and because $\varphi$ is an automorphism, this last equality holds if and only if 
\[
y \left(\frac{[\varepsilon^p] - 1}{[\varepsilon] - 1} \right) = \varphi(y),
\]
as claimed.
\end{proof}

The following result gives a concrete description of the $\ZZ_p$-module $\{ x \in \varprojlim_F T_p(W_n\Omega^1_A) \colon R(x) = x\}$, but we are not able to find a similar result in the same generality as Notation~\ref{has roots of unity}, so in the following proposition, in addition to the assumptions of Notation~\ref{has roots of unity}, we assume the ring $A$ is an integral domain. The exact same result certainly does not hold in general; for example, it will not hold for $A = \ZZ_p[\zeta_{p^\infty}]^{\wedge} \times \ZZ_p[\zeta_{p^\infty}]^{\wedge}$, where the $\ZZ_p$-module in question will be isomorphic to $\ZZ_p \times \ZZ_p$.  The proof of Proposition~\ref{fixed by R} is based on \cite[Corollary~1.3.3]{Hes06}.

\begin{proposition} \label{fixed by R}
Let $A$ be as in Notation~\ref{has roots of unity}, and assume furthermore that $A$ is an integral domain.  The elements $y \in W(A^{\flat})$ satisfying the equation
\[
\varphi(y) = \frac{[\varepsilon^p] - 1}{[\varepsilon] - 1} y
\]
are precisely the elements of the form $c ([\varepsilon] - 1)$ for $c \in \ZZ_p$.  In other words, the $\ZZ_p$-module
\[
\bigg\{ x \in \varprojlim_F T_p(W_n\Omega^1_A) \colon R(x) = x\bigg\}
\]
is a free $\ZZ_p$-module of rank~one, generated by 
\[
(\dlog [\zeta_p^{(\infty)}], \dlog [\zeta_p^{(\infty)}], \ldots) \in \varprojlim_F T_p(W_n\Omega^1_A),
\]
where the element $\dlog [\zeta_p^{(\infty)}] \in T_p(W_n\Omega^1_A)$ was defined in the paragraph after Corollary~\ref{identify generator}.
\end{proposition}

\begin{proof}
Recall that for a characteristic~$p$ ring $R$, the Witt vector Frobenius map $W_{n+1}(R) \rightarrow W_n(R)$ induces the map $\varphi: W_n(R) \rightarrow W_n(R)$ which sends $(r_1, r_2, \ldots, r_n) \mapsto (r_1^p, r_2^p, \ldots, r_n^p)$.  It suffices to show, for every integer $n \geq 1$, that 
\[
\left\{y \in W_n(A^{\flat}) \colon \varphi(y) = \frac{[\varepsilon^p] - 1}{[\varepsilon] - 1} y \right\} = \left\{x  ([\varepsilon] - 1) \in W_n(A^{\flat}) \colon x \in W_n(\FF_p) \right\}.
\]
The inclusion $\supseteq$ is obvious, because $\varphi(x) = x$ for all $x \in W_n(\FF_p)$.  We prove equality using induction on $n \geq 1$.  In the base case, $n = 1$, we consider
\[
\left\{y \in A^{\flat} \colon y^p -   \frac{\varepsilon^p - 1}{\varepsilon - 1} y = 0 \right\}.
\]
Notice that, because we assumed that $A$ was an integral domain, we know further that $A^{\flat}$ is an integral domain (for example, this is clear by considering the isomorphism of multiplicative monoids $A^{\flat} \cong \varprojlim_{x \rightarrow x^p} A$).  We already have $p$ solutions to the above equation, and because $A^{\flat}$ is an integral domain, there can be no other solutions.  This proves the base case.

Now assume the equality is known for the case of some fixed $n \geq 1$, and assume 
\[
y \in W_{n+1}(A^{\flat}) \text{ is such that } \varphi(y) = \frac{[\varepsilon^p] - 1}{[\varepsilon] - 1} y.
\]
By our induction hypothesis, we know that 
\[
y = x ([\varepsilon] - 1) + V^n(z), \text{ for some } x \in W_{n+1}(\FF_p), z \in A^{\flat}.
\]
We are finished if we show that there exists $a \in \FF_p$ such that $V^n(z) = ([\varepsilon] - 1) V^n(a) \in W_{n+1}(A^{\flat})$.
Using our assumption on $y$, we find
\[
\varphi(V^n(z)) = \frac{[\varepsilon^p] - 1}{[\varepsilon] - 1} V^n(z) \in W_{n+1}(A^{\flat}).
\]
Thus 
\[
z^p = \frac{\varepsilon^{p^{n+1}} - 1}{\varepsilon^{p^n} - 1}z \in A^{\flat}.
\]
The values $z = a(\varepsilon^{p^n} - 1)$ for $a \in \FF_p$ provide $p$ solutions, and hence all solutions (again using that $A^{\flat}$ is an integral domain).  This completes the proof.
\end{proof}


\begin{remark} Proposition \ref{fixed by R} can be generalized. Let $A$ be a ring as in Notation~\ref{has roots of unity}. Using algebraic $K$-theory and topological cyclic homology, it follows from \cite[Corollary 6.9]{CMM} and \cite[Corollary 6.5]{AL19} that the $\ZZ_p$-module 
\[
\bigg\{ x \in \varprojlim_F T_p(W_n\Omega^1_A) \colon R(x) = x\bigg\}
\]
is isomorphic to the $p$-adic Tate module $T_p(A^{\times})$. We do not know an algebraic proof of this fact. More details are discussed in Remark \ref{last remark} below. \end{remark}

\section{Higher degrees}

The analysis is easier in degrees $d \geq 2$.  The main result in this section, Proposition~\ref{higher degree prop}, is modeled after \cite[Proposition~2.2.1]{Hes06}.

\begin{lemma} \label{exterior power}
Assume that $R$ is a ring, that $M$ is an $R$-module, and that $r \in R$ is such that multiplication by $r$ is surjective on $M$.  Then for every integer $d \geq 2$, multiplication by $r$ is an isomorphism on the exterior power $\Lambda_R^d M$. 
\end{lemma}

\begin{proof}
Let $M[r^{\infty}]$ denote the submodule of $M$ consisting of all elements which are annihilated by a power of $r$.
We have a short exact sequence of $R$-modules
\[
0 \rightarrow M[r^{\infty}] \rightarrow M \rightarrow R[1/r] \otimes_R M \rightarrow 0.
\]
For any integer $n \geq 1$, let $T^n_R(M) := M \otimes_R M \otimes_R \cdots \otimes_R M$, with $n$ total $M$ terms.
Because tensor product is right exact,  we obtain an exact sequence of $R$-modules
\[
M[r^{\infty}] \otimes_R T^{d-1}_R(M) \rightarrow M \otimes_R T^{d-1}_R(M) \rightarrow R[1/r] \otimes_R M \otimes_R T^{d-1}_R(M) \rightarrow 0.
\]
On the other hand, the left-hand term is zero because multiplication by $r$ is surjective on $M$.  This shows that multiplication by $r$ is an isomorphism on the $R$-module $T^d_R(M)$.  It follows immediately that multiplication by $r$ is surjective on the exterior power $\Lambda^d_R M$.  It remains to show that multiplication by $r$ is injective on $\Lambda^d_R M$.

Assume $\sum rm_{1i} \wedge \cdots \wedge m_{di} = 0 \in \Lambda^d_R M$.  We must show that $\sum m_{1i} \wedge \cdots \wedge m_{di} = 0 \in \Lambda^d_R M$.  By definition of $\Lambda^d_R M$ (see for example \cite[Tag~00DM]{stacks-project}), we know that 
\[
\sum rm_{1i} \otimes \cdots \otimes m_{di} = \sum a_j \cdot x_j \otimes x_j \cdot b_j\in T^d_R(M),
\]
for some suitable $a_j, b_j \in T_R(M)$ and $x_j \in M$.  Choose $y_j$ such that $ry_j = x_j$.  Then we have 
\[
\sum rm_{1i} \otimes \cdots \otimes m_{di} = \sum a_j \cdot ry_j \otimes ry_j \cdot b_j\in T^d_R(M),
\]
Because, as we saw above, multiplication by $r$ is injective on $T^d_R(M)$, we have
\[
\sum m_{1i} \otimes \cdots \otimes m_{di} = \sum ra_j \cdot y_j \otimes y_j \cdot b_j\in T^d_R(M),
\]
Thus $\sum m_{1i} \wedge \cdots \wedge m_{di} = 0 \in \Lambda^d_R M$, as required.
\end{proof}

\begin{corollary} \label{injective higher degree}
Let $A$ denote a $p$-torsion-free perfectoid ring and let $d \geq 2$ be an integer.  The multiplication-by-$p$ map $\Omega^d_A \stackrel{p}{\rightarrow} \Omega^d_A$ is an isomorphism of $A$-modules.
\end{corollary}

\begin{proof}
Because multiplication by $p$ is surjective on $\Omega^1_A$, this follows immediately from Lemma~\ref{exterior power}.
\end{proof}

\begin{proposition} \label{higher degree prop}
Let $A$ denote a $p$-torsion-free perfectoid ring and let $d \geq 2$ and $n \geq 1$ be integers.  There is a short exact sequence of $W_{n+1}(A)$-modules
\[
0 \rightarrow \Omega^d_A \stackrel{V^n}{\rightarrow} W_{n+1}\Omega^d_A \rightarrow W_n \Omega^d_A \rightarrow 0,
\]
where the $W_{n+1}(A)$-module structure on $\Omega^d_A$ is defined via $F^n$ and where the $W_{n+1}(A)$-module structure on $W_n\Omega^d_A$ is defined via restriction.
\end{proposition}

\begin{proof}
By \cite[Proposition~3.2.6]{HM03}, we always have an exact sequence 
\[
\Omega^d_A \oplus \Omega^{d-1}_A \stackrel{V^n + dV^n}{\longrightarrow} W_{n+1}\Omega^d_A \rightarrow W_n \Omega^d_A \rightarrow 0,
\]
so it suffices to show that the map $V^n: \Omega^d_A \rightarrow W_{n+1}\Omega^d_A$ is injective, and that the image of $V^n: \Omega^d_A \rightarrow W_{n+1}\Omega^d_A$ is the same as the image of $V^n + dV^n:  \Omega^d_A \oplus \Omega^{d-1}_A \rightarrow W_{n+1}\Omega^d_A$.  

To see injectivity, notice that $p^n = F^n \circ V^n$ on $\Omega^d_A$, and by Corollary~\ref{injective higher degree}, multiplication by $p^n$ is injective on $\Omega^d_A$ (this is one of two places where we use that $d > 1$), so we have that $V^n$ is also injective on $\Omega^d_A$.

To see the claim about the image, it suffices to show that for every $\alpha \in \Omega^{d-1}_A$, there exists $\alpha' \in \Omega^d_A$ such that $dV^n(\alpha) = V^n(\alpha') \in W_{n+1}\Omega^d_A$.
Because multiplication by $p$ is surjective on $\Omega^{d-1}_A$ (this is the other place where we use that $d > 1$), we can write $\alpha = p^n \alpha_0$.   Thus $dV^n(\alpha) = V^n(d\alpha_0)$, and so we may take $\alpha' = d\alpha_0$.
\end{proof}

\begin{corollary}
Let $A$ denote a $p$-torsion-free perfectoid ring and let $d \geq 2$ and $n \geq 1$ be integers.  Multiplication by $p$ on $W_n\Omega^d_A$ is an isomorphism of $W_n(A)$-modules.
\end{corollary}

\begin{proof}
This follows easily from the above results using the five lemma and induction on $n$.
\end{proof}

\section{Results on the logarithmic de\thinspace Rham-Witt complex} \label{log section}

In this section, we consider the logarithmic de\thinspace Rham-Witt complex, $W_{\cdot}\Omega^{\bullet}_{(A,M)}$, which is defined by Hesselholt-Madsen in \cite[Proposition~3.2.2]{HM03} as the initial object in the category of log Witt complexes.  Many of the ideas in this section were based on \cite[Proof of Proposition~2.2.2]{HM03} and conversations with Lars Hesselholt.  Our main result, which is an application of Theorem~\ref{sequence theorem}, is that, for certain choices of log ring $(A,M)$, the degree~one component of the logarithmic de\thinspace Rham-Witt complex is isomorphic to the ordinary de\thinspace Rham-Witt complex.  The precise statement is the following (the proof comes later in this section).

\begin{theorem} \label{log iso val ring}
Let $A$ denote a $p$-torsion-free perfectoid ring, and assume further that $A$ is a valuation ring.  Let $M = A \setminus \{0\}$.  Because $W_{\cdot}\Omega^{\bullet}_{(A,M)}$ is a Witt complex, we have a natural map
\[
W_{\cdot}\Omega^{\bullet}_{A} \rightarrow W_{\cdot}\Omega^{\bullet}_{(A,M)}.
\]
For every integer $n \geq 1$, the corresponding map
\[
W_n\Omega^1_A \rightarrow W_n\Omega^1_{(A,M)}
\]
is an isomorphism of $W_n(A)$-modules.
\end{theorem}

We will prove Theorem~\ref{log iso val ring} below.  We first prove some analogous results concerning the level~1 case, $\Omega^1_{(A,M)}$.  We will need the following explicit description, which we quote from Hesselholt and Madsen, but see also \cite[Section~(1.7)]{Kat89}.

\begin{definition}[{\cite[Section~2.2]{HM03}}] \label{explicit log}
A log ring $(A,M)$ consists of a ring $A$, a monoid $M$, and a map of monoids $\alpha: M \rightarrow A$, where $A$ is considered as a monoid under multiplication.  We define the $A$-module $\Omega^1_{(A,M)}$ to be
\[
\Omega^1_{(A,M)} = \big( \Omega^1_A \oplus (A \otimes_{\ZZ} M^{\text{gp}})\big)/ J,
\]
where $J \subseteq \Omega^1_A \oplus (A \otimes_{\ZZ} M^{\text{gp}})$ denotes the $A$-submodule generated by the elements of the form $(d\alpha(m), -\alpha(m) \otimes m)$, for $m \in M$.
\end{definition}

Most of our results in this section require that the ring $A$ be a valuation ring.  The following preliminary result moreover assumes algebraic closedness.  We have a similar preliminary result below (Lemma~\ref{log alg closed}) concerning the logarithmic de\thinspace Rham-Witt complex.

\begin{lemma} \label{log dR alg closed}
Let $\overline{R}$ denote a $p$-torsion-free perfectoid valuation ring for which $\Frac \overline{R}$ is algebraically closed.  Let $\overline{M}$ be the monoid $\overline{M} = \overline{R} \setminus \{0\}$.  The natural map
\[
f: \Omega^1_{\overline{R}} \rightarrow \Omega^1_{(\overline{R}, \overline{M})}
\]
is an isomorphism of $\overline{R}$-modules.
\end{lemma}

\begin{proof}
We will define an inverse map of $\overline{R}$-modules, $g: \Omega^1_{(\overline{R}, \overline{M})} \rightarrow \Omega^1_{\overline{R}}$, using the explicit description of $\Omega^1_{(\overline{R}, \overline{M})}$ from Definition~\ref{explicit log}. For arbitrary $m \in \overline{M}$, we first describe the image of $(0, 1 \otimes m) = \dlog m$ under $g$.  Choose an integer $N \geq 0$ such that $m$ divides $p^N$ in $\overline{R}$; this is possible because $\overline{R}$ is a $p$-adically separated valuation ring.  Choose $y \in \overline{R}$ such that $y^{p^N} = m$.  Note that $y$ divides $p^N$ also, and let $x \in \overline{R}$ be such that $xy = p^N$.  We then set 
\[
g(0,1 \otimes m) = x dy \in \Omega^1_{\overline{R}}.
\] 
(As a reality check, notice that in the case of a unit $u \in \overline{M}$, we may take $N = 0$ and $y = u$ and $x = u^{-1}$, and then we are mapping $\dlog u$ to $u^{-1} du$.)  We check that this depends only on $m$ and not on the values of $N$ or $y$.  In other words, we must show that if $m \mid p^{N_1}$, $y_1^{p^{N_1}} = m$, $x_1 y_1 = p^{N_1}$, and $m \mid p^{N_2}$, $y_2^{p^{N_2}} = m$, $x_2 y_2 = p^{N_2}$, then $x_1 dy_1 = x_2 dy_2 \in \Omega^1_{\overline{R}}$.  

First assume we have $m \mid p^N$, $z^{p^{N+1}} = m$, and $wz = p^{N+1}$.  Set $y = z^p$ (so $y^{p^N} = m$) and set $x$ to be the unique element in $\overline{R}$ such that $xy = p^N$.  One checks that $w = pxz^{p-1}$.  We claim that $x dy = w dz \in \Omega^1_{\overline{R}}$.  Indeed, we compute
\begin{align*}
x dy &= x d(z^p) \\
&= pz^{p-1} x dz \\
&= w dz.
\end{align*}
Thus, in our proof that $g(0, 1 \otimes m)$ depends only on $m$ (and does not depend on $y$ and does not depend on $N$), we reduce to proving the following.  We must show that if $m \mid p^{N}$, $y_1^{p^{N}} = m$, $x_1 y_1 = p^{N}$, and $y_2^{p^{N}} = m$, $x_2 y_2 = p^{N}$, then $x_1 dy_1 = x_2 dy_2 \in \Omega^1_{\overline{R}}$.  Without loss of generality, we may assume $y_1 \mid y_2$ (because $\overline{R}$ is a valuation ring), say $zy_1 = y_2$, and thus $z^{p^N} m = m$, and because $\overline{R}$ is an integral domain and $m \neq 0$, we have $z^{p^N} = 1$.  So we must check that
\[
x dy = (\zeta^{-1} x) d(\zeta y) \in \Omega^1_{\overline{R}},
\]
where $\zeta$ is a (not necessarily primitive) $p^N$-th root of unity.  This equality follows immediately from the Leibniz rule and the fact that $p^N d\zeta = 0$.  This completes the proof that our description of $g(0, 1 \otimes m)$  depends only on $m$.

We get a corresponding map $g: \overline{M} \rightarrow \Omega^1_{\overline{R}}$ given by $m \mapsto g(0, 1 \otimes m)$ as defined in the previous paragraph.  We skip the proof that this map $g:\overline{M} \rightarrow \Omega^1_{\overline{R}}$ satisfies $g(m_1 m_2) = g(m_1) + g(m_2)$, but the proof is straighforward, and we prove a more general result carefully in the proof of Proposition~\ref{log dR val ring}.  The map of monoids $\overline{M} \rightarrow \Omega^1_{\overline{R}}$ in turn defines a map of abelian groups $\overline{M}^{\text{gp}} \rightarrow \Omega^1_{\overline{R}}$ and of $\overline{R}$-modules $\overline{R} \otimes_{\ZZ} \overline{M}^{\text{gp}} \rightarrow \Omega^1_{\overline{R}}$.  We also set $g(\alpha, 0) = \alpha$ for every $\alpha \in \Omega^1_{\overline{R}}$.  Together these induce a unique map of $\overline{R}$-modules $\Omega^1_{\overline{R}} \oplus \big( \overline{R} \otimes_{\ZZ} \overline{M}^{\text{gp}}  \big)$.  It is straightforward to check that $g(dm, -m \otimes m) = 0$ for every $m \in \overline{M}$.  In particular, we have a well-defined map of $\overline{R}$-modules, $g: \Omega^1_{(\overline{R}, \overline{M})} \rightarrow \Omega^1_{\overline{R}}$.  

We now claim that $g$ is the inverse of $f$.  It is clear that $g \circ f$ is the identity on $\Omega^1_{\overline{R}}$.  To check that $f \circ g$ is the identity on $\Omega^1_{(\overline{R}, \overline{M})}$, it suffices to check that $(f \circ g)(\alpha, 0) = (\alpha, 0)$ for every $\alpha \in \Omega^1_{\overline{R}}$ and that $(f \circ g)(0, 1 \otimes m) = (0, 1 \otimes m)$ for every $m \in M$.  The first claim is obvious from the definitions.  To prove the second claim, fix $m \in M$ and let $N, y, x$ be as above in the definition of $g$.  We must check that
\[
(x dy, 0) = (0, 1 \otimes m) \in \Omega^1_{(\overline{R}, M)}.
\]  
Indeed, we have
\[
(x dy, 0) = (0, xy \otimes y) = (0, p^N \otimes y) = (0, 1 \otimes m) \in \Omega^1_{(\overline{R}, M)}.
\]
This completes the proof that the natural map
\[
f: \Omega^1_{\overline{R}} \rightarrow \Omega^1_{(\overline{R}, \overline{M})}
\]
is an isomorphism of $\overline{R}$-modules.
\end{proof}

\begin{remark}
 The analogous result to Lemma~\ref{log dR alg closed} does hold for arbitrary $p$-torsion-free perfectoid valuation rings (see Proposition~\ref{log dR val ring}), but it will require a different proof.  
The proof of Lemma~\ref{log dR alg closed} does not carry over to arbitrary $p$-torsion-free perfectoid valuation rings, because the proof requires taking $p$-power roots.  

If we drop the valuation ring hypothesis, it seems unlikely that the analogous result to Lemma~\ref{log dR alg closed} will continue to hold.  For example, is the natural map $\Omega^1_A \rightarrow \Omega^1_{(A,M)}$ an isomorphism for the perfectoid ring $A = \OCp \langle x^{1/p^{\infty}} \rangle$ and $M = A \setminus \{0\}$?  We suspect that $\dlog x$ is not in the image, but we have not proved this.  
\end{remark}

\begin{lemma} \label{dlog unit}
Let $A$ denote an integral domain and let $M$ denote the monoid $M = A \setminus \{0\}$. If we have $y^{p^N} = um$ and $yx = p^N$, where $m,x,y \in M$, $N \geq 0$ is an integer, and $u \in A$ is a unit, then
\[
\dlog m = x dy - u^{-1} du \in \Omega^1_{(A,M)}.
\]
\end{lemma}

\begin{proof}
This is straightforward, using the explicit description of $\Omega^1_{(A,M)}$ from Definition~\ref{explicit log}.  We have
\begin{align*}
(xdy - u^{-1} du, 0) &= \big(0, (xy \otimes y) - (u^{-1} u \otimes u)\big) \in \big( \Omega^1_{A} \oplus (A \otimes_{\ZZ} M^{\text{gp}})\big)/ J\\
&= \big(0, (p^N \otimes y) - (1 \otimes u)\big) \\
&= \big(0, (1 \otimes um) - (1 \otimes u)\big) \\
&= \dlog m,
\end{align*}
as required.
\end{proof}

The motivation for Lemma~\ref{dlog unit} is that, in an arbitrary $p$-torsion-free perfectoid valuation ring $A$ (i.e., without the algebraically closed hypothesis that we had in Lemma~\ref{log dR alg closed}), for an arbitrary element $m \in A \setminus \{0\},$ we can find elements $u,x,y,N$ as in the set-up of Lemma~\ref{dlog unit}.  We will use such elements to define $\dlog m \in \Omega^1_{A}$, and Lemma~\ref{dlog unit} will be used in the proof that the definition is independent of choices.

\begin{lemma} \label{unit roots}
Let $A$ denote a $p$-torsion-free perfectoid valuation ring.  Then for every element $a \in A$, there exists a unit $u \in A^{\times}$ and an element $\omega = (\omega^{(0)}, \omega^{(1)}, \ldots) \in \varprojlim_{x \mapsto x^p} A$ such that $ua = \omega^{(0)}$.
\end{lemma}

\begin{proof}
We adapt the proof of the ``moreover" assertion in \cite[Lemma~3.9]{BMS16}.  Let $a \in A$ be an arbitrary non-zero element, and let $\pi \in A$ be as in the definition of perfectoid.  Because $A$ is $\pi$-adically separated, we can write $a = \pi^m x$ for some integer $m \geq 0$ and some $x \in A$ such that $x \not\in \pi A$.  It suffices to prove the claim for the special cases $a = \pi$ and $a = x$ separately.  The case $a = \pi$ is explicitly stated in \cite[Lemma~3.9]{BMS16}.  As in that proof, we can find $\eta \in \varprojlim_{x \mapsto x^p} A$ and $y \in A$ with 
\[
\eta^{(0)} = x + \pi p y = x(1 + \frac{\pi}{x} p y);
\] 
note that $\frac{\pi}{x} \in A$ because $A$ is a valuation ring and $x \not\in \pi A$.  Because $A$ is $p$-adically complete, the element $1 + \frac{\pi}{x} p y$ is a unit in $A$, which completes the proof.
\end{proof}

\begin{lemma} \label{embed into alg closed}
Let $A$ denote a $p$-torsion-free perfectoid valuation ring.  There exists a $p$-torsion-free perfectoid valuation ring $\overline{R}$ and an injective ring homomorphism $A \rightarrow \overline{R}$ such that $\Frac \overline{R}$ is algebraically closed, such that the valuation on $\overline{R}$ extends the valuation on $A$, and such that, for every integer $n \geq 1$, the induced map
\[
\Omega^1_A[p^n] \rightarrow \Omega^1_{\overline{R}}[p^n]
\]
is injective.
\end{lemma}

\begin{proof}
Let $\overline{R}$ be constructed as in the proof of Lemma~\ref{embed into product}.   By considering valuations, we see that $A \cap p^n \overline{R} = p^n A$ for every integer $n \geq 1$.  In particular, the natural map $A/p^nA \rightarrow \overline{R}/p^n\overline{R}$ is injective.

Choose a generator for $\ker \theta$, where $\theta: W(A^{\flat}) \rightarrow A$ is the usual $\theta$ map, as in Definition~\ref{theta def}.  As in the proof of Theorem~\ref{main relative Omega theorem}, we may choose the same generator for $\ker \theta$ in $W(\overline{R}^{\flat})$, and these choices determine isomorphisms $\Omega^1_{A/\ZZ_p}[p^n] \cong A/p^nA$ and $\Omega^1_{\overline{R}/\ZZ_p} \cong \overline{R}/p^n\overline{R}$ in such a way that the natural map $\Omega^1_{A/\ZZ_p}[p^n] \rightarrow \Omega^1_{\overline{R}/\ZZ_p}[p^n]$ corresponds to the natural map $A/p^nA \rightarrow \overline{R}/p^n\overline{R}$, and hence is injective.  By Lemma~\ref{absolute isomorphism}, the natural map $\Omega^1_{A}[p^n] \rightarrow \Omega^1_{\overline{R}}[p^n]$ is then also injective.
\end{proof}

Our main result about the logarithmic de\thinspace Rham-complex is the following.  It is the level-one case of Theorem~\ref{log iso val ring}.

\begin{proposition} \label{log dR val ring}
Let $A$ denote a $p$-torsion-free perfectoid ring, and assume further that $A$ is a valuation ring.  Let $M$ be the monoid $A \setminus \{0\}$.  The natural map
\[
f: \Omega^1_A \rightarrow \Omega^1_{(A,M)}
\]
is an isomorphism of $A$-modules.
\end{proposition}

\begin{proof}
As in the proof of Lemma~\ref{log dR alg closed}, we construct an inverse map $g: \Omega^1_{(A,M)} \rightarrow \Omega^1_A$.  For fixed $m \in M$, choose an integer $N \geq 0$ such that $m \mid p^N$; this is possible because $A$ is a $p$-adically separated valuation ring.  Then choose a unit $u \in A$ and an element $y \in A$ such that $y^{p^N} = um$; this is possible by Lemma~\ref{unit roots}.  Choose $x \in A$ such that $xy = p^N$.  We set 
\[
g(0, 1 \otimes m) = x dy - u^{-1} du \in \Omega^1_A.
\]
We claim that this definition depends only on $m$ and not on the choices of $N,u,x,y$.  Notice that we have
\begin{align*}
m(x dy - u^{-1} du) &= u^{-1} y^{p^N}(x dy - u^{-1} du)\\
&= u^{-1} p^N y^{p^N - 1} dy - u^{-2} y^{p^N} du\\
&= dm \in \Omega^1_A.
\end{align*}
In particular, if we have $y_1^{p^{N_1}} = u_1 m, x_1 y_1 = p^{N_1}$ and $y_2^{p^{N_2}} = u_2 m, x_2 y_2 = p^{N_2}$, then 
\[
\bigg(x_1 dy_1 - u_1^{-1} du_1\bigg) - \bigg(x_2 dy_2 - u_2^{-1} du_2\bigg) \in \Omega^1_A
\]
is $m$-torsion, and hence is $p^N$-torsion.  

Choose an embedding $\eta: A \rightarrow \overline{R}$ as in Lemma~\ref{embed into alg closed}.  Notice that
\[
\eta : \bigg(x_1 dy_1 - u_1^{-1} du_1\bigg) - \bigg(x_2 dy_2 - u_2^{-1} du_2\bigg) \mapsto \dlog m - \dlog m = 0 \in \Omega^1_{\overline{R}}
\]
by Lemma~\ref{dlog unit}.  On the other hand, $\eta: \Omega^1_A[p^N] \rightarrow \Omega^1_{\overline{R}}[p^N]$ is injective by Lemma~\ref{embed into alg closed}.  Thus 
\[
\bigg(x_1 dy_1 - u_1^{-1} du_1\bigg) - \bigg(x_2 dy_2 - u_2^{-1} du_2\bigg) = 0 \in \Omega^1_A.
\]
This shows that our construction of $g(0,1 \otimes m)$ depends only on $m$ and not on any other choices.

We next claim that the corresponding map $g: M \rightarrow \Omega^1_A$ is a map of monoids.  Let $m_1, m_2 \in M$ be arbitrary, and choose an integer $N$ such that $m_1\mid p^N$ and $m_2 \mid p^N$.  For this value of $N$, using Lemma~\ref{unit roots}, choose $u_1, y_1, u_2, y_2$ such that $y_1^{p^{2N}} = u_1 m_1$, $y_1 x_1 = p^N$, $y_2^{p^{2N}} = u_2 m_2$, $y_2 x_2 = p^N$.  Notice that $(y_1 y_2)^{p^{2N}} = u_1 u_2 m_1 m_2$ and $y_1 y_2 x_1 x_2 = p^{2N}$.  Thus on one hand
\[
g(m_1 m_2) = x_1 x_2 d(y_1 y_2) - (u_1 u_2)^{-1} d(u_1 u_2).
\]
On the other hand, using that $y_1 p^N x_1 = p^{2N}$ and $y_2 p^N x_2 = p^{2N}$, we also have
\[
g(m_1) + g(m_2) = p^N x_1 dy_1 - u_1^{-1} du_1 + p^N x_2 dy_2 - u_2^{-1} du_2.
\]
Using the Leibniz rule and the definitions, one sees that these two expressions are equal.  This shows that $g: M \rightarrow \Omega^1_A$ is a map of monoids.

We return to defining $g: \Omega^1_{(A,M)} \rightarrow \Omega^1_A$.  We require that $g(\alpha, 0) = \alpha \in \Omega^1_A$ for every $\alpha \in \Omega^1_A$.  This requirement, together with our definition of $g(0, 1 \otimes m)$, determines a unique $A$-module map $g: \Omega^1_{(A,M)} \rightarrow \Omega^1_A$, as in the proof of Lemma~\ref{log dR alg closed}.  It is clear that $g \circ f$ is the identity map on $\Omega^1_A$.  To prove that $f \circ g$ is the identity map on $\Omega^1_{(A,M)}$, we must show that, under the notation above, we have
\[
(x dy - u^{-1} du, 0 ) = (0, 1 \otimes m) \in \Omega^1_{(A,M)}.
\]
This was shown in Lemma~\ref{dlog unit}.  This shows that $g$ is the inverse of $f$, and hence that $f$ is an isomorphism of $A$-modules.
\end{proof}

This concludes our treatment of the level~1 case of Theorem~\ref{log iso val ring}.  We next state a preliminary result about Witt vectors.

\begin{lemma} \label{divide by Teichmuller}
Let $A$ denote a $p$-torsion-free perfectoid valuation ring.  Fix an integer $n \geq 1$.  For every non-zero $a \in A$, there exists an integer $N \geq 1$ and a Witt vector $x \in W_n(A)$ such that
\[
[a]x = p^N \in W_n(A).
\]
For a fixed value of $N$, the corresponding value of $x$ is unique.
\end{lemma}

\begin{proof}
The uniqueness is obvious because the ghost components of $[a]$ are non-zero.  
We prove existence by proving the following stronger result using induction on $n$.  
\begin{itemize}
\item Let $n \geq 1$ be an integer.  Let $a \in A$ be non-zero.  There exists an integer $N \geq 1$ (depending on $n$ and on $a$) such that, if $y \in W_n(A)$ has all ghost components in $p^N A$, then there exists $x \in W_n(A)$ such that $[a]x = y \in W_n(A)$.
\end{itemize}
The base case $n = 1$ follows immediately from the fact that $A$ is a valuation ring which is $p$-adically separated.  Now inductively assume the result has been proven for some fixed value of $n$. We prove the result for $n + 1$.  Let $a \in A$ be non-zero, and let $N$ be such that, if $y' \in W_n(A)$ has all ghost components in $p^{N}A$, then $y'$ is divisible by $[a^p]$.  We claim that if $y \in W_{n+1}(A)$ has all ghost components in $p^{2N} A$, then $y$ is divisible by $[a]$ in $W_{n+1}(A)$.  Thus, assume we have such a $y$, and let $p^{N} z_0$ denote the first Witt component of $y$.  We have that $z_0 \in p^N A$, so $z_0 = ax_0$ for some $x_0 \in A$.  Then we have
\[
y - [p^N x_0] [a] = V(z') \in W_{n+1}(A)
\]
for some $z' \in W_n(A)$, and note that all the ghost components of $z'$ are in $p^N A$.  Thus 
\[
z' = [a^p] z'' = F([a]) z'' \in W_n(A)
\]
for some $z'' \in W_n(A)$.  Thus
\[
y = [a] \bigg( [p^N x_0] + V(z'') \bigg) \in W_{n+1}(A).
\]
This completes the induction.
\end{proof}

Proving the surjectivity of the map in Theorem~\ref{log iso val ring} is easier than proving injectivity.  We prove surjectivity now.

\begin{lemma} \label{log surjective}
Let $A$ denote a $p$-torsion-free perfectoid valuation ring.  Let $M = A \setminus \{0\}$.  Fix an integer $n \geq 1$.  The natural map
\[
W_n\Omega^1_A \rightarrow W_n\Omega^1_{(A,M)}
\]
is surjective.
\end{lemma}

\begin{proof}
For any log ring $(A,M)$, with map of monoids $\alpha: M \rightarrow A$, and for any integer $n \geq 1$, we consider $(W_n(A),M)$ as a log ring using the monoid map $m \mapsto [\alpha(m)]$. In our case, that $M = A \setminus \{0\}$, we will write commit a slight abuse of notation and write $\dlog [m] \in \Omega^1_{(W_n(A),M)}$ instead of $\dlog m$, to remind ourselves about the presence of this Teichm\"uller lift. 

The canonical map 
\[
\Omega^1_{(W_n(A),M)} \rightarrow W_n\Omega^1_{(A,M)}
\]
is surjective by \cite[Proposition~3.2.2]{HM03}; this surjectivity holds in great generality, for any $\ZZ_{(p)}$-algebra $A$ with $p$ odd.  So, to prove our lemma, it suffices to show that every element $\dlog [m] \in \Omega^1_{(W_n(A),M)}$ is in the image of the natural map $\Omega^1_{W_n(A)} \rightarrow \Omega^1_{(W_n(A),M)}$.  

First assume that $u \in A$ is a unit.  We know by definition that 
\[
[u] \dlog [u] = d[u] \in \Omega^1_{(W_n(A),M)},
\]
and so, multiplying both sides by $[u^{-1}]$, we have that $\dlog [u] = [u^{-1}] d[u]$.  Thus, for every unit $u \in A^{\times},$ the element $\dlog [u] \in \Omega^1_{(W_n(A),M)}$ is in the image of $\Omega^1_{W_n(A)} \rightarrow \Omega^1_{(W_n(A),M)}$.

Now let $m \in M$ be arbitrary.   Choose $N$ so that $[m] \mid p^N$ in $W_n(A)$ as in Lemma~\ref{divide by Teichmuller} and, using Lemma~\ref{unit roots}, choose a unit $u \in A$ and an element $y \in A$ so that $y^{p^N} = um$.   We have 
\[
p^N \dlog [y] = \dlog [u] + \dlog [m] \in \Omega^1_{(W_n(A),M)}.
\]
Because $[m]$ divides $p^N$ and because $[y]$ divides $[m]$, we have
\[
p^N\dlog[y] = \frac{p^N}{[y]} [y] \dlog[y] = \frac{p^N}{[y]} d[y];
\]
notice that these fractions are well-defined because $W_n(A)$ is $p$-torsion-free.
This shows that $\dlog [u] + \dlog [m]$ is in the image of $\Omega^1_{W_n(A)} \rightarrow \Omega^1_{(W_n(A),M)}$, and we already saw that $\dlog [u]$ is in the image, hence $\dlog[m]$ is also in the image.  This completes the proof.
\end{proof}

As in Lemma~\ref{log dR alg closed}, we treat the algebraically closed case before treating the case of an arbitrary $p$-torsion-free perfectoid valuation ring.

\begin{lemma} \label{log alg closed}
Let $\overline{R}$ denote a $p$-torsion-free perfectoid valuation ring for which $\Frac \overline{R}$ is algebraically closed.  Let $\overline{M}$ be the multiplicative monoid $\overline{R} \setminus \{0\}$.  Let $n \geq 1$ be an integer.  The natural map
\[
W_n\Omega^1_{\overline{R}} \rightarrow W_n\Omega^1_{(\overline{R}, \overline{M})}
\]
is an isomorphism of $W_n(\overline{R})$-modules.
\end{lemma}

\begin{proof}
For every integer $n \geq 1$, set $E^0_n := W_n(\overline{R})$, set $E^1_n := W_n\Omega^1_{\overline{R}}$, and for $d \geq 2$, set $E^d_n := 0$.  We will now equip $E_{\cdot}^{\bullet}$ with additional structures making it into a log Witt complex over $(\overline{R}, \overline{M})$.  

First, for every integer $n \geq 1$ and every $m \in \overline{M}$, we will define $\dlog [m] \in E_n^1 = W_n\Omega^1_{\overline{R}}$.  Using Lemma~\ref{divide by Teichmuller}, choose an integer $N \geq 1$ such that $[m] \mid p^N$ in $W_n(\overline{R})$.  Choose any $y \in \overline{R}$ such that $y^{p^N} = m$.  Note that $[y]$ divides $[m]$ and $[m]$ divides $p^N$, so there exists $x \in W_n(\overline{R})$ such that $[y]x = p^N$.  This value of $x$ is unique.  We then set
\[
\dlog [m] := x d[y] \in W_n\Omega^1_{\overline{R}}.
\]
The proof that $\dlog [m]$ depends only on $m$ and not on the choices of $N$ or $y$ is the same as in the proof of Lemma~\ref{log dR alg closed}.  The fact that $\dlog: M \rightarrow E_n^1$ is a map of monoids has the same proof as in Proposition~\ref{log dR val ring}.

Equip $E_{\cdot}^{\bullet}$ with the structure of a log Witt complex over the log ring $(\overline{R}, \overline{M})$ (see \cite[Definition~3.2.1]{HM03}) using the Witt complex structure on $W_{\cdot}\Omega^{\bullet}_{\overline{R}}$ and the above map $\dlog$.  Most of the requirements for being a log Witt complex over $(\overline{R}, \overline{M})$  are automatic because they hold for $W_{\cdot}\Omega^{\bullet}_{\overline{R}}$.  The only remaining requirement which is not obvious is that, for every $m \in \overline{M}$, we have $F \dlog [m] = \dlog [m]$.  (In the last equality, there is implicitly some fixed integer $n \geq 1$, and the $\dlog [m]$ on the left-hand side of the equality is in $E_{n+1}^1$, and the $\dlog [m]$ on the right-hand side is in $E_n^1$.)  We check $F \dlog [m] = \dlog [m]$ using the definition of $\dlog$.  Thus choose $N \geq 0$ so that $[m]$ divides $p^N$ in $W_{n+1}(\overline{R})$; choose $y \in \overline{R}$ so that $y^{p^N} = m$; and choose $x \in W_{n+1}(\overline{R})$ so that $[y]x = p^N \in W_{n+1}(\overline{R})$.  Then, applying $F$ to both sides, we have $[y]^p F(x) = p^N \in W_n(\overline{R})$, and so $[y] w = p^N \in W_n(\overline{R})$ for $w = [y]^{p-1} F(x)$.   We check
\[
F(\dlog [m]) = F( x d[y]) = F(x) F(d[y]) = F(x) [y]^{p-1} d[y] = \dlog [m] \in W_n\Omega^1_{\overline{R}}.
\]

Therefore $E_{\cdot}^{\bullet}$ forms a log Witt complex over $(\overline{R}, \overline{M})$, so we have a natural map $W_{\cdot} \Omega^{\bullet}_{(\overline{R},M)} \rightarrow E_{\cdot}^{\bullet}$.  On the other hand, considering these objects as Witt complexes (instead of log Witt complexes), we have maps
\[
W_{\cdot} \Omega^{\bullet}_{\overline{R}} \rightarrow W_{\cdot} \Omega^{\bullet}_{(\overline{R},M)} \rightarrow E_{\cdot}^{\bullet}
\]
Because there is a unique map of Witt complexes $W_{\cdot} \Omega^{\bullet}_{\overline{R}} \rightarrow E_{\cdot}^{\bullet}$, and there is the obvious map of Witt complexes which is the identity in degrees 0 and 1 and is zero in higher degrees, we deduce that the above composition is in particular injective in degree~1.  Thus for every integer $n \geq 1$, the natural map $W_{n} \Omega^{1}_{\overline{R}} \rightarrow W_{n} \Omega^{1}_{(\overline{R},M)}$ is injective.  We have already seen in Lemma~\ref{log surjective} that this map is surjective; hence it is an isomorphism.
\end{proof}

To prove Theorem~\ref{log iso val ring}, we will use the algebraically closed case (Lemma~\ref{log alg closed}) and an exact sequence analogous to the sequence (\ref{sequence}) from Theorem~\ref{sequence theorem}.  As in the proof of Theorem~\ref{sequence theorem}, most of the exactness holds in complete generality, due to the following result by Hesselholt and Madsen.  We already referenced this result for the de\thinspace Rham-Witt complex, with no log structure, above in Proposition~\ref{Theorem A HM}.

\begin{proposition}[{\cite[Proposition~3.2.6]{HM03}}] \label{log HM}
Let $A$ denote a $p$-torsion-free $\ZZ_{(p)}$-algebra and let $M$ denote a submonoid of the monoid of all non-zero-divisors in $A$.  For a fixed integer $n \geq 1$, consider the sequence of $W_{n+1}(A)$-modules
\begin{equation} \label{log sequence}
0 \rightarrow A \xrightarrow{(-d, p^n)}  \Omega^1_{(A,M)} \oplus A \xrightarrow{V^n \oplus dV^n} W_{n+1} \Omega^1_{(A,M)} \xrightarrow{R} W_n \Omega^1_{(A,M)} \rightarrow 0,
\end{equation}
where the module structure is the same as in Theorem~\ref{sequence theorem}.  This sequence is exact at all slots, except that possibly the segment $A \xrightarrow{(-d, p^n)}  \Omega^1_{(A,M)} \oplus A \xrightarrow{V^n \oplus dV^n} W_{n+1} \Omega^1_{(A,M)}$ is not exact.  Furthermore, we have $\im\, \big(-d, p^n\big) \subseteq \ker \big(V^n \oplus dV^n\big).$ 
\end{proposition}

The following result, Lemma~\ref{log sequence is exact}, is obvious from Theorem~\ref{sequence theorem} if we know that $W_n\Omega^1_{A}$ is isomorphic to $W_n\Omega^1_{(A,M)}$, but we will instead use  Lemma~\ref{log sequence is exact} to prove that $W_n\Omega^1_{A}$ is isomorphic to $W_n\Omega^1_{(A,M)}$.

\begin{lemma} \label{log sequence is exact}
Let $A$ denote a $p$-torsion-free perfectoid valuation ring, and let $M$ denote the monoid $A \setminus \{0\}$.  For every integer $n \geq 1$, the sequence~(\ref{log sequence}) is exact for this pair $(A,M)$.
\end{lemma}

\begin{proof}
Using Lemma~\ref{embed into alg closed}, we may embed $A$ into a $p$-torsion-free perfectoid valuation ring $\overline{R}$ so that $\Frac \overline{R}$ is algebraically closed and so that, for every integer $n \geq 1$, the natural map $\Omega^1_{A}[p^n] \rightarrow \Omega^1_{\overline{R}}[p^n]$ is injective. By Proposition~\ref{log dR val ring}, the natural map $\Omega^1_{(A,A \setminus \{0\})}[p^n] \rightarrow \Omega^1_{(\overline{R}, \overline{R} \setminus \{0\})}[p^n]$ is also injective.  We also have that $A \cap p^n \overline{R} = p^n A$ for every $n \geq 1$ (see the proof of Lemma~\ref{embed into alg closed}).

 Our goal is to prove that the sequence~(\ref{log sequence}) is exact for $(A, A \setminus \{0\})$.  The sequence (\ref{log sequence}) is exact for $(\overline{R}, \overline{R} \setminus \{0\}),$ by Lemma~\ref{log alg closed} and Theorem~\ref{sequence theorem}.   By Proposition~\ref{log HM}, to prove exactness for $(A, A \setminus \{0\})$, we need only show that, if $\alpha \in \Omega^1_{(A, A \setminus \{0\})}$ and $a \in A$ satisfy $V^n(\alpha) + dV^n(a) = 0 \in W_{n+1}\Omega^1_{(A, A \setminus \{0\})}$, then there exists $a_0 \in A$ such that $\alpha = -da_0$ and $a = p^n a_0$.  This is proved in exactly the same way as Proposition~\ref{deduce exactness for A} was proved.
\end{proof}

Using the above results, we easily deduce Theorem~\ref{log iso val ring}.

\begin{proof}[Proof of Theorem~\ref{log iso val ring}]
We prove that for every integer $n \geq 1$, the natural map $W_n\Omega^1_{A} \rightarrow W_n\Omega^1_{(A, A \setminus \{0\})}$ is an isomorphism using induction on $n$.  The base case $n = 1$ is precisely Proposition~\ref{log dR val ring}, by \cite[Addendum~3.2.3]{HM03}.  Now assume the result holds for some fixed value of $n$.  Consider the following commutative diagram:
\[
\xymatrix{
0 \ar[r] & A \ar^{(-d, p^n)}[rr] & \hspace{.2in} & \Omega^1_{(A,A \setminus \{0\})} \oplus A \ar^{V^n \oplus dV^n}[rr] &\hspace{.2in}& W_{n+1} \Omega^1_{(A,A \setminus \{0\})} \ar^{R}[r] & W_n \Omega^1_{(A,A \setminus \{0\})} \ar[r] & 0 \\
0 \ar[r] & A \ar^{(-d, p^n)}[rr] \ar[u] &  & \Omega^1_{A} \oplus A \ar^{V^n \oplus dV^n}[rr] \ar[u] & & W_{n+1} \Omega^1_A \ar^{R}[r] \ar[u] & W_n \Omega^1_A \ar[u] \ar[r] & 0
}
\]
By Theorem~\ref{sequence theorem}, the bottom sequence is exact, and by Lemma~\ref{log sequence is exact}, the top sequence is exact.  By the induction hypothesis, we know that three of the four vertical maps are isomorphisms.  Then, by the five lemma, we know also that $W_{n+1}\Omega^1_{A} \rightarrow W_{n+1}\Omega^1_{(A, A \setminus \{0\})}$ is an isomorphism.  This completes the induction.
\end{proof}

Only the degree one piece of the logarithmic de\thinspace Rham-Witt complex is relevant to our applications.  We are unsure if the analogous result to Theorem~\ref{log iso val ring} holds in degree~two, because we are unsure if the requirement $d \circ \dlog = 0$ introduces any new relations.

\begin{question}
Let $A$ denote a $p$-torsion-free perfectoid ring, and assume further that $A$ is a valuation ring.  Let $M = A \setminus \{0\}$.  Is the natural map
\[
W_n\Omega^d_A \rightarrow W_n\Omega^d_{(A,M)}
\]
an isomorphism for every integer $d \geq 0$?
\end{question}

\section{Connections to algebraic K-theory and topological Hochschild and cyclic homology} \label{tc section}

Algebraic K-theory provides one of the motivations for studying the $p$-adic Tate module $T_p \left( W_n \Omega^1_A \right)$ and the generators $\alpha$ and $\alpha_n^{(\infty)}$.  We describe these connections in this section following mainly \cite{Hes06}.  We will assume that the reader has some familiarity with algebraic K-theory; the readers only interested in algebraic aspects of the de\thinspace Rham-Witt complex can safely skip this section. Most of the results below are well-known to the experts and some of them are more general than we state here (see \cite{Hes06}, \cite{BMS18}, \cite{CMM}). Hence we do not claim any originality in this section. We only put the results of Sections \ref{main section} and \ref{log section} in a topological context. 

Let $V$ be a complete discrete valuation ring with quotient field $K$ of characteristic $0$ and with a perfect residue field of odd characteristic $p$. In \cite{Hes06}, Hesselholt studies the p-adic Tate module of $W_n\Omega^1_{(\overline{V},\overline{M})}$, where $\overline{V}$ is the ring of integers of the algebraic closure $\overline{K}$ and $\overline{M}=\overline{V} \setminus \{0\}$. He shows that $T_p\left(W_n\Omega^1_{(\overline{V},\overline{M})}\right)$ is a free $W_n(\overline{V})^{\wedge}$-module of rank $1$ on a certain generator $\alpha_{n, \varepsilon}$. It turns out that one can describe the image of the trace map from the $p$-adic algeberaic K-theory group $K_2(\overline{K}, \ZZ_p)$ to $T_p\left(W_n\Omega^1_{(\overline{V},\overline{M})}\right)$ in terms of this element  $\alpha_{n, \varepsilon}$ \cite{Hes06}. In particular, one can understand the image of a certain Bott class $\beta_{\varepsilon} \in K_2(\overline{K}, \ZZ_p)$ using the element $\alpha_{n, \varepsilon}$. We note that this Bott class corresponds to the classical Bott class in complex K-theory by results of Suslin \cite{Suslin1}, \cite{Suslin2} which show that $ku^{\wedge} \simeq K(\overline{K})^{\wedge}$, where $ku$ is the connective complex topological K-theory spectrum. In this section we show that $\alpha_{n, \varepsilon}$ is in fact a special case of the element constructed in Theorem~\ref{canonical in tate} and we compute the image of the Bott class in a more general setting of $p$-torsion-free perfectoid rings containing a compatible system of $p$-power roots of unity.

The main tools which connect $K_2(\overline{K}, \ZZ_p)$ with $T_p\left(W_n\Omega^1_{(\overline{V},\overline{M})}\right)$ are the B\"okstedt trace and the cyclotomic trace as constructed by \cite{BokHsiaMad}. These are maps from algebraic K-theory to invariants called \emph{topological Hochschild homology} and \emph{topological cyclic homology}, respectively. We give now a very brief overview of these objects.

Given a commutative ring $A$, the topological Hochschild homology spectrum $\THH(A)$ is defined as the tensor construction $S^1 \otimes A$ which is isomorphic to the geometric realization of the cyclic bar construction of $A$ over the sphere spectrum. For various equivalent definitions of $\THH(A)$ and the equivalences between them, see \cite{Bok}, \cite{Shi}, \cite{ABGHLM}, \cite{NS}, \cite{DMPSW}. The classical Dennis trace map
\[K(A) \to \HH(A)\]
going from algebraic K-theory of $A$ to Hochschild homology can be refined to the map of spectra
\[\trace: K(A) \to \THH(A),\]
known as the B\"okstedt trace. The spectrum $\THH(A)$ has an action of the circle group $S^1$ by definition and also has a structure of a $p$-\emph{cyclotomic spectrum} in the sense of \cite{NS}. More precisely it  comes equipped with an $S^1$-equivariant map $\phi_{\cycl} \colon \THH(A) \to \THH(A)^{tC_p}$, where $(-)^{tC_p}$ denotes the Tate construction (see \cite{Green-May} and \cite{NS}). This map is referred to as the \emph{cyclotomic Frobenius}. Let $\THH(A, \ZZ_p)$ denote the $p$-completion of $\THH(A)$, also referred to as the $p$-adic THH. We recall from \cite{NS} and \cite{BMS18}, the following notations:
\[\TC^{-}(A, \ZZ_p)=\THH(A, \ZZ_p)^{hS^1}\]
and
\[\TP(A, \ZZ_p)=\THH(A, \ZZ_p)^{tS^1}.\]
Hete $t$ again stands for the Tate construction and $h$ for the homotopy fixed points. Using the equivalence $(\THH(A)^{tC_p})^{hS^1} \simeq \TP(A, \ZZ_p)$ (see \cite[Lemma II.4.2]{NS}), one gets a map
\[\phi_{\cycl}^{hS^1} \colon \TC^{-}(A, \ZZ_p) \to \TP(A, \ZZ_p).\]
Additionally one has the canonical map from homotopy fixed points to the Tate construction
\[\cann \colon \TC^{-}(A, \ZZ_p) \to \TP(A, \ZZ_p).\]
One then defines  the $p$-adic topological cyclic homology $\TC$ via the fiber sequence of spectra (see \cite{NS})
\[\xymatrix{\TC(A, \ZZ_p) \ar[rr] & &  \TC^{-}(A, \ZZ_p) \ar[rr]^{\cann-\phi_{\cycl}^{hS^1}}  & & \TP(A, \ZZ_p).}\]

One can also consider $\THH(A)$ as a genuine $S^1$-equivariant spectrum with respect to the family of finite subgroups and take the derived fixed points $\THH(A)^{C_{p^{n-1}}}$. This spectrum is denoted by $\TR^n(A)$. We note that $\TR^n(A)$ spectra can be also constructed using the cyclotomic Frobenius as in \cite[Theorem II.4.10]{NS}. The spectra $\TR^n(A)$ for various values of $n$ are related by morphisms
\[F : \TR^{n+1}(A) \to \TR^n(A)\]
and 
\[V : \TR^{n}(A) \to \TR^{n+1}(A).\]
The map $F$ is induced by the fixed points inclusion and $V$ is the transfer. Moreover, the cyclotomic structure induces an $S^1$-equivariant map (with respect residual $S^1$-actions)
\[R :  \TR^{n+1}(A) \to \TR^n(A).\]
These maps induce obvious maps on graded homotopy rings $\pi_*\TR^n(A)$, $n \geq 1$ (denoted by the same letters). Moreover, the circle action induces the differential
\[d : \pi_*\TR^n(A) \to \pi_*\TR^n(A)[-1].\]
It follows from \cite[Theorem 2.3]{HM95} that there is natural isomorphism $\lambda_n^0 \colon \pi_0 \TR^n(A) \cong W_n(A)$. Now the results of \cite{HM04} imply that  for any $\ZZ_{(p)}$-algebra $A$, where the prime $p$ is odd, 
\[(\pi_*\TR^{\bullet}(A), R, F, V, d, \lambda_{\bullet}^0)\] 
forms a Witt complex. and hence there is a unique map of Witt complexes 
from the de Rham-Witt complex over $A$ to the latter Witt complex:
\[\lambda_{\bullet}^* : W_{\bullet}\Omega^*_A \to  \pi_*\TR^{\bullet}(A).\] 
A  theorem of Hesselholt shows that in fact the map
\[\lambda_{\bullet}^1 : W_{\bullet}\Omega^1_A \to  \pi_1\TR^{\bullet}(A)\]
is an isomorphism (for $p$ an odd prime) \cite{Hes04}. 

The trace map $\trace: K(A) \to \THH(A)$ is $S^1$-invariant and has refinements
\[\trace: K(A) \to \TR^n(A)\]
and
\[\trace: K(A, \ZZ_p) \to \TC(A, \ZZ_p).\]
It follows from \cite[Proposition 7.17]{BMS18}, \cite[Corollary 6.5]{AL19} and \cite[Corollary 6.9]{CMM} the canonical maps
\[\xymatrix{K_2(A, \ZZ_p) \ar[r]^-{\cong} & T_p(K_1(A))  & T_p(A^{\times})  \ar[l]_-\cong} \]
are isomorphisms. Here the first map is the map from the universal coefficient sequence for $K(A)$ (see Theorem~\ref{universal coefficient theorem} which is stated for chain complexes but works similarly for spectra) and the second is induced by the homomorphism $A^{\times} \to GL(A)$. By naturality we have a commutative diagram
\[\xymatrix{ K_2(A, \ZZ_p)  \ar[d]_{\trace} \ar[r]^-{\cong}  &T_p(K_1(A)) \ar[d]^{T_p(\trace)} \\ \pi_2 \TR^n(A, \ZZ_p) \ar[r] & T_p(\pi_1\TR^n(A)).}\]
We also recall a result of Geisser and Hesselholt \cite[Lemma 4.2.3]{GeisHes}, which states that the composite
\[\xymatrix{ A^{\times} \ar[r]  & K_1(A)  \ar[r]^-{\trace} &  \pi_1\TR^{\bullet}(A) \cong W_{\bullet}\Omega^1_A  }\]
is given by $a \mapsto \dlog [a]=[a]^{-1}d[a]$. 

Let $A$ be a $p$-torsion-free perfectoid ring containing a compatible system of $p$-power roots of unity 
\[\varepsilon=(1, \zeta_p, \zeta_{p^{2}}, \zeta_{p^{3}}, \dots ).\]
We define the Bott class $\beta_{\varepsilon} \in K_2(A, \ZZ_p)$ to be the element corresponding to $\varepsilon$ under the isomorphism $K_2(A, \ZZ_p) \cong T_p(A^{\times})$. 

\begin{proposition} \label{image of Bott} Let $A$ be a $p$-torsion-free perfectoid ring containing a compatible system of $p$-power roots of unity 
\[\varepsilon=(1, \zeta_p, \zeta_{p^{2}}, \zeta_{p^{3}}, \dots ).\]
Then the composite
\[\xymatrix{K_2(A, \ZZ_p) \ar[r]^-{\trace} & \TR^n(A, \ZZ_p) \ar[r] & T_p(\pi_1\TR^n(A)) \cong  T_p(W_{n}\Omega^1_A) } \]
sends the Bott class $\beta_{\varepsilon}$ to $([\zeta_{p^n}]-1)\alpha_{n}^{(\infty)}$. 
\end{proposition}

\begin{proof} By definition of $\beta_{\varepsilon}$ and the commutative square above, it suffices to compute the image of $\varepsilon$ under the composite
\[\xymatrix{ T_pA^{\times} \ar[r]  & T_pK_1(A)  \ar[rr]^-{T_p(\trace)} & &  T_p(\pi_1\TR^{\bullet}(A)) \cong T_p(W_{\bullet}\Omega^1_A) . }\]
Now by  \cite[Lemma 4.2.3]{GeisHes}, it follows that $\varepsilon$ goes to $\dlog[\zeta_p^{(\infty)}]$. By Theorem \ref{canonical in tate}, the latter is equal to $([\zeta_{p^n}]-1)\alpha_{n}^{(\infty)}$.  
\end{proof}


Let $\TF(A)$ denote the homotopy inverse limit $\holim_{F} \TR^n(A)$. The trace maps $\trace : K(A) \to \TR^n(A)$ assemble into a trace map
$\trace : K(A) \to \TF(A)$.

\begin{corollary} \label{image of Bott in TF} Let $A$ be a $p$-torsion-free perfectoid ring containing a compatible system of $p$-power roots of unity 
\[\varepsilon=(1, \zeta_p, \zeta_{p^{2}}, \zeta_{p^{3}}, \dots ).\]
Then the composite
\[\xymatrix{K_2(A, \ZZ_p) \ar[r]^-{\trace} & \TF(A, \ZZ_p) \ar[r] & \varprojlim_F T_p( \pi_1\TR^n(A)) \cong  \varprojlim_F T_p( W_{n}\Omega^1_A) } \]
sends the Bott class $\beta_{\varepsilon}$ to $([\varepsilon]-1)\alpha$, where $\alpha$ is the generator from Corollary \ref{inverse F}. 
\end{corollary}

\begin{proof} By Theorem \ref{canonical in tate} the map $F : T_p( W_{n+1}\Omega^1_A)  \to T_p( W_{n}\Omega^1_A)$ sends $\alpha_{n+1}^{(\infty)}$ to $\alpha_{n}^{(\infty)}$ and the diagram 
\[\xymatrix{ K_2(A, \ZZ_p) \ar[r]^-{\trace}  \ar[dr]_-{\trace} &  \TR^{n+1}(A, \ZZ_p) \ar[d]^-F \ar[r]  & T_p( \pi_1\TR^{n+1}(A))  \ar[d]^-F &  T_p( W_{n+1}\Omega^1_A) \ar[d]^F \ar[l]^-{\cong}  \\ & \TR^{n}(A, \ZZ_p) \ar[r]  &  T_p( \pi_1\TR^{n}(A))  &  T_p( W_{n}\Omega^1_A)  \ar[l]^-{\cong}   }\]
commutes. This implies the desired result. 
\end{proof}

It turns out that $\alpha$ and $\alpha_n^{(\infty)}$ in fact determine polynomial generators of $\TF(A, \ZZ_p) $ and $\TR^n(A, \ZZ_p)$. It follows from \cite[Lemma II.4.9]{NS} and \cite[Proposition 6.5, Remark 6.6]{BMS18} $\pi_*\TR^n(A, \ZZ_p)$ is a polynomial algebra in one variable of degree $2$ over $W_n(A)$. In particular $\pi_2\TR^n(A, \ZZ_p)$ isomorphic to $W_n(A)$  as a $W_n(A)$-module, and hence the canonical surjective map
\[ \pi_2\TR^n(A, \ZZ_p) \to T_p( \pi_1\TR^n(A)) \cong  T_p( W_{n}\Omega^1_A) \]
is an isomorphism since the target $T_p( W_{n}\Omega^1_A)$ is also a free $W_n(A)$-module on one generator by Theorem \ref{canonical in tate}. Consider the inverse
\[ \xymatrix{T_p( W_{n}\Omega^1_A) \ar[r]^-{\cong} &   \pi_2\TR^n(A, \ZZ_p). }\]
Since $T_p( W_{n}\Omega^1_A)$ is a free $W_n(A)$-module on the generator $\alpha_n^{(\infty)}$, we get that the map
\[\xymatrix{W_n(A)[\alpha_n^{(\infty)}] \ar[r]^{\cong} &  \pi_*\TR^n(A, \ZZ_p) }\]
is an isomorphism of graded rings. Now again using $F\alpha_{n+1}^{(\infty)}=\alpha_{n}^{(\infty)}$, one obtains the following corollary (see \cite[Remark 6.6]{BMS18}).

\begin{corollary} \label{TF computation} Let $A$ be a $p$-torsion-free perfectoid ring containing a compatible system of $p$-power roots of unity 
\[\varepsilon=(1, \zeta_p, \zeta_{p^{2}}, \zeta_{p^{3}}, \dots ).\] 
Then the canonical map
\[ \pi_2\TF(A, \ZZ_p) \to  \varprojlim_F T_p(\pi_1\TR^n(A)) \cong \varprojlim_F T_p( W_{n}\Omega^1_A)  .\]
is an isomorphism. This isomorphism induces an isomorphism of graded rings
\[\xymatrix{W(A^{\flat})[\alpha] \ar[r]^-{\cong} & \pi_*\TF(A, \ZZ_p).}\]
\end{corollary}

\begin{proof} The $F$ maps induce surjections on homotopy groups. Hence by the Milnor sequence $\pi_*\TF(A, \ZZ_p) \cong \varprojlim_F \pi_*\TR^n(A, \ZZ_p)$. The rest follows from the isomorphism of Lemma \ref{inverse limit iso} $W(A^{\flat}) \cong \varprojlim_F W_n(A)$. 

\end{proof}

\begin{remark} The latter results together with \cite[Lemma~3.13]{BMS16} (referred to as Tor-independence in \cite{BMS16}) in fact imply the following:
Let $A \rightarrow B$ denote a ring homomorphism, where both $A$ and $B$ are $p$-torsion free perfectoid rings with compatible systems of $p$-power roots of unity. Assume that the homomorphism preserves these systems. Then for any $1 \leq n \leq r$, 
\[\TR^n(A, \ZZ_p) \otimes^L_{\TR^r(A, \ZZ_p)}    \TR^r(B, \ZZ_p) \simeq \TR^n(B, \ZZ_p), \]
where $\TR^n(A, \ZZ_p)$ is a $\TR^r(A, \ZZ_p)$-module via $F$. (This uses that $\TF(-, \ZZ_p)/ (\sum_{i = 0}^{p^n - 1} [\varepsilon]^i) \simeq \TR^n(-, \ZZ_p)$.) 
In fact \cite[Lemma~3.13]{BMS16} has two versions: one has
\[W_n(A) \otimes^L_{W_r(A)}    W_r(B) \simeq W_n(B),\]
where the $W_n(A)$ is considered as a $W_r(A)$-module via $F$ or $R$. It turns out that using the results from Section \ref{main section}, we also get the following: Let  $W_n\Omega^1_A$ be equipped with a $W_r(A)$-module structure via either Frobenius or restriction.
Then the induced map on derived $p$-completions
\[
\bigg( W_n\Omega^1_A \otimes^L_{W_r(A)} W_r(B) \bigg)^{\wedge} \rightarrow \bigg( W_n \Omega^1_B \bigg)^{\wedge}
\]
is a quasi-isomorphism. To see this, it suffices to prove that
\[
\bigg( \FF_p \otimes_{\ZZ}^L W_n\Omega^1_A \bigg) \otimes^L_{W_r(A)} W_r(B) \stackrel{\simeq}{\rightarrow} \FF_p \otimes_{\ZZ}^L W_n  \Omega^1_B.
\]
As usual, we replace $\FF_p$ with the complex $\cdots \rightarrow 0 \rightarrow \ZZ \stackrel{p}{\rightarrow} \ZZ \rightarrow 0 \rightarrow \cdots.$  Because multiplication by $p$ is surjective on both $W_n\Omega^1_A$ and $W_n\Omega^1_B$, we reduce to showing that the following is a quasi-isomorphism
\[
W_n\Omega^1_A[p] \otimes_{W_r(A)}^L W_r(B) \stackrel{\simeq}{\rightarrow} W_n\Omega^1_B[p],
\]
where the de\thinspace Rham-Witt groups are viewed as complexes concentrated in degree~$-1$.  By our earlier results, we reduce to showing that the following is a quasi-isomorphism
\[
W_n(A)/pW_n(A) \otimes_{W_r(A)}^L W_r(B) \stackrel{\simeq}{\rightarrow} W_n(B)/pW_n(B).
\]
Because $A$ and $B$ (and hence $W_n(A)$ and $W_n(B)$) are $p$-torsion free, we reduce to showing that the following is a quasi-isomorphism
\[
\left(\FF_p \otimes_{\ZZ}^L W_n(A)\right) \otimes_{W_r(A)}^L W_r(B) \stackrel{\sim}{\rightarrow} \FF_p \otimes_{\ZZ}^L W_n(B).
\]
The result now follows from \cite[Lemma~3.13]{BMS16}.

The topological result at the beginning of this remark recovers the $F$-versions of the algebraic equivalences on $\pi_0$ and $\pi_2$. We do not know what is the topological analogue of the equivalences involving the map $R$. 

\end{remark}

Finally, we explain the connection to the results of \cite{Hes06}. Let $V$ be a complete discrete valuation ring with quotient field $K$ of characteristic $0$ and with a perfect residue field of odd characteristic $p$. In \cite{Hes06}, Hesselholt studies the p-adic Tate module of $W_n\Omega^1_{(\overline{V},\overline{M})}$, where $\overline{V}$ is the ring of integers of the algebraic closure $\overline{K}$ and $\overline{M}=\overline{V} \setminus \{0\}$. He shows that $T_pW_n\Omega^1_{(\overline{V},\overline{M})}$ is a free $W_n(\overline{V})^{\wedge}$-module of rank $1$ on certain generator $\alpha_{n, \varepsilon}$. 

The ring $\overline{V}^{\wedge}$ satisfies the conditions of Theorem \ref{canonical in tate} and Theorem \ref{log iso val ring}. Hence the canonical maps
\[ \xymatrix{T_pW_n\Omega^1_{(\overline{V},\overline{M})} \ar[r] & T_pW_n\Omega^1_{(\overline{V}^{\wedge},\overline{M}^{\wedge})}   & T_pW_n\Omega^1_{\overline{V}^{\wedge}} \ar[l]_-{\cong}  } \]
are module maps between rank $1$ free $W_n(\overline{V})^{\wedge} \cong W_n(\overline{V}^{\wedge} )$-modules and the right hand map is an isomorphism. Now by 
\cite[Theorem B, Lemma 2.4.1, Proposition 2.4.2]{Hes06}, there is a generator $\alpha_{n, \varepsilon}$ of the left hand side such that
\[  ([\zeta_{p^{n}}]-1) \alpha_{n, \varepsilon}=\dlog[\zeta_p^{(\infty)}]. \]
Since the the latter zig-zag consists of maps of $W_n(\overline{V}^{\wedge} )$-modules, we see that by Theorem \ref{canonical in tate} the generator $\alpha_{n, \varepsilon}$ on the left is mapped to $\alpha_n^{(\infty)}$ and hence the first map in the zig-zag is an isomorphism too. Finally, we conclude that the generator $\alpha_{\varepsilon}$ constructed in \cite[Proposition 2.4.2]{Hes06} corresponds to $\alpha$ from Corollary \ref{inverse F} under the isomorphism 
\[ \xymatrix{\varprojlim_F T_pW_n\Omega^1_{(\overline{V},\overline{M})} \ar[r]^-{\cong} & \varprojlim_F T_pW_n\Omega^1_{(\overline{V}^{\wedge},\overline{M}^{\wedge})} &
\varprojlim_F T_pW_n\Omega^1_{\overline{V}^{\wedge}} \ar[l]_-{\cong} .} \]

\begin{remark} It follows from the results of Suslin that $K_*(\overline{K}, \ZZ_p) \cong \ZZ_p[\beta_{\varepsilon}]$. We also note that $K_2(\overline{K}, \ZZ_p) \cong K_2(\overline{V}^{\wedge}, \ZZ_p)$ and that the Bott class corresponds to the Bott class in $K_2(\overline{V}^{\wedge}, \ZZ_p)$. Hence $\beta_{\varepsilon} \in K_2(A, \ZZ_p)$ corresponding to $\varepsilon=(1, \zeta_p, \zeta_{p^{2}}, \zeta_{p^{3}}, \dots )$ under the isomorphism $K_2(A, \ZZ_p) \cong T_p A^{\times}$, generalizes the Bott class of \cite{Hes06}, \cite{Suslin1} and \cite{Suslin2} to any $p$-torsion-free perfectoid ring $A$ containing a compatible system of $p$-power roots of unity. We note that $T_p A^{\times}\cong \ZZ_p$ with $\varepsilon$ a generator if additionally $A$ is an integral domain but not in general. \end{remark}

\begin{remark} \label{last remark} We also point out how the generators $\alpha_n^{(\infty)}$ are related to the recent result of \cite{AL19}. Let $A$ be a $p$-torsion-free perfectoid ring containing a compatible system of $p$-power roots of unity. Theorem 6.4 of \cite{AL19} computes the composition
\[\xymatrix{T_pA^{\times} \cong  K_2(A, \ZZ_p) \ar[r]^-{\trace} & \pi_2\TC(A, \ZZ_p) }\]
which is an isomorphism by \cite[Corollary 6.9]{CMM} and \cite[Corollary 6.5]{AL19}. It follows from \cite[Section 6]{BMS18}, the definition of $\TC$ and \cite[Section 6]{AL19} that 
\[\pi_2\TC(A, \ZZ_p) \cong \left\{y \in W_n(A^{\flat}) \colon \varphi(y) = \frac{[\varepsilon^p] - 1}{[\varepsilon] - 1} y \right\},\]
and under this isomorphism the latter composite is given by the $q$-logarithm:
\[\log_q([x]) = \sum_{n = 1}^{\infty} (-1)^{n-1} q^{-n(n-1)/2} \frac{([x]- 1)([x] - q)([x] - q^2) \cdots ([x] - q^{n-1})}{[n]_q} \in W(A^{\flat}),\]
where $q = [\varepsilon] \in W(A^{\flat})$ and $[n]_q = \frac{q^n - 1}{q-1}$. Recall that there is a cofiber sequence \cite[Theorem II.4.10]{NS}
\[\xymatrix{\TC(A, \ZZ_p) \ar[r] & \TF(A, \ZZ_p) \ar[r]^{R-1} &  \TF(A, \ZZ_p).}\]
The diagram 
\[\xymatrix{T_pA^{\times} \cong  K_2(A, \ZZ_p)  \ar[r]^-{\trace} \ar[dr]_-{\trace} & \pi_2\TC(A, \ZZ_p) \ar[d] \\ & \pi_2 \TF(A, \ZZ_p)   }\]
commutes. Under the isomorphism $\pi_2 \TF(A, \ZZ_p) \cong \varprojlim_F T_p( W_{n}\Omega^1_A)$, the vertical map in this diagram corresponds to the inclusion 
\begin{align*}
\bigg\{ y \in W_n(A^{\flat}) \colon \varphi(y) = \frac{[\varepsilon^p] - 1}{[\varepsilon] - 1} y\bigg\} \cong \bigg\{ y\alpha \in \varprojlim_F T_p(W_n\Omega^1_A) \colon y \in W(A^{\flat}) \text{ satisfies } \varphi(y) = \frac{[\varepsilon^p] - 1}{[\varepsilon] - 1} \, y\bigg\} = \\  \bigg\{ x \in \varprojlim_F T_p(W_n\Omega^1_A) \colon R(x) = x\bigg\} \subset \varprojlim_F T_p( W_{n}\Omega^1_A).
\end{align*}
(See Proposition \ref{equal sets}) If we now plug in $\varepsilon = x$ in the above formula, we get $\log_q([\varepsilon])=[\varepsilon]-1$ and hence we recover that the composite
 \[\xymatrix{T_pA^{\times} \cong  K_2(A, \ZZ_p)  \ar[r]^-{\trace} & \pi_2\TC(A, \ZZ_p) \ar[r] & \pi_2 \TF(A, \ZZ_p)  \cong \varprojlim_F T_p( W_{n}\Omega^1_A)  }\]  
sends $\varepsilon$ to $([\varepsilon]-1)\alpha$. This is an alternative way of computing the image of the Bott class in $\varprojlim_F T_p( W_{n}\Omega^1_A)$. 

\end{remark}

\bibliography{Tate}
\bibliographystyle{apalike}

\end{document}